\documentclass[letterpaper, 1in, 12pt]{article}
\usepackage{geometry}
\newcommand{\qed}{\hfill$\Box$}

\usepackage[utf8]{inputenc}
\usepackage[english]{babel}
\usepackage{clm-macros}
\usepackage{amsmath}
\usepackage{psfrag}
\usepackage{epsfig}
\usepackage{graphicx}
\usepackage{verbatim}

\usepackage[numbers]{natbib} 
\usepackage[colorlinks=true, urlcolor=blue, citecolor=black, linkcolor=black]{hyperref}
\usepackage{doi}

\usepackage{MnSymbol}
\usepackage{subcaption}
\usepackage{datetime}
\usepackage{enumerate}
\usepackage{bookmark}

\usepackage{tikz}

\usepackage{multirow}
\usepackage{booktabs} 

\usepackage{graphicx}

\usepackage{mathtools}
\usepackage{footnote}

\newtheorem{rthm}{Theorem}
\newenvironment{reptheorem}[1]
{\renewcommand\therthm{\ref*{#1}}\rthm}
{\endrthm}

\makesavenoteenv{table}
\usetikzlibrary{calc,scopes}
\usetikzlibrary{backgrounds}
\usetikzlibrary{decorations}
\usetikzlibrary{decorations.pathmorphing}
\usetikzlibrary{arrows}
\tikzstyle{every node}=[circle, draw, fill=white,inner sep=0pt, minimum width=4pt]
\tikzstyle{nodelabel}=[rounded corners,fill=none,inner sep=5pt,draw=none]
\tikzstyle{matching} = [ultra thick]
\tikzset{snake/.style={decorate, decoration=snake}}
\hypersetup{citecolor=blue,linkcolor=blue,colorlinks=true}

\def\Notname{Notation}
\def\casname{Case}
\def\cnjname{Conjecture}
\def\corname{Corollary}
\def\Defname{Definition}
\def\exename{Exercise}
\def\exaname{Example}
\def\lemname{Lemma}
\def\listassertionname{List of Assertions}
\def\prbname{Problem}
\def\prpname{Proposition}
\def\proofOfname{\proofname{} of}
\def\ssbname{Case}
\def\subname{Case}
\def\thmname{Theorem}

\newcommand{\mc}{matching covered}

\newcommand{\rg}{\mbox{$r$-graph}}
\newcommand{\rgs}{\mbox{$r$-graphs}}
\newcommand{\cc}[1]{\mbox{$#1$-connected} \mbox{$3$-regular} graph}

\newcommand{\ecc}[1]{\mbox{$#1$-edge-connected}}
\newcommand{\dis}{{\sf dist}}

\newcommand{\distg}[3]{$\dis_{#1} (#2,#3)$}
\newcommand{\distgl}[3]{$\dis_{#1} (#2,#3)\le 1$}

\newcommand{\ses}{solitary edges}
\newcommand{\sd}{solitary doubleton}

\newcommand{\pmg}{perfect matching}
\newcommand{\pms}{perfect matchings}
\newcommand{\mcg}{matching covered graph}
\newcommand{\sepcut}{separating cut}

\newcommand{\mdc}{matching double covered}
\newcommand{\dumbbell}{dumbbell}
\newcommand{\mbb}[3]{${{#1}_{#2}^{#3}}$}

\title{Solitude patterns of $r$-graphs via CLM's dependence relation\\ \vspace{0.5cm} \footnotesize{Dedicated to CLM (Marcelo H. de Carvalho, Cl\'audio L. Lucchesi and U. S. R. Murty)\\ for their invaluable contributions to the theory of perfect matchings}}
\author{D.\ V.\ V.\ Narayana${}^{1}$\thanks{Supported by IITM Pravartak Technologies Foundation.}
, D.\ Mattiolo${}^{2}$\thanks{The author was supported by a postdoctoral fellowship from the Research Foundation Flanders (FWO), with grant number 1268323N, during the completion of this work.}, K.\ Gohokar${}^{3}$\thanks{Supported by Google.}, N.\ Kothari${}^{1}$\thanks{Supported by IC\&SR IIT Madras.} \\
{\footnotesize ${}^{1}$ IIT Madras, India.} 
\\{\footnotesize ${}^{2}$ Department of Computer Science, KU Leuven Kulak, 8500 Kortrijk, Belgium.}
\\{\footnotesize ${}^{3}$ Chennai Mathematical Institute, India.}}

\begin{document}


  \maketitle 

\begin{abstract}
    Recently, Goedgebeur, Mazzuoccolo, Renders, Wolf and the second author
     [Cubic graphs with edges in exactly one perfect matching, \emph{J.\ Graph Theory} 112 (2026), 276-289] proved that every \mbox{$3$-connected} $3$-regular graph has at most six solitary edges; an edge is {\em solitary} if it participates in precisely one perfect matching. They also gave complete characterizations of \mbox{$3$-connected} $3$-regular graphs that have $k$ solitary edges for each $k \ge 3$. 
    
    A connected $r$-regular graph, where $r \geq 3$, is an {\em \rg} if each odd cut has at least $r$ edges; for instance, $3$-graphs are precisely the $2$-connected $3$-regular graphs.
    We generalize the bound and characterizations, mentioned in the preceding paragraph, to \mbox{$3$-edge-connected} \rgs\ of order four or more; in particular, for $r \geq 4$, we establish stronger bounds:\ (i) at most four solitary edges, and (ii) if the graph is simple then it is devoid of solitary edges. Apart from characterizing all $3$-edge-connected \rgs\ that have three or more solitary edges, for the case $r=3$, we provide a recursive characterization of those that have precisely two solitary edges such that they lie in the same perfect matching.
    
    Finally, for all \rgs\ (not necessarily $3$-edge-connected), we establish that:\ (i) every such graph, except for a few small graphs, has at most $\frac{n}{2}$ solitary edges, and (ii) every such graph decomposes uniquely into $3$-edge-connected \rgs, and its solitary edges may be computed recursively via this decomposition.
Our proofs and insights rely heavily on the dependence and mutual dependence relationships introduced and investigated by Carvalho, Lucchesi and Murty [Ear decompositions of matching covered graphs, \emph{Combinatorica} 19 (1999), 151-174].
\end{abstract}



\newpage
 \tableofcontents

\newpage


\section*{Overview}

An \emph{$r$-graph}, where $r \geq 3$, is any connected $r$-regular graph $G$ with the additional property that, for every $X \subseteq V(G)$ of odd cardinality, there are at least $r$ edges that have precisely one end in~$X$.
It is easily deduced from Tutte’s $1$-factor Theorem that each edge of an $r$-graph participates in some perfect matching. Recently \cite{gmmrw24}, \emph{solitary edges} --- those that participate in precisely one perfect matching --- were investigated in the case of $3$-graphs. They proved that every \mbox{$3$-edge-connected} $3$-graph has at most six solitary edges, and provided a complete characterization of those having three or more solitary edges.
In this manuscript, we generalize their results to all $r$-graphs using techniques from the theory of `matching covered graphs'.

We investigate solitary edges of \emph{matching covered} graphs (that is, those connected graphs whose each edge belongs to some perfect matching) through the lens of the dependence relationship that was studied by Carvalho, Lucchesi and Murty \cite{carv97}, and further investigated in the recently published monograph \cite{lumu24}. Two edges of a matching covered graph $G$ are {\em mutually dependent} if every \pmg\ containing either of them also contains the other.
Clearly, this is an equivalence relation and thus induces a partition of~$E(G)$; each part is called an {\em equivalence class}. It is worth noting that if any member of an equivalence class is \mbox{solitary} then so is every other member; we refer to such an equivalence class as a {\em solitary class}. This observation leads us to introduce the notion of {\em solitude pattern} of a \mcg\ --- the sequence of cardinalities of its solitary classes in nonincreasing order. For instance, $K_4$ has solitude pattern~$(2,2,2)$ whereas the Petersen graph has solitude pattern~$()$.

Section \ref{sec:introduction} is devoted to a thorough introduction and development of the main tools we use to investigate solitude. In fact, in the same section, we already begin our study of solitary edges.
For instance, we establish that every $r$-graph, that is not \ecc{3}, has at most one solitary class, and thus at most $\frac{n}{2}$ solitary edges; in fact, the latter upper bound applies to all $r$-graphs of order ten or more; see Theorem~\ref{thm:at-most-n/2-solitary-edges} and Corollary \ref{lem:2-cut-decomposition-r-graphs-new}. A folklore result states that every \mbox{$r$-graph} uniquely decomposes into \ecc{3} \mbox{$r$-graphs}; we show that its solitary edges may be computed recursively via this decomposition; see Proposition \ref{prp:marked-c-components-solitary-edges}.
On the other hand, using a result of Lucchesi and Murty, we deduce that in a \mbox{$3$-edge-connected} \mbox{$r$-graph}, every solitary class has cardinality one or two. Furthermore, if the order is four or more, we establish that the number of solitary classes is at most three and that equality holds if and only if $r=3$. Consequently, every \mbox{$3$-edge-connected} $r$-graph, of order four or more, has one of these ten solitude patterns: $(2,2,2)$, $(2,2,1)$, $(2,1,1)$, $(1,1,1)$, $(2,2)$, $(2,1)$, $(2)$, $(1,1)$, $(1)$~or~$()$, and the first four are possible only if $r=3$. See Corollaries~\ref{cor:main-thm}~and~\ref{lem:solitary-patterns-in-r-graphs}.

In Section \ref{sec:characterizations}, we provide complete characterizations of \mbox{$3$-edge-connected} $r$-graphs that have one of the first six solitude patterns (as listed above); see Theorems \ref{thm:more-than-one-SD}, \ref{thm:SDSS} and \ref{thm:3SS}. We also provide a recursive description of \mbox{$3$-edge-connected} $3$-graphs that have solitude pattern~$(2)$; see Theorem \ref{thm:SP2}. The proofs appear in Section \ref{sec:proofs-of-main-results}, and they rely on the study of distance between solitary classes that is carried out in Section \ref{sec:further-developing-toolkit}. Using structural arguments pertaining to cuts, edge-colorings and the dependence relation, we prove that the distance between any two solitary classes in any \mbox{$3$-edge-connected} $r$-graph is at most three; see Corollaries \ref{cor:dist-at-most-one-in-r-graphs} and \ref{cor:dist-at-most-3-unless-pattern(2)}.

Note that~$\frac{n}{2}$ is an upper bound on the cardinality of any equivalence class; if equality holds then each largest equivalence class is a solitary class. In Section \ref{sec:largest-EC-n/2}, we characterize all \mcg s that attain this upper bound; this class includes certain \mbox{$r$-graphs} for each $r\ge 3$; see Corollary \ref{cor:r-graphs-n/2-solitary-edges}. It turns out that all such graphs, of order six or more, contain $2$-cuts.

    We mention a few consequences. Firstly, every \ecc{3} $r$-graph has at most four solitary edges if $r \geq 4$ and $n \geq 4$, and at most six solitary edges if $r=3$. Secondly, in a simple \mbox{$3$-edge-connected} $r$-graph, where $r\ge 4$, each edge participates in at least two perfect matchings; see Corollary \ref{cor:simple_r-graph_matching_double_covered}. 
    Thirdly, every (not necessarily simple) $r$-graph, that has a solitary edge, is \mbox{$r$-edge-colorable}; see Corollary~\ref{cor:solitary-implies-class-I}. This is interesting in light of Seymour's conjecture~\cite{seym80} proved recently by Chen, Jing and Zang~\cite{ggw25,ggw25v2}: every $r$-graph is $(r+1)$-edge-colorable.

\section{Introduction and summary}\label{sec:introduction}

For general graph-theoretic notation and terminology, we follow Bondy and Murty~\cite{bomu08}, whereas for terminology specific to matching theory, we follow Lucchesi and Murty~\cite{lumu24}.
We use $n$ to denote the {\em order} of a graph --- that is, its number of vertices. All graphs in this paper are loopless; however, we allow multiple/parallel edges. 

\subsection{Cuts and \texorpdfstring{$r$}{}-graphs}

For any graph $G$ and $X\subseteq V(G)$, we use $\partial(X)$ to denote the corresponding {\em cut} --- that is, the set comprising those edges that have one end in $X$ and the other end in $\overline{X}:=V(G)-X$. We refer to $X$ and $\overline{X}$ as the {\em shores} of the cut $C:=\partial(X)$. 
A cut is {\em even} if both of its shores are of even cardinality; otherwise it is {\em odd}; it is {\em trivial} if either of its shores is a singleton, say $\{v\}$, in which case we simplify the notation $\partial(\{v\})$ to $\partial(v)$.
A {\em bond} is any minimal nonempty cut. It is well-known that, for a connected graph, a nonempty cut is a bond if and only if both shores induce connected subgraphs.
Finally, $C$ is a {\em $k$-cut} if $|C|=k$. We mention that cuts play a crucial role in our work. A \textit{matching} is any subset of $E(G)$, say~$M$, such that $|M\cap \partial(v)|\le 1$ for each $v\in V(G)$; it is \textit{perfect} if equality holds for each vertex.
Using these definitions and simple counting arguments, the reader may easily verify the following facts that we shall invoke implicitly throughout this paper.

\begin{lem}\label{lem:basic-facts-about-parity}
For any cut $\partial(X)$ of a graph~$G$, the following statements hold:
\begin{enumerate}[(i)]
\item for each \pmg~$M$: $|M\cap \partial(X)|\equiv |X|~({\rm mod~}2)$, and
\item if each vertex in $X$ has odd degree in~$G$ then $|\partial(X)|\equiv |X|~({\rm mod~}2)$. \qed
\end{enumerate}
\end{lem}

Regular graphs have been studied extensively in graph theory, especially in the context of edge-colorings and matchings. In particular, \mbox{$3$-regular} graphs have received immense attention; we note the following well-known fact that is easily proved and that shall be used implicitly throughout this paper.

\begin{prp}
    For a $3$-regular graph, its vertex-connectivity is the same as its edge-connectivity. \qed
\end{prp}

A connected $r$-regular graph, where $r\ge 3$, is an {\em $r$-graph} if each odd cut has at least~$r$ edges. It is worth noting that \mbox{$3$-graphs} are precisely the \mbox{$2$-connected} \mbox{$3$-regular} graphs.
The following facts are easily verified using the definition, and the subtle point that $\partial(V(G))=\emptyset$ is also a cut of any graph~$G$.

\begin{prp}\label{prp:r-graphs-basic-prp}
Every $r$-graph~$G$ satisfies the following:
\begin{enumerate}[(i)]
    \item it is of even order,
    \item it is \ecc{2}; consequently, each (even) $2$-cut is a bond,
    \item each odd $r$-cut is a bond, and
    \item it is \ecc{3} if and only if it is devoid of even $2$-cuts. \qed
\end{enumerate} 
\end{prp}

It is worth noting that $3$-edge-connected $3$-graphs are precisely the $3$-connected \mbox{$3$-regular} graphs.
The following is proved by noting that, given any proper \mbox{$r$-edge-coloring} of an $r$-regular graph, each color class is a \pmg, and thus meets each odd cut.

\begin{prp}
    Every connected \mbox{\mbox{$r$-edge-colorable}} \mbox{$r$-regular} graph~$G$, where $r\ge3$, is an \mbox{$r$-graph}. \qed
\end{prp}

However, the converse does not hold in general; for instance, the Petersen graph is the smallest \mbox{$3$-graph} that is not $3$-edge-colorable. In summary, connected \mbox{$r$-edge-colorable} \mbox{$r$-regular} graphs, where $r\ge 3$, comprise a proper subset of $r$-graphs. Extensive research has been conducted on $r$-graphs. We mention a particular conjecture of Seymour~\cite{s77} that has been proved recently by Chen, Jing and Zang; see~\cite{ggw25,ggw25v2,gj26}.


\begin{thm}\label{conj:Seymour}
    Every $r$-graph is $(r+1)$-edge-colorable.
\end{thm}

Observe that the above holds for simple $r$-graphs due to the famous Vizing-Gupta Theorem. In light of the above result, one may partition the set of $r$-graphs into {\em class I} comprising those with edge chromatic number $r$, and {\em class II} comprising those with edge chromatic number $r+1$; it follows from the result of Holyer~\cite{holyer81} that the corresponding decision problem is NP-complete. Ergo, it is worthwhile to look for conditions that are either necessary or sufficient; in light of this, we mention the following conjecture of Seymour.

\begin{conj}\label{conj:Seymour-planar}
    Every planar $r$-graph belongs to class I.
\end{conj}

In the following section, we discuss `solitude' in $r$-graphs.

\subsection{Solitary edges in \texorpdfstring{$r$}{}-graphs}

A graph is {\em matchable} if it has a \pmg. In the same spirit, an edge is {\em matchable} if it participates in some \pmg; otherwise it is {\em unmatchable}. A connected graph, of order two or more, is {\em \mc}~if each of its edges is matchable. For instance, a strengthening of Petersen's result (1891) due to Sch\"onberger (1934) states that every $3$-graph is matching covered. In fact, using Tutte's $1$-factor Theorem (1947), one may easily prove the following generalization.

\begin{prp}\label{prp:r-graph-mc}
Every $r$-graph is matching covered. \qed
\end{prp}

A connected graph is {\em matching double covered} if each of its edges belongs to two or more (distinct)  \pms. In light of the above proposition, it is natural to ask whether one can characterize those $r$-graphs that are matching double covered. It is easy to observe that not all $r$-graphs are \mdc. For instance, in Figure~\ref{fig:2-connected cubic graphs with too many solitary edges}, each of the edges, except for those that are colored black, belong to exactly one \pmg\ in the corresponding $r$-graph.

An edge of a graph is {\em solitary} --- aka {\em lonely} in~\cite{gmmrw24}, or {\em solo} in~\cite{kss09} --- if it belongs to precisely one perfect matching. For a \pmg~$M$ of a graph~$G$, a path or a cycle~$Q$ is \textit{$M$-alternating} if its edges, on traversal, alternate between being in~$M$ and not being in~$M$. The following is easily verified.

\begin{prp}\label{prp:mcgs-hcycle}
Let $G$ be a (matchable) graph that has a solitary edge~$e$, and let $M_e$ denote the unique \pmg\ containing~$e$. Then, every $M_e$-alternating cycle contains~$e$. Consequently,
if $M$ is any \pmg\ that is disjoint with~$M_e$, then $M+ M_e$ is a hamiltonian cycle. \qed
\end{prp}

The notion of solitary edges was used in \cite{clm05}, \cite{kss09} and \cite{ekss10} to either establish lower bounds on the number of perfect matchings in \mbox{$3$-graphs}, or to characterize those $3$-graphs that attain a certain lower bound. 
Similarly, solitary edges were employed in~\cite{clm13} to prove lower bounds on the number of perfect matchings in bipartite matching covered graphs.

Ergo, characterizing $r$-graphs that are matching double covered is equivalent to characterizing those that are free of solitary edges. As a first step, one may ask whether there is an upper bound on the number of solitary edges in $r$-graphs.
It is not too difficult to construct examples of $r$-graphs that have $\frac{n}{2}$ \ses. 
In fact, this is best possible as we shall prove in this paper; see Theorem~\ref{thm:at-most-n/2-solitary-edges}. Before stating our result, we need some notation and terminology to describe the small exceptions.


Let $G$ be a matchable $r$-regular graph, and let $M$ denote a fixed \pmg. The \mbox{$(r+k-1)$-regular} (matchable) graph obtained from~$G$ by introducing $k-1$ new copies of each edge in~$M$, denoted by $G\oplus (k-1)M$, is said to be obtained from~$G$ {\em by multiplying the \pmg~$M$ by the positive integer~$k$}. Note that if $k=1$ then $G=G\oplus (k-1)M$. For instance, the graphs shown in Figure~\ref{fig:triangular-prism} are obtained from the triangular prism~$\overline{C_6}$ by multiplying the unique \pmg\ containing the blue edges by $1, 2$~and~$3$, respectively.

\begin{figure}[!htb]
    \centering
    \begin{subfigure}[b]{.46\textwidth}
        \centering
        \begin{tikzpicture}[scale=0.7]
            \node[circle,fill=white] (1) at (0,0){};
            \node[circle,fill=white] (2) at (2,0){};
			\draw[bend left, thick, red] (1) to (2);
			\draw[bend right, thick, blue] (1) to (2);
			\draw[ thick, green] (1) -- (2);

            \node[circle,fill=white] (1) at (3,0){};
            \node[circle,fill=white] (2) at (5,0){};
			\draw[bend left, thick, red] (1) to (2);
			\draw[bend right, thick, blue] (1) to (2);
            \draw[bend left=105, looseness=1.25, thick, pink] (1) to (2);
			\draw[ thick, green] (1) -- (2);

            \node[circle,fill=white] (1) at (6,0){};
            \node[circle,fill=white] (2) at (8,0){};
			\draw[bend left, thick, red] (1) to (2);
			\draw[bend right, thick, blue] (1) to (2);
            \draw[bend left=105, looseness=1.25, thick, pink] (1) to (2);
            \draw[bend right=105, looseness=1.25, thick, yellow] (1) to (2);
			\draw[ thick, green] (1) -- (2);
        \end{tikzpicture}
  \caption{Smallest members of \mbb{\theta}{}{1}}
\label{fig:theta-graph}
    \end{subfigure}
    \begin{subfigure}[b]{.5\textwidth}
        \centering
        \begin{tikzpicture}[scale=0.9]
            \node[circle,fill=white] (1) at (1, 1){};
            \node[circle,fill=white] (2) at (0, -0.5){};
            \node[circle,fill=white] (3) at (2, -0.5){};
            \node[circle,fill=white] (0) at (1, 0){};
                \draw[thick,red] (0) to (1);
                \draw[thick,red] (2) to (3);
                \draw[thick,blue] (0) to (2);
                \draw[thick,blue] (1) to (3);
                \draw[thick,green] (0) to (3);
                \draw[thick,green] (1) to (2);

                \node[circle,fill=white] (1) at (4, 1){};
            \node[circle,fill=white] (2) at (3, -0.5){};
            \node[circle,fill=white] (3) at (5, -0.5){};
            \node[circle,fill=white] (0) at (4, 0){};
                \draw[thick] (0) to (1);
                \draw[thick, bend left=15] (2) to (3);
                \draw[thick, bend left=15] (0) to (1);
                \draw[thick] (2) to (3);
                \draw[thick,blue] (0) to (2);
                \draw[thick,blue] (1) to (3);
                \draw[thick,green] (0) to (3);
                \draw[thick,green] (1) to (2);

            \node[circle,fill=white] (1) at (7, 1){};
            \node[circle,fill=white] (2) at (6, -0.5){};
            \node[circle,fill=white] (3) at (8, -0.5){};
            \node[circle,fill=white] (0) at (7, 0){};
                \draw[thick] (0) to (1);
                \draw[thick, bend left=15] (2) to (3);
                \draw[thick, bend left=15] (0) to (1);
                \draw[thick, bend right=15] (2) to (3);
                \draw[thick, bend right=15] (0) to (1);
                \draw[thick] (2) to (3);
                \draw[thick,blue] (0) to (2);
                \draw[thick,blue] (1) to (3);
                \draw[thick,green] (0) to (3);
                \draw[thick,green] (1) to (2);
        \end{tikzpicture}
        \caption{Smallest members of \mbb{K}{4}{1}}
        \label{fig:K4}
    \end{subfigure}
    \begin{subfigure}[b]{.5\textwidth}
        \centering
        \begin{tikzpicture}[scale=0.7]
            \node[circle,fill=white] (1) at (-0.5,0){};
            \node[circle,fill=white] (2) at (0.5,0){};
            \node[circle,fill=white] (3) at (1.5,-1){};
            \node[circle,fill=white] (4) at (1.5,1){};
            \node[circle,fill=white] (5) at (-1.5,1){};
            \node[circle,fill=white] (6) at (-1.5,-1){};
            \draw (1) -- (2);
            \draw (5) -- (4);
            \draw (6) -- (3);
                    \draw[thick,red] (3) -- (4);
                    \draw[thick,red] (6) -- (5);
                    \draw[thick,green] (2) -- (3);
                    \draw[thick,blue] (1) -- (5);
                    \draw[thick,blue] (2) -- (4);
                    \draw[thick,green] (1) -- (6);

                    \node[circle,fill=white] (1) at (2.5+1,0){};
            \node[circle,fill=white] (2) at (3.5+1,0){};
            \node[circle,fill=white] (3) at (4.5+1,-1){};
            \node[circle,fill=white] (4) at (4.5+1,1){};
            \node[circle,fill=white] (5) at (1.5+1,1){};
            \node[circle,fill=white] (6) at (1.5+1,-1){};
            \draw (1) -- (2);
            \draw (5) -- (4);
            \draw[thick] (6) -- (3);
                    \draw[thick, bend left=15] (6) to (3);
                    \draw (6) -- (3);
                    \draw[thick] (1) -- (5);
                    \draw[thick] (2) -- (4);
                    \draw[thick,red] (3) to (4);
                    \draw[thick,red] (6) to (5);
                    \draw[thick,green] (2) to (3);
                    \draw[thick,bend left=15] (1) to (5);
                    \draw[thick,bend left=15] (2) to (4);
                    \draw[thick,green] (1) to (6);

                    \node[circle,fill=white] (1) at (5.5+2,0){};
            \node[circle,fill=white] (2) at (6.5+2,0){};
            \node[circle,fill=white] (3) at (7.5+2,-1){};
            \node[circle,fill=white] (4) at (7.5+2,1){};
            \node[circle,fill=white] (5) at (4.5+2,1){};
            \node[circle,fill=white] (6) at (4.5+2,-1){};
            \draw (1) -- (2);
            \draw (5) -- (4);
            \draw[thick] (6) -- (3);
                    \draw[thick, bend left=15] (6) to (3);
                    \draw[thick, bend right=15] (6) to (3);
                    \draw (6) -- (3);
                    \draw[thick] (1) -- (5);
                    \draw[thick] (2) -- (4);
                    \draw[thick,red] (3) to (4);
                    \draw[thick,red] (6) to (5);
                    \draw[thick,green] (2) to (3);
                    \draw[thick,bend left=15] (1) to (5);
                    \draw[thick,bend left=15] (2) to (4);
                    \draw[thick,bend right=15] (1) to (5);
                    \draw[thick,bend right=15] (2) to (4);
                    \draw[thick,green] (1) to (6);
        \end{tikzpicture}
        \caption{Smallest members of ${\overline{C_6}}^1$}
        \label{fig:triangular-prism}
    \end{subfigure}
    \begin{subfigure}[b]{.45\textwidth}
        \centering
        \begin{tikzpicture}[scale=0.7]
            \node[circle,fill=white] (1) at (-0.5-0.25,0){};
		\node [draw=none] at (-0.5-0.25, 0.37) {$1$};
		\node [draw=none] at (0.5+0.25, 0.37) {$4$};
		\node [draw=none] at (1.5+0.25,-1.37) {$5$};
		\node [draw=none] at (1.5+0.25,1.37) {$3$};
		\node [draw=none] at (-1.5-0.25,1.37) {$0$};
		\node [draw=none] at (-1.5-0.25,-1.37) {$2$};
		\node [draw=none] at (0,0.37) {$6$};
		\node [draw=none] at (0,-1.37) {$7$};
            \node[circle,fill=white] (2) at (0.5+0.25,0){};
            \node[circle,fill=white] (3) at (1.5+0.25,-1){};
            \node[circle,fill=white] (4) at (1.5+0.25,1){};
            \node[circle,fill=white] (5) at (-1.5-0.25,1){};
            \node[circle,fill=white] (6) at (-1.5-0.25,-1){};
            \node[circle,fill=white] (7) at (0,0){};
            \node[circle,fill=white] (8) at (0,-1){};
            \draw (1) -- (2);
            \draw (2) -- (3);
            \draw (1) -- (6);
            \draw (5) -- (4);
            \draw (6) -- (3);
                \draw[thick,red] (3) -- (4);
                \draw[thick,green] (6) -- (5);
                \draw[thick,green] (2) -- (4);
                \draw[thick,red] (1) -- (5);
                \draw[thick,blue] (7) -- (8);
                \node[circle,fill=white] (1) at (-0.5-0.25,0){};
            \node[circle,fill=white] (2) at (0.5+0.25,0){};
            \node[circle,fill=white] (3) at (1.5+0.25,-1){};
            \node[circle,fill=white] (4) at (1.5+0.25,1){};
            \node[circle,fill=white] (5) at (-1.5-0.25,1){};
            \node[circle,fill=white] (6) at (-1.5-0.25,-1){};
            \node[circle,fill=white] (7) at (0,0){};
            \node[circle,fill=white] (8) at (0,-1){};   
        \end{tikzpicture}
		 \caption{$R_8$}
  \label{fig:R8}
    \end{subfigure}
    \caption{$r$-graphs with more than $\frac{n}{2}$ solitary edges}
    \label{fig:2-connected cubic graphs with too many solitary edges}
\end{figure}

The first part of the following is easily verified using the definition of $r$-graphs and the fact that each perfect matching meets each odd cut, whereas the second part is immediate.

\begin{prp}
    Every graph~$H$ obtained from an $r$-graph~$G$, by multiplying any \pmg\ by a positive integer~$k$, is an {$(r+k-1)$-graph}. Furthermore, if $G$ is \mbox{$r$-edge-colorable} then $H$ is $(r+k-1)$-edge-colorable.
    \qed
\end{prp}

For each $3$-regular graph~$G\in \{\theta, K_4, \overline{C_6}\}$, as shown in Figures \ref{fig:theta-graph}, \ref{fig:K4} and \ref{fig:triangular-prism}, we let $(M_1,M_2,M_3)$ denote its unique proper \mbox{$3$-edge-coloring}, and we define the family $G^i$, where $i\in \{1,2,3\}$, as 
$\{G\oplus (k_1-1)M_1 \oplus$ {$(k_2-1)M_2$} $\oplus (k_3-1)M_3~|~k_1,k_2,k_3\in \mathbb{Z}^+\ {\rm and}\ {\rm at\ most}\ i$ ${\rm of\ them}$~${\rm are\ greater\ than\ one}\}$. We define the families \mbb{C}{4}{1} and \mbb{C}{4}{2} analogously by considering its unique proper \mbox{$2$-edge-coloring}. Note that each graph in \mbb{C}{4}{2}$- C_4$ is an $r$-graph.
The following two propositions are easily verified.




\begin{prp}\label{r-graphs-of-order-two}
    The set \mbb{\theta}{}{1} comprises all $r$-graphs of order two; furthermore, each of its members is \ecc{3} and has the property that every edge is solitary. \qed
\end{prp}

\begin{prp}\label{r-graphs-of-order-four}
    The set $($\mbb{C}{4}{2}$-C_4)$ $\cup$ \mbb{K}{4}{3} comprises all $r$-graphs of order four; furthermore: (i) each member, except for those in $($\mbb{C}{4}{1}$-C_4)$, is \ecc{3}, and (ii) a member has a solitary edge if and only if it belongs to $($\mbb{C}{4}{1}$-C_4)$ $\cup$ \mbb{K}{4}{2}. \qed
\end{prp}







By $R_8$, we mean the bicorn graph shown in Figure \ref{fig:R8}. We are now ready to state the aforementioned result upper bounding the number of solitary edges in any $r$-graph.

\begin{thm}\label{thm:at-most-n/2-solitary-edges}
    For any $r$-graph~$G$, precisely one of the following holds:
    \begin{enumerate}[(i)]
        \item either $G$ has at most $\frac{n}{2}$ \ses, 
        \item or otherwise $G\in $ \mbb{\theta}{}{1}   $\cup$~\mbb{K}{4}{1}~  $\cup~{\overline{C_6}}^1 \cup \{R_8\}$.
    \end{enumerate}
\end{thm}

A proof of the above appears in Section~\ref{sec:proof-of-Theorem-at-most-n/2-solitary-edges}.
In Section~\ref{sec:largest-EC-n/2}, we provide a characterization of those $r$-graphs that have precisely $\frac{n}{2}$ solitary edges; see Corollary~\ref{cor:r-graphs-n/2-solitary-edges}. All such graphs, of order six or more, contain a $2$-cut. In fact, in the case of \ecc{3} $r$-graphs of order four or more, we shall establish constant upper bounds on the number of solitary edges; see Theorem~\ref{thm:main-thm-in-simple-words}. 
In the following section, we discuss a folklore result which states that every $r$-graph ``decomposes'' uniquely into \ecc{3} $r$-graphs, and we shall observe that in order to study solitary edges in $r$-graphs, one may restrict themself to \ecc{3} $r$-graphs.

\subsection{The \texorpdfstring{\ecc{3}}{} pieces of an \texorpdfstring{$r$}{}-graph}
Let $G_1$ and $G_2$ denote disjoint graphs of even orders $n_1$~and~$n_2$, respectively, with specified edges $e_1:=u_1v_1$~and $e_2:=u_2v_2$, respectively. 
We refer to the graph~$G$ obtained from the union of $G_1-e_1$ and $G_2-e_2$, by adding two edges $f:=u_1u_2$ and $f':=v_1v_2$, as a graph obtained by {{\em gluing $G_1$~and~$G_2$ at edges $e_1$~and~$e_2$}}, or simply as a graph obtained by {\em gluing $G_1$~and~$G_2$}, and we refer to $\{f,f'\}$ as the {\em corresponding even $2$-cut}. Note that $n=n_1+n_2$.
The following is easily observed.

\begin{prp}\label{prp:gluing}
Let $G$ be a graph obtained by gluing two matchable graphs $G_1$~and~$G_2$ at edges $e_1$~and~$e_2$, let~$C$ denote the corresponding even $2$-cut, and let $M_1$~and~$M_2$ denote \pmg s of $G_1$~and~$G_2$, respectively. If $e_i\notin M_i$ for each $i\in \{1,2\}$, then $M_1\cup M_2$ is a \pmg\ of~$G$. Furthermore, if $e_i\in M_i$ for each $i\in \{1,2\}$, then $(M_1-e_1)+(M_2-e_2)+C$ is a \pmg\ of~$G$. \qed
\end{prp}

Next, we proceed to define an operation that may be seen as the inverse of the gluing operation. 
Let $G$ be a \mbox{$2$-connected} graph of even order that has an even $2$-cut~$C:=\{f,f'\}$, and let $H_1$~and~$H_2$ denote the components of~$G-C$. For each $i\in\{1,2\}$, we let $G_i$ denote the ($2$-connected) graph obtained from~$H_i$ by adding a new edge~$e_i$ joining the end of~$f$ (that lies in~$G_i$) with the end of~$f'$ (that lies in~$G_i$). Inspired by the terminology used in Bondy and Murty \cite[Chapter~$9$]{bomu08}, we refer to $G_1$~and~$G_2$ as the {\em marked $C$-components of~$G$}, and their edges $e_1$~and~$e_2$ as the corresponding {\em marker edges}. The following may be viewed as a converse of Proposition~\ref{prp:gluing}, and is easily verified using Lemma~\ref{lem:basic-facts-about-parity}.

\begin{prp}\label{prp:2-bond}
Let $G$ be a $2$-connected matchable graph that has an even $2$-cut~$C$, let $G_1$~and~$G_2$ denote its marked $C$-components with marker edges $e_1$~and~$e_2$, respectively, and let $M$ denote a \pmg\ of~$G$. Then, either $M\cap C$ is empty and $M\cap E(G_i)$ is a \pmg\ of~$G_i$ for each $i\in \{1,2\}$, or otherwise, $M\cap C=C$ and $(M\cap E(G_i))+e_i$ is a \pmg\ of~$G_i$ for each $i\in \{1,2\}$. \qed
\end{prp}

The following is a consequence of the above two propositions.
\begin{cor}\label{prp:2-cut-connection-preserves-mc}
 Let $G$ be a $2$-connected graph of even order that has an even $2$-cut~$C$, and let $G_1$~and~$G_2$ denote its marked $C$-components. Then, $G$ is matching covered if and only if both $G_1$~and~$G_2$ are matching covered.
\qed
\end{cor}


We now prove a strengthening of the above in the context of $r$-graphs.

\begin{prp}\label{prp:marked-C-even-2-cut-components-in-r-graphs-are-also-r-graphs}
    For an $r$-graph~$G$ that has a $2$-cut~$C$, both of its marked $C$-components are also $r$-graphs for the same value of~$r$.
\end{prp}
\begin{proof}
    Let $C:=\{f,f'\}$ denote an (even) $2$-cut of~$G$, and let $f:=u_1u_2$ and $f':=v_1v_2$ 
    so that $u_1$~and~$v_1$ belong to the same shore of~$C$.
    Let $G_1$~and~$G_2$ denote its marked $C$-components so that $u_1,v_1\in V(G_1)$.
    Clearly, $G_1$~and~$G_2$ are $r$-regular graphs of even order;
    since $C$ is a bond of~$G$, both of them are connected.

    By symmetry, it suffices to argue that $G_1$ is an $r$-graph. Suppose not, and let $F$ be an odd cut of~$G_1$ that has fewer than $r$ edges. If $u_1$~and~$v_1$ lie in the same shore of~$F$, then $F$ is an odd cut in~$G$; a contradiction. Now, let $F:=\partial_{G_1}(X)$ and adjust notation so that $u_1\in X$~and~$v_1\notin X$. Observe that $\partial_{G}(X)=F-u_1v_1+f$ is an odd cut in~$G$ with fewer than $r$ edges; a contradiction once again.
\end{proof}

The above proposition inspires the following recursive process that may be used to obtain a list of \ecc{3} $r$-graphs from any given $r$-graph~$G$. If $G$ itself is \ecc{3}, then there is nothing to be done; otherwise, let $C$ denote any $2$-cut, and we consider its marked $C$-components, say $G_1$~and~$G_2$. If either of them has a $2$-cut then we repeat this recursively. Clearly, we obtain a list of \ecc{3} $r$-graphs at this end of this process --- that we refer to as an application of the \textit{$2$-cut decomposition procedure to~$G$}. Since one may choose any $2$-cut at each step,
it is not immediately clear whether any two applications of the $2$-cut decomposition procedure to an $r$-graph~$G$ result in the same list of \ecc{3} $r$-graphs. However, this is indeed the case, and it seems to be a folklore result; see Theorem~\ref{thm:unique-2-cut-decomposition-of-r-graphs}. We shall provide a proof for the sake of completeness. 

For two cuts $C:=\partial(X)$ and $D:=\partial(Y)$ of a graph $G$, we refer to the four sets $X\cap Y, X\cap \overline{Y}, \overline{X}\cap Y$ and $\overline{X}\cap \overline{Y}$ \textit{as the quadrants corresponding to $C$~and~$D$}. The two  cuts
are said to be {\em crossing} if all quadrants are nonempty; otherwise, they are said to be {\em laminar}.
To put it differently, $C$ and $D$ are laminar if and only if one of the two shores of $C$ is a subset of a shore of $D$. We are now ready to prove the uniqueness of the result of the $2$-cut decomposition procedure; throughout, we shall invoke Proposition~\ref{prp:marked-C-even-2-cut-components-in-r-graphs-are-also-r-graphs} implicitly.


\begin{thm}\label{thm:unique-2-cut-decomposition-of-r-graphs}
    Any two applications of the $2$-cut decomposition procedure to an $r$-graph~$G$ result in the same list of \ecc{3} $r$-graphs.
\end{thm}
\begin{proof}
    We proceed by induction on the order of~$G$. If $G$ is \ecc{3}, the desired conclusion holds trivially.
    Now, let $C_1:=\partial(X)$ and $C_2:=\partial(Y)$ denote the initial $2$-cuts (of~$G$) that are chosen by the first and second applications of the $2$-cut decomposition procedure, respectively. 
    If $C_1=C_2$, the desired conclusion follows immediately by the induction hypothesis.
Now, suppose that $C_1$ and $C_2$ are distinct; they are either laminar or otherwise crossing, and we shall consider these cases separately. In each case, we let $G_{11}$~and~$G_{12}$ denote the marked $C_1$-components of~$G$, and we 
let $G_{21}$~and~$G_{22}$ denote the marked $C_2$-components of~$G$. By the induction hypothesis, for any of these four graphs, any two applications of the $2$-cut decomposition procedure result in the same list of \ecc{3} $r$-graphs; we shall use this implicitly in both cases.

    \begin{figure}[!htb]
    \centering
\begin{subfigure}[b]{0.98\textwidth}
\centering
\begin{tikzpicture}[scale=0.8]
\node [draw=none] (0) at (0, 2) {};
		\node [draw=none] (1) at (0, 0) {};
		\node [draw=none] (2) at (5, 2) {};
		\node [draw=none] (3) at (5, 0) {};
		\node [draw=none] (4) at (0.5, 1.5) {};
		\node [draw=none] (5) at (1.5, 1.5) {};
		\node [draw=none] (6) at (0.5, 0.5) {};
		\node [draw=none] (7) at (1.5, 0.5) {};
		\node [draw=none] (8) at (3.5, 1.5) {};
		\node [draw=none] (9) at (4.5, 1.5) {};
		\node [draw=none] (10) at (3.5, 0.5) {};
		\node [draw=none] (11) at (4.5, 0.5) {};
		\node [draw=none] (12) at (1, 2.5) {};
		\node [draw=none] (13) at (1, -0.5) {};
		\node [draw=none] (14) at (4, 2.5) {};
		\node [draw=none] (15) at (4, -0.5) {};
		\node [draw=none] (16) at (1, 2.75) {$C_1$};
		\node [draw=none] (17) at (4, 2.75) {$C_2$};
		\node [draw=none] (18) at (0.75, 2.25) {$X$};
		\node [draw=none] (19) at (2.5, 2.25) {$\overline{X} \cap \overline{Y}$};
		\node [draw=none] (20) at (4.25, 2.25) {$Y$};
        \draw (0.center) to (1.center);
		\draw (1.center) to (3.center);
		\draw (3.center) to (2.center);
		\draw (2.center) to (0.center);
		\draw (4.center) to (5.center);
		\draw (6.center) to (7.center);
		\draw (8.center) to (9.center);
		\draw (10.center) to (11.center);
		\draw (12.center) to (13.center);
		\draw (14.center) to (15.center);
\end{tikzpicture}
\caption{$C_1$ and $C_2$ are laminar $2$-cuts of the $r$-graph~$G$}
  \label{fig:G-laminar-case}
\end{subfigure}
\begin{subfigure}[b]{.48\textwidth}
\centering
\begin{tikzpicture}[scale=0.8]
\node [draw=none] (0) at (0, 2) {};
		\node [draw=none] (1) at (0, 0) {};
		\node [draw=none] (2) at (2, 2) {};
		\node [draw=none] (3) at (2, 0) {};
		\node [draw=none] (4) at (4, 2) {};
		\node [draw=none] (5) at (4, 0) {};
		\node [draw=none] (6) at (8, 2) {};
		\node [draw=none] (7) at (8, 0) {};
		\node [draw=none] (8) at (1.5, 1.5) {};
		\node [draw=none] (9) at (1.5, 0.5) {};
		\node [draw=none] (10) at (4.5, 1.5) {};
		\node [draw=none] (11) at (4.5, 0.5) {};
		\node [draw=none] (12) at (6.5, 1.5) {};
		\node [draw=none] (13) at (7.5, 1.5) {};
		\node [draw=none] (14) at (6.5, 0.5) {};
		\node [draw=none] (15) at (7.5, 0.5) {};
		\node [draw=none] (16) at (7, 2.5) {};
		\node [draw=none] (17) at (7, -0.5) {};
		\node [draw=none] (18) at (7, 2.9) {$C_2$};
		\node [draw=none] (19) at (7.5, 2.25) {$Y$};
		\node [draw=none] (20) at (6, 2.25) {$\overline{X}\cap \overline{Y}$};
		\node [draw=none] (21) at (1, -0.5) {$G_{11}$};
		\node [draw=none] (22) at (6, -0.5) {$G_{12}$};
        \node [draw=none] (23) at (1, 2.25) {$X$};
        \draw (0.center) to (2.center);
		\draw (2.center) to (3.center);
		\draw (0.center) to (1.center);
		\draw (1.center) to (3.center);
		\draw (4.center) to (5.center);
		\draw (5.center) to (7.center);
		\draw (7.center) to (6.center);
		\draw (6.center) to (4.center);
		\draw [bend right=105, looseness=3.75] (10.center) to (11.center);
		\draw (12.center) to (13.center);
		\draw (14.center) to (15.center);
		\draw (16.center) to (17.center);
		\draw [bend left=105, looseness=3.75] (8.center) to (9.center);
\end{tikzpicture}
\caption{Marked $C_1$-components of~$G$}
  \label{fig:G11-G12-laminar-case}
            \end{subfigure}
\begin{subfigure}[b]{.48\textwidth}
\centering
\begin{tikzpicture}[scale=0.8]
\node [draw=none] (0) at (0, 2) {};
		\node [draw=none] (1) at (0, 0) {};
		\node [draw=none] (2) at (4, 2) {};
		\node [draw=none] (3) at (4, 0) {};
		\node [draw=none] (4) at (0.5, 1.5) {};
		\node [draw=none] (5) at (0.5, 0.5) {};
		\node [draw=none] (6) at (1.5, 1.5) {};
		\node [draw=none] (7) at (1.5, 0.5) {};
		\node [draw=none] (8) at (3.5, 1.5) {};
		\node [draw=none] (9) at (3.5, 0.5) {};
		\node [draw=none] (10) at (1, 2.5) {};
		\node [draw=none] (11) at (1, -0.5) {};
		\node [draw=none] (12) at (1, 2.9) {$C_1$};
		\node [draw=none] (13) at (6, 2) {};
		\node [draw=none] (14) at (6, 0) {};
		\node [draw=none] (15) at (8, 2) {};
		\node [draw=none] (16) at (8, 0) {};
		\node [draw=none] (17) at (6.5, 1.5) {};
		\node [draw=none] (18) at (6.5, 0.5) {};
		\node [draw=none] (19) at (0.5, 2.25) {$X$};
		\node [draw=none] (20) at (1.75, 2.25) {$\overline{X} \cap \overline{Y}$};
		\node [draw=none] (21) at (2, -0.5) {$G_{21}$};
		\node [draw=none] (22) at (7, -0.5) {$G_{22}$};
		\node [draw=none] (23) at (7, 2.25) {$Y$};
        \draw (0.center) to (1.center);
		\draw (1.center) to (3.center);
		\draw (3.center) to (2.center);
		\draw (2.center) to (0.center);
		\draw (4.center) to (6.center);
		\draw (5.center) to (7.center);
		\draw (10.center) to (11.center);
		\draw [bend left=105, looseness=3.75] (8.center) to (9.center);
		\draw (13.center) to (15.center);
		\draw (15.center) to (16.center);
		\draw (16.center) to (14.center);
		\draw (13.center) to (14.center);
		\draw [bend right=105, looseness=3.75] (17.center) to (18.center);
\end{tikzpicture}
\caption{Marked $C_2$-components of~$G$}
  \label{fig:G21-G22-laminar-case}
            \end{subfigure}
 \begin{subfigure}[b]{.48\textwidth}
\centering
\begin{tikzpicture}[scale=0.8]
\node [draw=none] (0) at (0, 2) {};
		\node [draw=none] (1) at (0, 0) {};
		\node [draw=none] (2) at (3, 2) {};
		\node [draw=none] (3) at (3, 0) {};
		\node [draw=none] (4) at (0.5, 1.5) {};
		\node [draw=none] (5) at (0.5, 0.5) {};
		\node [draw=none] (6) at (2.5, 1.5) {};
		\node [draw=none] (7) at (2.5, 0.5) {};
		\node [draw=none] (8) at (5, 2) {};
		\node [draw=none] (9) at (5, 0) {};
		\node [draw=none] (10) at (7, 2) {};
		\node [draw=none] (11) at (7, 0) {};
		\node [draw=none] (12) at (5.5, 1.5) {};
		\node [draw=none] (13) at (5.5, 0.5) {};
		\node [draw=none] (14) at (1.5, 2.25) {$\overline{X}\cap \overline{Y}$};
		\node [draw=none] (15) at (6, 2.25) {$Y$};
		\node [draw=none] (16) at (1.5, -0.5) {$G_{12}'$};
		\node [draw=none] (17) at (6, -0.5) {$G_{12}''$};
        \draw (0.center) to (2.center);
		\draw (2.center) to (3.center);
		\draw (3.center) to (1.center);
		\draw (1.center) to (0.center);
		\draw [bend right=105, looseness=3.75] (4.center) to (5.center);
		\draw [bend left=105, looseness=3.75] (6.center) to (7.center);
		\draw (8.center) to (10.center);
		\draw (10.center) to (11.center);
		\draw (11.center) to (9.center);
		\draw (9.center) to (8.center);
		\draw [bend right=105, looseness=3.75] (12.center) to (13.center);
\end{tikzpicture}
\caption{Marked $C_2$-components of~$G_{12}$}
  \label{fig:G12'-G12''-laminar-case}
            \end{subfigure}
            \begin{subfigure}[b]{.48\textwidth}
            \centering
            \begin{tikzpicture}[scale=0.8]
\node [draw=none] (0) at (5, 2) {};
		\node [draw=none] (1) at (5, 0) {};
		\node [draw=none] (2) at (8, 2) {};
		\node [draw=none] (3) at (8, 0) {};
		\node [draw=none] (4) at (5.5, 1.5) {};
		\node [draw=none] (5) at (5.5, 0.5) {};
		\node [draw=none] (6) at (7.5, 1.5) {};
		\node [draw=none] (7) at (7.5, 0.5) {};
		\node [draw=none] (8) at (0, 2) {};
		\node [draw=none] (9) at (0, 0) {};
		\node [draw=none] (10) at (2, 2) {};
		\node [draw=none] (11) at (2, 0) {};
		\node [draw=none] (12) at (1.5, 1.5) {};
		\node [draw=none] (13) at (1.5, 0.5) {};
		\node [draw=none] (14) at (6.5, 2.25) {$\overline{X}\cap \overline{Y}$};
		\node [draw=none] (15) at (1, 2.25) {$X$};
		\node [draw=none] (16) at (6.5, -0.5) {$G_{21}'$};
		\node [draw=none] (17) at (1, -0.5) {$G_{21}''$};
        \draw (0.center) to (2.center);
		\draw (2.center) to (3.center);
		\draw (3.center) to (1.center);
		\draw (1.center) to (0.center);
		\draw [bend right=105, looseness=3.75] (4.center) to (5.center);
		\draw [bend left=105, looseness=3.75] (6.center) to (7.center);
		\draw (8.center) to (10.center);
		\draw (10.center) to (11.center);
		\draw (11.center) to (9.center);
		\draw (9.center) to (8.center);
		\draw [bend left=105, looseness=3.75] (12.center) to (13.center);
\end{tikzpicture}
\caption{Marked $C_1$-components of~$G_{21}$}
  \label{fig:G21''-G21'-laminar-case}
            \end{subfigure}
            \caption{Illustration for the proof of Theorem~\ref{thm:unique-2-cut-decomposition-of-r-graphs} in the laminar case}
            \label{fig:the-laminar-case}
        \end{figure}

    First suppose that $C_1=\partial(X)$~and~$C_2=\partial(Y)$ are laminar in~$G$; adjust notation so that the quadrant $X\cap{Y}$ is empty, $C_2$ is a $2$-cut of~$G_{12}$ and $C_1$ is a $2$-cut of~$G_{21}$ as shown in Figures \ref{fig:G-laminar-case}, \ref{fig:G11-G12-laminar-case} and \ref{fig:G21-G22-laminar-case}. We conveniently choose an application of the $2$-cut decomposition procedure to~$G_{12}$ that starts with~$C_{2}$,
    and we let $G_{12}'$~and~$G_{12}''$ denote the marked $C_2$-components of~$G_{12}$ adjusting notation as shown in Figure~\ref{fig:G12'-G12''-laminar-case}. Likewise, we conveniently choose an application of the $2$-cut decomposition to~$G_{21}$ starting with~$C_{1}$, and we let $G_{21}'$~and~$G_{21}''$ denote the marked $C_1$-components of~$G_{21}$ adjusting notation as shown in Figure~\ref{fig:G21''-G21'-laminar-case}. Observe that the graphs $G_{11},G_{12}'$ and $G_{12}''$ are isomorphic to $G_{21}'',G_{21}'$ and $G_{22}$, respectively. By the induction hypothesis, these smaller $r$-graphs satisfy the desired conclusion; hence, so does $G$.

        \begin{figure}[!htb]
        \centering
\begin{tikzpicture}[scale=1.2]

		\node [draw=none] (0) at (0, 4) {};
		\node [draw=none] (1) at (0, 0) {};
		\node [draw=none] (2) at (4, 4) {};
		\node [draw=none] (3) at (4, 0) {};
		\node [draw=none] (4) at (2, 4.5) {};
		\node [draw=none] (5) at (2, -0.5) {};
		\node [draw=none] (6) at (-0.5, 2) {};
		\node [draw=none] (7) at (4.5, 2) {};
		\node [draw=none] (8) at (1.5, 3) {};
		\node [draw=none] (9) at (2.5, 3) {};
		\node [draw=none] (10) at (1, 2.5) {};
		\node [draw=none] (11) at (1, 1.5) {};
		\node [draw=none] (12) at (1.5, 1) {};
		\node [draw=none] (13) at (2.5, 1) {};
		\node [draw=none] (14) at (3, 1.5) {};
		\node [draw=none] (15) at (3, 2.5) {};
		\node [draw=none] (16) at (-0.8, 2) {$C_1$};
		\node [draw=none] (17) at (2, -0.8) {$C_2$};
		\node [draw=none] (18) at (-0.5, 2.5) {$X$};
		\node [draw=none] (19) at (-0.5, 1.5) {$\overline{X}$};
		\node [draw=none] (20) at (1.5, -0.5) {$Y$};
		\node [draw=none] (21) at (2.5, -0.5) {$\overline{Y}$};
        \node [draw=none] (26) at (0.5, 3.5) {$X\cap Y$};
		\node [draw=none] (27) at (3.5, 3.5) {$X\cap \overline{Y}$};
		\node [draw=none] (28) at (0.5, 0.5) {$\overline{X}\cap Y$};
		\node [draw=none] (29) at (3.5, 0.5) {$\overline{X}\cap\overline{Y}$};

        \draw (0.center) to (2.center);
		\draw (2.center) to (3.center);
		\draw (3.center) to (1.center);
		\draw (1.center) to (0.center);
		\draw (4.center) to (5.center);
		\draw (7.center) to (6.center);
		\draw (8.center) to (9.center);
		\draw (10.center) to (11.center);
		\draw (12.center) to (13.center);
		\draw (14.center) to (15.center);
\end{tikzpicture}
\caption{Illustration for the proof of Theorem~\ref{thm:unique-2-cut-decomposition-of-r-graphs} in the crossing cuts case}
  \label{fig:G-crossing-case}
        \end{figure}


    Now suppose that $C_1$~and~$C_2$ are crossing cuts in~$G$. Since $C_1$~and~$C_2$ are bonds, their shores induce connected subgraphs. Consequently, for any two quadrants whose union is a shore of~$C_1$ or of $C_2$, there is at least one edge of~$G$ that has an end in each of these quadrants. This observation implies the following: (i) $|C_1\cup C_2|=4$, (ii) each quadrant $Q$ satisfies $|\partial_G(Q)|=2$, and (iii) by definition of $r$-graph, cardinality of each quadrant is even; see Figure~\ref{fig:G-crossing-case}.

    Now, let $i,j\in\{1,2\}$. The subgraph of~$G$ induced by $V(G_{ij})$ has a unique even $1$-cut $e_{ij}$ that lies in~$C_1\cup C_2$. The edge~$e_{ij}$, along with the marker edge of~$G_{ij}$, comprise a $2$-cut $C_{ij}$ of~$G_{ij}$. Ergo, we conveniently choose an application of the $2$-cut decomposition procedure to~$G_{ij}$ that starts with $C_{ij}$, and we let $G_{ij}'$ and $G_{ij}''$ denote the marked $C_{ij}$-components of~$G_{ij}$. Observe that each of $G_{ij}'$ and $G_{ij}''$ is isomorphic to $G[Q]+uv$, where $Q$ is some quadrant, and $u,v\in Q$ denote ends of the edges in~$\partial_{G}(Q)$; see Figure~\ref{fig:G-crossing-case}.
    Consequently, the lists $\{G_{11}',G_{11}'',G_{12}',G_{12}''\}$ and $\{G_{21}',G_{21}'',G_{22}',G_{22}''\}$ comprise the same graphs up to isomorphism. By the induction hypothesis, their members satisfy the desired conclusion; hence, so does~$G$.
\end{proof}

Consequently, for any $r$-graph~$G$, we refer to the list of \ecc{3} $r$-graphs obtained from it by an application of the $2$-cut decomposition procedure as the \textit{\ecc{3} pieces of~$G$}.
The following proposition demonstrates that one may recursively compute or determine the solitary edges of an $r$-graph using the solitary edges of its \ecc{3} pieces. We leave its proof as an exercise to the reader. In fact, we shall prove a stronger version in Section~\ref{sec:r-edge-colorability}; see Corollaries~\ref{lem:2-cut-decomposition-r-graphs-new}~and~\ref{cor:2-cut-solitary-stronger-version}.

\begin{prp}\label{prp:marked-c-components-solitary-edges}
    Let $G$ be an $r$-graph that has an even $2$-cut~$C$, let $G_1$~and~$G_2$ denote its marked $C$-components with marker edges $e_1$~and~$e_2$, respectively. Then the following statements hold:
    \begin{enumerate}[(i)]
        \item for any edge $f\in C$, the edge $f$ is solitary in~$G$ if and only if $e_i$ is solitary in~$G_i$, for each $i\in \{1,2\}$, and
        \item for an edge $e\in E(G_1)-e_1$, the edge $e$ is solitary in~$G$ if and only if $e$ is solitary in~$G_1$, the unique perfect matching containing $e$ in~$G_1$ also contains~$e_1$, and $e_2$ is solitary in~$G_2$. \qed
    \end{enumerate} 
\end{prp}

\begin{figure}[!htb]
    \begin{subfigure}[b]{.4\textwidth}
    \centering
    \begin{tikzpicture}[scale=1]

        \node [circle,fill=white] (0) at (4.5, 5.5) {};
		\node [circle,fill=white] (1) at (5.5, 5.5) {};
		\node [circle,fill=white] (2) at (3, 4) {};
		\node [circle,fill=white] (3) at (2.5, 3.5) {};
		\node [circle,fill=white] (4) at (3.5, 3.5) {};
		\node [circle,fill=white] (5) at (3, 3) {};
		\node [circle,fill=white] (6) at (4, 2) {};
		\node [circle,fill=white] (7) at (4, 1) {};
		\node [circle,fill=white] (8) at (4.5, 1.5) {};
		\node [circle,fill=white] (9) at (5.5, 1.5) {};
		\node [circle,fill=white] (10) at (6, 2) {};
		\node [circle,fill=white] (11) at (6, 1) {};
		\node [circle,fill=white] (12) at (7, 3) {};
		\node [circle,fill=white] (13) at (6.5, 3.5) {};
		\node [circle,fill=white] (14) at (7.5, 3.5) {};
		\node [circle,fill=white] (15) at (6.5, 4) {};
		\node [circle,fill=white] (16) at (7.5, 4) {};
		\node [circle,fill=white] (17) at (7, 4.5) {};
		\node [draw=none] (18) at (2.25, 5.25) {};
		\node [draw=none] (19) at (7.5, 2) {};
		\node [draw=none] (20) at (2.5, 1.5) {};
		\node [draw=none] (21) at (3.5, 6) {};
        \node [draw=none] (22) at (2, 5.5) {$C$};
		\node [draw=none] (23) at (3.25, 6.25) {$D$};

        \draw (2) to (3);
		\draw (3) to (5);
		\draw (5) to (4);
		\draw (4) to (2);
		\draw (3) to (4);
		\draw (2) to (0);
		\draw (6) to (7);
		\draw (7) to (8);
		\draw (8) to (6);
		\draw (8) to (9);
		\draw (10) to (9);
		\draw (9) to (11);
		\draw (11) to (10);
		\draw (7) to (11);
		\draw (17) to (15);
		\draw (15) to (16);
		\draw (16) to (17);
		\draw (15) to (13);
		\draw (13) to (14);
		\draw (14) to (16);
		\draw (13) to (12);
		\draw (12) to (14);
		\draw (10) to (12);
		\draw (5) to (6);
		\draw [bend left] (0) to (1);
		\draw [bend right] (0) to (1);
		\draw (1) to (17);
		\draw [bend left=75, looseness=1.25, color=red] (18) to (20);
		\draw [bend right, color=blue] (21) to (19);
    \end{tikzpicture}
    \caption{A $3$-graph~$G$}
    \end{subfigure}
\begin{subfigure}[b]{.46\textwidth}
    \centering
    \begin{tikzpicture}
        \node [circle,fill=white] (0) at (3.5, 5.75) {};
		\node [circle,fill=white] (1) at (4.5, 5.75) {};
		\node [circle,fill=white] (2) at (1.5, 4.5) {};
		\node [circle,fill=white] (3) at (1, 4) {};
		\node [circle,fill=white] (4) at (2, 4) {};
		\node [circle,fill=white] (5) at (1.5, 3.5) {};
		\node [circle,fill=white] (6) at (3, 2.5) {};
		\node [circle,fill=white] (7) at (3, 1.5) {};
		\node [circle,fill=white] (8) at (3.5, 2) {};
		\node [circle,fill=white] (9) at (4.5, 2) {};
		\node [circle,fill=white] (10) at (5, 2.5) {};
		\node [circle,fill=white] (11) at (5, 1.5) {};
		\node [circle,fill=white] (12) at (7, 3) {};
		\node [circle,fill=white] (13) at (6.5, 3.5) {};
		\node [circle,fill=white] (14) at (7.5, 3.5) {};
		\node [circle,fill=white] (15) at (6.5, 4) {};
		\node [circle,fill=white] (16) at (7.5, 4) {};
		\node [circle,fill=white] (17) at (7, 4.5) {};

\draw (2) to (3);
		\draw (3) to (5);
		\draw (5) to (4);
		\draw (4) to (2);
		\draw (3) to (4);
		\draw (6) to (7);
		\draw (7) to (8);
		\draw (8) to (6);
		\draw (8) to (9);
		\draw (10) to (9);
		\draw (9) to (11);
		\draw (11) to (10);
		\draw (7) to (11);
		\draw (17) to (15);
		\draw (15) to (16);
		\draw (16) to (17);
		\draw (15) to (13);
		\draw (13) to (14);
		\draw (14) to (16);
		\draw (13) to (12);
		\draw (12) to (14);
		\draw [bend left] (0) to (1);
		\draw [bend right] (0) to (1);
		\draw [bend left=105, looseness=3.00] (2) to (5);
		\draw  (0) to (1);
		\draw [bend left] (6) to (10);
		\draw [bend right=270, looseness=2.25] (12) to (17);
    \end{tikzpicture}
    \caption{$3$-edge-connected pieces of~$G$}
\end{subfigure}
\begin{subfigure}[b]{.5\textwidth}
    \centering
    \begin{tikzpicture}
        \node [circle,fill=white] (0) at (4.5, 5.5) {};
		\node [circle,fill=white] (1) at (5.5, 5.5) {};
		\node [circle,fill=white] (2) at (1.5, 4) {};
		\node [circle,fill=white] (3) at (1, 3.5) {};
		\node [circle,fill=white] (4) at (2, 3.5) {};
		\node [circle,fill=white] (5) at (1.5, 3) {};
		\node [circle,fill=white] (6) at (4, 2) {};
		\node [circle,fill=white] (7) at (4, 1) {};
		\node [circle,fill=white] (8) at (4.5, 1.5) {};
		\node [circle,fill=white] (9) at (5.5, 1.5) {};
		\node [circle,fill=white] (10) at (6, 2) {};
		\node [circle,fill=white] (11) at (6, 1) {};
		\node [circle,fill=white] (12) at (7, 3) {};
		\node [circle,fill=white] (13) at (6.5, 3.5) {};
		\node [circle,fill=white] (14) at (7.5, 3.5) {};
		\node [circle,fill=white] (15) at (6.5, 4) {};
		\node [circle,fill=white] (16) at (7.5, 4) {};
		\node [circle,fill=white] (17) at (7, 4.5) {};

\draw (2) to (3);
		\draw (3) to (5);
		\draw (5) to (4);
		\draw (4) to (2);
		\draw (3) to (4);
		\draw (6) to (7);
		\draw (7) to (8);
		\draw (8) to (6);
		\draw (8) to (9);
		\draw (10) to (9);
		\draw (9) to (11);
		\draw (11) to (10);
		\draw (7) to (11);
		\draw (17) to (15);
		\draw (15) to (16);
		\draw (16) to (17);
		\draw (15) to (13);
		\draw (13) to (14);
		\draw (14) to (16);
		\draw (13) to (12);
		\draw (12) to (14);
		\draw (10) to (12);
		\draw [bend left] (0) to (1);
		\draw [bend right] (0) to (1);
		\draw (1) to (17);
		\draw [bend left=105, looseness=3.00, color=red] (2) to (5);
		\draw [bend right=45, looseness=1.25, color=red] (0) to (6);
    \end{tikzpicture}
    \caption{Marked $C$-components of~$G$}
\end{subfigure}
\begin{subfigure}[b]{.5\textwidth}
    \centering
    \begin{tikzpicture}
        \node [circle,fill=white] (0) at (4.5, 5.5) {};
		\node [circle,fill=white] (1) at (5.5, 5.5) {};
		\node [circle,fill=white] (2) at (1.5, 4.5) {};
		\node [circle,fill=white] (3) at (1, 4) {};
		\node [circle,fill=white] (4) at (2, 4) {};
		\node [circle,fill=white] (5) at (1.5, 3.5) {};
		\node [circle,fill=white] (6) at (2.5, 2.5) {};
		\node [circle,fill=white] (7) at (2.5, 1.5) {};
		\node [circle,fill=white] (8) at (3, 2) {};
		\node [circle,fill=white] (9) at (4, 2) {};
		\node [circle,fill=white] (10) at (4.5, 2.5) {};
		\node [circle,fill=white] (11) at (4.5, 1.5) {};
		\node [circle,fill=white] (12) at (7, 3) {};
		\node [circle,fill=white] (13) at (6.5, 3.5) {};
		\node [circle,fill=white] (14) at (7.5, 3.5) {};
		\node [circle,fill=white] (15) at (6.5, 4) {};
		\node [circle,fill=white] (16) at (7.5, 4) {};
		\node [circle,fill=white] (17) at (7, 4.5) {};

\draw (2) to (3);
		\draw (3) to (5);
		\draw (5) to (4);
		\draw (4) to (2);
		\draw (3) to (4);
		\draw (6) to (7);
		\draw (7) to (8);
		\draw (8) to (6);
		\draw (8) to (9);
		\draw (10) to (9);
		\draw (9) to (11);
		\draw (11) to (10);
		\draw (7) to (11);
		\draw (17) to (15);
		\draw (15) to (16);
		\draw (16) to (17);
		\draw (15) to (13);
		\draw (13) to (14);
		\draw (14) to (16);
		\draw (13) to (12);
		\draw (12) to (14);
		\draw (5) to (6);
		\draw [bend left] (0) to (1);
		\draw [bend right] (0) to (1);
		\draw (1) to (17);
		\draw [bend left=45, color=blue] (2) to (10);
		\draw [bend right=45,color=blue] (0) to (12);
    \end{tikzpicture}
    \caption{Marked $D$-components of~$G$}
\end{subfigure}
\caption{An illustration for the $2$-cut decomposition procedure}
\label{An-illustration-for-the-$2$-cut-decomposition-procedure}
\end{figure}


Thus, we now switch our attention to solitude in \ecc{3} $r$-graphs.

\subsection{Solitary edges in \texorpdfstring{\ecc{3}}{} \texorpdfstring{$r$}{}-graphs}

In the case of \ecc{3} $r$-graphs, we establish the following constant upper bounds.

\begin{thm}
    \label{thm:main-thm-in-simple-words}
    Every \ecc{3} $r$-graph~$G$, of order four or more, satisfies the following:
    \begin{enumerate}[(i)]
        \item if $r=3$ then $G$ has at most six solitary edges, and
        \item if $r\ge 4$ then $G$ has at most four solitary edges.
    \end{enumerate}
\end{thm}

A proof of the above appears in Section~\ref{sec:LM}.
Goedgebeur, Mazzuoccolo, Renders, Wolf and the second author~\cite{gmmrw24} proved Theorem~\ref{thm:main-thm-in-simple-words}~(i), which applies only to $3$-connected $3$-regular graphs, in their recent work. However, we prove stronger results that apply to all \mbox{\ecc{3}} $r$-graphs (see Corollaries~\ref{cor:main-thm}~and~\ref{lem:solitary-patterns-in-r-graphs}); in order to describe these, we need additional concepts from the theory of \mcg s, and these are discussed in the next section. Before that, we briefly outline the approach 
used in ~\cite{gmmrw24} which is quite different from ours.


The class of {\em Klee graphs}, denoted by $\mathcal{K}$, is defined recursively as follows: $\theta \in \mathcal{K}$, and if $H\in \mathcal{K}$ then every $3$-regular graph~$G$ obtained from~$H$ by replacing any vertex by a \textit{triangle} (that is, $C_3$) also belongs to~$\mathcal{K}$; this operation is equivalent to {\em splicing} with~$K_4$ that is defined in Section~\ref{sec:Splicing and separating cuts}.
Fowler~\cite{f98} proved that the class~$\mathcal{K}$ comprises those planar $3$-regular graphs that are uniquely \mbox{$3$-edge-colorable}. 
Esperet, Kr\'al, \v{S}koda and \v{S}krekovski~\cite{ekss10} proved the following.

\begin{thm}\label{thm:Klee}
Every $3$-connected $3$-regular graph, that is not in $\mathcal{K}$, is matching double covered.
\end{thm}

Equivalently, if $G$ is a $3$-connected $3$-regular graph that has a solitary edge then $G\in \mathcal{K}$. However, the converse does not hold; for instance, the Klee graph of order twelve, that is obtained from~$K_4$ by replacing each of its vertices by a triangle, is devoid of solitary edges. The above theorem suggests that one may be able to provide a recursive  description of all $3$-connected $3$-regular graphs that have a solitary edge; this is precisely what 
 Goedgebeur, Mazzuoccolo, Renders, Wolf and the second author~\cite{gmmrw24} demonstrate in their recent work. Additionally, they provide a complete characterization of those $3$-connected $3$-regular graphs that have at least three solitary edges.

In this paper, we generalize their result to obtain a complete characterization of those \mbox{\ecc{3}} \mbox{$r$-graphs} that have at least three solitary edges. Our approach for proving Theorem~\ref{thm:main-thm-in-simple-words} (and its stronger version Corollary~\ref{lem:solitary-patterns-in-r-graphs}) relies heavily on the dependence relation and its associated theory introduced and developed by Carvalho, Lucchesi and Murty, abbreviated to CLM henceforth; see~\cite{clm99} and \cite{lumu24}. 





\subsection{Matching covered graphs}

Proposition~\ref{prp:r-graph-mc} suggests that in order to prove results pertaining to $r$-graphs, one may find the theory of \mcg s useful, and this indeed turns out to be the case in our work.

\subsubsection{CLM's dependence relation}\label{sec:CLM's-DR}

Let $G$ be a \mcg.
An edge $e$ {\em depends on} an edge $f$, denoted as $e \xrightarrow{G} f$, if every perfect matching that includes $e$ also includes $f$. We note that this relation is not symmetric. In Figure~\ref{fig:R10}, the edge $45$ depends on the edge $23$; however, $23$ does not depend on $45$.

 Carvalho, Lucchesi and Murty \cite{clm99} used the dependence relation to define an equivalence relation called {\em mutual dependence} as follows.
Edges $e$ and $f$ are {\em mutually dependent}, denoted as $e \xleftrightarrow{G} f$, if both $e \xrightarrow{G} f$ and $f \xrightarrow{G} e$. When $G$ is clear from the context, we simplify $e \xrightarrow{G} f$ and $e \xleftrightarrow{G} f$ to $e \rightarrow f$ and $e \leftrightarrow f$, respectively.
Thus, mutual dependence induces a partition of the edge set $E(G)$ of a \mcg\ $G$. Throughout this paper, we denote this partition of $E(G)$ by $\mathcal{E}_G$. Each part of $\mathcal{E}_G$ is referred to as an {\em equivalence class} of~$G$. The graph $R_{10}$, shown in Figure~\ref{fig:R10}, has thirteen equivalence classes; two of them are $\{12,09\}$ and $\{13,08\}$, and each of the remaining equivalence classes is a singleton.
We invite the reader to make the following observation.
\begin{prp}\label{prp:pm-union-of-eqs}
In a \mcg, every \pmg\ is the union of some equivalence classes. \qed
\end{prp}


We now combine the above discussed equivalence classes and dependence relation to define the {\em dependence poset} $(\mathcal{E}_G,\rightarrow)$ as follows.
For a matching covered graph $G$, an equivalence class $D_1$ {\em depends on} an equivalence class $D_2$, denoted as $D_1 \rightarrow D_2$, if there exist edges $e_1\in D_1$ and $e_2\in D_2$ such that $e_1\rightarrow e_2$. It follows from the definitions that $D_1 \rightarrow D_2$ if and only if  $e_1 \rightarrow e_2$ for each pair $e_1\in D_1$ and $e_2\in D_2$. Figure~\ref{fig:R10-poset} shows the dependence poset $(\mathcal{E}_{R_{10}},\rightarrow)$.

We refer to each minimal element of this poset $(\mathcal{E}_G,\rightarrow)$ as a {\em minimal class} of $G$. Observe that an equivalence class $D$ is minimal if and only if $e$ does not depend on $f$ for each pair $e\in E(G)-D$ and $f\in D$. For the graph $R_{10}$, the minimal classes are $\{12,09\},\{13,08\},\{45\}$ and $\{67\}$; see Figure~\ref{fig:R10-poset}.


\begin{figure}[!htb]
    \centering
    \begin{subfigure}[b]{.3\textwidth}
        \centering
        \begin{tikzpicture}[scale=0.7]
           \node [circle,fill=white] (0) at (0, -0.5) {};
		\node [circle,fill=white] (1) at (1, 0) {};
		\node [circle,fill=white] (2) at (1, -1) {};
		\node [circle,fill=white] (3) at (2, 0) {};
		\node [circle,fill=white] (4) at (2, -1) {};
		\node [circle,fill=white] (5) at (3, 0) {};
		\node [circle,fill=white] (6) at (3, -1) {};
		\node [circle,fill=white] (7) at (4, 0) {};
		\node [circle,fill=white] (8) at (4, -1) {};
		\node [circle,fill=white] (9) at (5, -0.5) {};
		\node [draw=none] (10) at (-0.475, -0.5) {$1$};
		\node [draw=none] (11) at (1.025, 0.4) {$2$};
		\node [draw=none] (12) at (1, -1.475) {$3$};
		\node [circle,fill=white] (13) at (2, -1.5) {};
		\node [draw=none] (14) at (2, -1.475) {$4$};
		\node [draw=none] (16) at (2, 0.45) {$5$};
		\node [draw=none] (17) at (3, 0.45) {$6$};
		\node [draw=none] (18) at (3, -1.475) {$7$};
		\node [draw=none] (19) at (4, -1.475) {$8$};
		\node [draw=none] (20) at (3.875, 0.425) {$9$};
		\node [circle,fill=white] (21) at (4, 0) {};
		\node [draw=none] (22) at (5.4, -0.5) {$0$};
        \draw (0) -- (1);
		\draw (1) -- (3);
		\draw (3) -- (5);
		\draw (5) -- (7);
		\draw (7) -- (9);
		\draw (9) -- (8);
		\draw (8) -- (6);
		\draw (6) -- (4);
		\draw (4) -- (2);
		\draw (2) -- (0);
		\draw (2) -- (1);
		\draw (3) -- (4);
		\draw (5) -- (6);
		\draw (7) -- (8);
		\draw [bend left=70, looseness=1.25] (0) to (9);
        \end{tikzpicture}
\caption{$R_{10}$}
  \label{fig:R10}
    \end{subfigure}
    \begin{subfigure}[b]{.6\textwidth}
        \centering
        \begin{tikzpicture}[scale=0.8]
            \node [draw=none] (0) at (-2, -0.5) {$\{12,09\}$};
		\node [draw=none] (1) at (-3, 1) {$34$};
		\node [draw=none] (2) at (-2, 1) {$56$};
		\node [draw=none] (3) at (-1, 1) {$78$};
		\node [draw=none] (4) at (2, -0.5) {$\{13,08\}$};
		\node [draw=none] (5) at (1, 1) {$25$};
		\node [draw=none] (6) at (2, 1) {$47$};
		\node [draw=none] (7) at (3, 1) {$69$};
		\node [draw=none] (8) at (5, -0.5) {$45$};
		\node [draw=none] (9) at (5, 1) {$23$};
		\node [draw=none] (10) at (6, -0.5) {$67$};
		\node [draw=none] (11) at (6, 1) {$89$};
		\node [draw=none] (12) at (5.5, 2.5) {$01$};
        \draw[->] (0) -- (1);
        \draw[->] (0) -- (2);
        \draw[->] (0) -- (3);
        \draw[->] (4) -- (5);
        \draw[->] (4) -- (6);
        \draw[->] (4) -- (7);
        \draw[->] (8) -- (9);
        \draw[->] (9) -- (12);
        \draw[->] (10) -- (11);
        \draw[->] (11) -- (12);
        \end{tikzpicture}
\caption{Hasse diagram of $(\mathcal{E}_{R_{10}},\rightarrow)$}
  \label{fig:R10-poset}
    \end{subfigure}
    \caption{The graph $R_{10}$ and its dependence poset}
    \label{fig:posetexample}
\end{figure}

Observe that if $D$ is a minimal class of a \mcg~$G$ distinct from~$K_2$, then each component of $G-D$ is \mc; however, $G-D$ may have more than one component. We say that a minimal class~$D$ is a {\em removable class} if $G-D$ is connected. In other words, an equivalence class $D$ is removable if and only if $G-D$ is \mc. This notion plays a key role in matching theory, and is intrinsically related to the ear decomposition theory of \mcg s that we do not discuss in this paper; see Lucchesi and Murty~\cite{lumu24}.


We now proceed to make simple observations pertaining to solitary edges with respect to the dependence poset. Let $e$ denote a solitary edge of a \mcg~$G$, and let $M_e$ denote the unique \pmg\ containing $e$. Recall that if some edge~$f$ depends on~$e$ then every perfect matching containing $f$ also contains $e$. Since $G$ is \mc, $f\in M_e$, the edge $f$ is also solitary and $e\rightarrow f$.
This proves the following.

\begin{lem}
\label{lem:solitary-source-clasee}
Let $e$~and~$f$ denote edges of a \mcg~$G$. If $f \rightarrow e$ and $e$ is solitary, then $f$ is also solitary and $e \leftrightarrow f$. \hfill \qed
\end{lem}

Using the above lemma and our preceding discussion on minimal classes, we infer the following.

\begin{cor}
\label{prp:solitary-dependence-clasee}
    Let $D$ be the equivalence class of any solitary edge in a \mcg~$G$. Then, every edge in $D$ is solitary. Furthermore, $D$ is a minimal class. \qed
\end{cor}

In light of the above fact, we refer to the equivalence class of any solitary edge as a {\em solitary class}, and we refer to the remaining equivalence classes as {\em non-solitary classes}. We thus have the following consequence.

\begin{cor}
\label{cor:solitary-source-clasee}
    The members of $\mathcal{E}_G$ of any matching covered graph $G$ may be classified into two different types of equivalence classes: {solitary classes} and {non-solitary classes}. Furthermore, every solitary class is a {minimal class}. \hfill \qed
\end{cor}

Note that the converse of the second part need not hold in general. For instance, for the graph $R_{10}$, shown in Figure~\ref{fig:R10}, the minimal class $\{45\}$ is not solitary. 

In view of Corollary~\ref{cor:solitary-source-clasee}, for a \mcg, we use the term {\em solitude pattern} to refer to the (possibly empty) sequence of cardinalities of its solitary classes in nonincreasing order. For instance, $\theta$ has solitude pattern $(1,1,1)$, each of $K_4$ and $\overline{C_6}$ has solitude pattern $(2,2,2)$, whereas the bicorn $R_8$ has solitude pattern $(2,2,1)$; see Figure~\ref{fig:2-connected cubic graphs with too many solitary edges}. On the other hand, the Petersen graph has the trivial solitude pattern~$()$. 
The reader may find it instructive to verify the following by invoking Propositions~\ref{r-graphs-of-order-two}~and~\ref{r-graphs-of-order-four}. 
\begin{prp}\label{prp:solitary-patterns-of-small-r-graphs}
    In an $r$-graph of order two, each edge comprises a solitary class, whereas every $r$-graph of order four has one of the following four solitude patterns: $(2,2,2)$, $(2,2)$, $(2)$ and $()$. \qed
\end{prp}


Our next goal is to establish an upper bound on the cardinality of any solitary class in a \mcg\ that is free of even $2$-cuts; this will bring us one step closer towards proving Theorem~\ref{thm:main-thm-in-simple-words}. To this end, we shall find the following result of 
Lucchesi and Murty~\cite{lumu24} useful; it establishes that the existence of a ``large'' minimal class implies ``low'' edge-connectivity.
\begin{thm}{\sc[Lucchesi-Murty Theorem]}
\label{thm:source-class-2-cut}
\newline Any minimal class $D$ of a matching covered graph $G$, such that $|D|\ge 3$, includes an even $2$-cut of~$G$. Consequently, if $G$ is free of even $2$-cuts, then each minimal class has cardinality one or two.
\end{thm}


It follows from the above that in a \mcg\ free of even $2$-cuts, every removable class has cardinality one or two. In light of this, a \textit{removable edge} of a \mcg\ is any removable class of cardinality one, whereas a \textit{removable doubleton} is any removable class of cardinality two.

Using the above and Corollary~\ref{prp:solitary-dependence-clasee}, we arrive at the following upper bound.

\begin{cor}\label{cor:lm-cor}
    In a \mcg\ that is free of even $2$-cuts, each solitary class has cardinality one or two. \qed
\end{cor}

In view of this, we use the term {\em solitary singleton} to refer to a solitary class of cardinality one; likewise, we use {\em \sd} to refer to a solitary class of cardinality two. For instance, the bicorn $R_8$ has two \sd s and precisely one solitary singleton.

 It is worth noting that minimality is really crucial in the Lucchesi-Murty Theorem~(\ref{thm:source-class-2-cut}). In particular, 
 Lu, Kothari, Feng and Zhang~\cite{lkfz20} constructed matching covered graphs that have arbitrarily large equivalence classes as well as arbitrarily high vertex-connectivity.

We use $\varepsilon(G)$ to denote the cardinality of a largest equivalence class of a \mcg~$G$. The next observation follows from Proposition~\ref{prp:pm-union-of-eqs} and definitions.

\begin{cor}\label{cor:at-least-n/2-solitary-edges}
For a \mcg~$G$, the inequality  $\varepsilon(G)\le \frac{n}{2}$ holds. Furthermore, any equivalence class of cardinality~$\frac{n}{2}$ is a solitary class and comprises a \pmg\ of~$G$. \qed
\end{cor}

This leads to a natural question: what are those \mcg s that attain the above upper bound tightly? We answer this question in 
Section~\ref{sec:largest-EC-n/2} thus fixing an error in~\cite[Proposition~1.8]{lkfz20} that was spotted by the first author. 
As a consequence, we obtain a complete characterization of those $r$-graphs that attain the upper bound stated in Theorem~\ref{thm:at-most-n/2-solitary-edges}.

We now describe another relation on the edges of a \mcg\ that was also introduced by Carvalho, Lucchesi and Murty~\cite{clm99}, and relate it to solitude --- that is, the state of being solitary.
Two edges $e$ and $f$ of a \mcg~$G$ are {\em mutually exclusive} if each \pmg\ contains at most one of them. In light of Proposition~\ref{prp:pm-union-of-eqs}, we may extend this definition to equivalence classes. Two equivalence classes $D_1$~and~$D_2$ are {\em mutually exclusive} if each \pmg\ includes at most one of them.
We now proceed to prove the following lemma that we shall find useful in establishing that every $r$-graph with two or more solitary classes is $3$-connected; see Corollary~\ref{cor:r-graphs-3-connected}.

\begin{lem}\label{lem:r-graphs-new-lemma}
Let $e_1$ and $e_2$ denote mutually exclusive solitary edges in an~$r$-graph~$G$ that admits a proper \mbox{$r$-edge-coloring}, say $(M_1, M_2, M_3,$ $\dots, M_r)$, and let $M_1$~and~$M_2$ denote the perfect matchings containing $e_1$~and~$e_2$, respectively. Then, for each $i\in \{3,\dots,r\}$, the (spanning) subgraph $H:=M_1\cup M_2\cup M_i$ is a \mbox{$3$-connected (\mbox{$3$-regular})} graph (and $e_1$~and~$e_2$ are mutually exclusive and solitary in~$H$).
\end{lem}
\begin{proof}
We let $H:=M_1\cup M_2\cup M_i$ for any $i\in \{3,\dots,r\}$. 
Observe that it suffices to argue that $H$ is \mbox{$3$-connected}. By Proposition~\ref{prp:mcgs-hcycle}, the union of any two of $M_1, M_2$~and~$M_i$ is a hamiltonian cycle. In particular, $H$ is $2$-connected. By Proposition~\ref{prp:r-graphs-basic-prp}~(iii), it remains to show that there is no even $2$-cut in~$H$. Suppose that $H$ has an even $2$-cut, say $C$. By Lemma~\ref{lem:basic-facts-about-parity}~(i), $C$ is a subset of one of $M_1, M_2$~and~$M_i$; consequently, the union of the other two perfect matchings is a disconnected subgraph, and this contradicts what we have already established.
\end{proof}

We now invite the reader to use the definition of solitary class and Proposition~\ref{prp:pm-union-of-eqs} to observe the first part, and use Lemma~\ref{lem:solitary-source-clasee} to observe the second part, of the following.
\begin{cor}
\label{lem:solitary-class-unique-pm}
Each solitary class $S$ of a \mcg~$G$ meets precisely one perfect matching, say~$M$; furthermore $S\subseteq M$. 
Also, any perfect matching of $G$ contains at most one solitary class, or equivalently, any two solitary classes are mutually exclusive.\hfill \qed
\end{cor}

This yields the following consequence; to see this, consider any proper \mbox{$r$-edge-coloring}.
\begin{cor}
\label{lem:r-reg-rec}
Every connected \mbox{$r$-edge-colorable} $r$-regular graph has at most $r$ solitary classes.\hfill \qed
\end{cor}

Now suppose that $G$ is a connected \mbox{$r$-edge-colorable} $r$-regular graph that has at least $r-1$ solitary classes. It follows from Corollary~\ref{lem:solitary-class-unique-pm} that, in any proper \mbox{$r$-edge-coloring}, these solitary classes determine $r-1$ color classes, and thus the entire coloring (up to permutation of colors). This proves the following.

\begin{cor}\label{cor:urec}
Every connected \mbox{$r$-edge-colorable} $r$-regular graph that has at least $r-1$ solitary classes is uniquely \mbox{$r$-edge-colorable}. \qed
\end{cor}

Combining Theorem~\ref{thm:source-class-2-cut} with Corollaries~\ref{cor:solitary-source-clasee}~and~\ref{lem:r-reg-rec} yields the following.
\begin{cor}\label{cor:at-most-2r-SE}
Every \ecc{3} \mbox{$r$-edge-colorable} $r$-regular graph has at most $2r$ solitary edges. \qed
\end{cor}

We now turn our attention to special types of odd cuts that play an important role in the theory of \mcg s.






\subsubsection{Splicing and separating cuts}\label{sec:Splicing and separating cuts}

For a cut $C:=\partial(X)$ of a graph $G$, we denote the graph obtained from~$G$ by shrinking $X$ to a single vertex $x$ as $G/(X \to x)$, or simply as $G/X$, and we refer to $x$ as the {\em contraction vertex}. The graphs $G/X$ and $G/\overline{X}$ are called the {\em $C$-contractions of $G$}. Observe that each cut of $G/X$ is also a cut of~$G$. This implies that the edge-connectivity of $G/X$ is at least as large as the edge-connectivity of~$G$. 

A cut $C$ of a \mcg~$G$ is a {\em separating cut} if both of its $C$-contractions are also matching covered. The following is an immediate consequence of Proposition~\ref{prp:r-graph-mc} and the preceding discussion.

\begin{prp}\label{prp:r-cuts}
For any odd $r$-cut $C$ of an $r$-graph~$G$, both of its $C$-contractions, say $G_1$~and~$G_2$, are also $r$-graphs; consequently, $C$ is a separating cut of~$G$. Furthermore, if $G$ is $3$-edge-connected then both $G_1$~and~$G_2$ are also $3$-edge-connected. \qed
\end{prp}


It is worth noting that an $r$-graph may have separating cuts that are not $r$-cuts; for instance, the Petersen graph has separating $5$-cuts. We shall find the following easy observation useful.

\begin{prp}\label{prp:pmg-extension-across-sep-cuts}
For a separating cut $C:=\partial(X)$ of a \mcg~$G$, each \pmg~of~$G/X$ extends to a \pmg~of~$G$. \qed
\end{prp}

The following characterization of separating cuts is easily proved; see~\cite[Proposition~2.2]{clm05}.
\begin{prp} {\sc[Characterization of Separating Cuts]}
\label{prp:sep-cuts-characterization}
\newline A cut $C$ of a \mcg~$G$ is a \sepcut\ if and only if, for each edge $e$, there exists a \pmg~$M$ containing $e$ such that $|M\cap C|=1$. \qed
\end{prp}

In many of our proofs, we shall encounter odd $r$-cuts (that are indeed separating cuts by Proposition~\ref{prp:r-cuts}) in our $r$-graph~$G$, and we shall consider the corresponding cut contractions $G_1$~and~$G_2$. In these situations, we shall find it useful to arrive at conclusions pertaining to solitude and mutual exclusiveness (recall~Corollary~\ref{lem:solitary-class-unique-pm}) in $G_1$~and~$G_2$ using our hypotheses regarding the same notions in~$G$. The next lemma deals with this.
\begin{lem}{\sc[Solitude and Mutual Exclusiveness across Separating Cuts]}
\label{lem:solitary-ME-sepcuts} 
\newline
Let $C:=\partial(X)$ be a \sepcut\ in a \mcg~$G$ and let $G':=G/X$. For each pair $e,f\in E(G')$, the following two statements hold:
 \begin{enumerate}[(i)]
            \item if $e$ is solitary in $G$ then $e$ is also solitary in $G'$, and
            \item if $e$ and $f$ are mutually exclusive in~$G$ then they are also mutually exclusive in~$G'$.
        \end{enumerate}
Furthermore, for each pair of edges $d$ and $d'$ in $G[X]$ and $G[\overline{X}]$, respectively, and for each $e\in C$ such that $d\xrightarrow{G} e$, the following statements hold:
        \begin{enumerate}
            \item[(iii)] if $d$ is solitary in $G$ then $e$ is solitary in~$G'$, and
            \item[(iv)] if $d$ and $d'$ are mutually exclusive in $G$ then $e$~and~$d'$ are mutually exclusive in~$G'$.
        \end{enumerate}
    \end{lem}
\begin{proof}
We invite the reader to prove the contrapositives of statements (i)~and~(ii) by using Proposition~\ref{prp:pmg-extension-across-sep-cuts}.

We now proceed to prove the contrapositives of statements~(iii)~and~(iv). Since $C$ is a separating cut, by Proposition~\ref{prp:sep-cuts-characterization}, we let $M$ denote a \pmg~of~$G$ that contains~$d$ such that $|M\cap C|=1$; since $d\xrightarrow{G} e$, we infer that $M\cap C=\{e\}$.

If $M'_1$ and $M'_2$ are distinct \pmg s of $G'$ containing~$e$, then $(M-E(G'))\cup M'_1$ and $(M-E(G'))\cup M'_2$ are distinct \pmg s of~$G$ containing $d$; this proves statement~(iii).
If $M'$ is a \pmg\ of $G'$ containing $e$ and $d'$, then \mbox{$(M-E(G'))\cup M'$} is a \pmg\ of~$G$ containing $d$ and $d'$; this proves statement~(iv).
\end{proof}

We now discuss the inverse of the operation of taking $C$-contractions with respect to a cut~$C$, and relate this to the notion of separating cuts in the context of \mcg s.

For each $i\in \{1,2\}$, we let $G_i$ denote a graph with a specified vertex $v_i$ such that the degree of $v_1$ in $G_1$ is the same as the degree of $v_2$ in $G_2$. Let $\pi$ denote a bijection between $\partial_{G_1}(v_1)$~and~$\partial_{G_2}(v_2)$. We let $(G_1\odot G_2)_{v_1,v_2,\pi}$ denote the graph obtained from the disjoint union of $G_1-v_1$ and $G_2-v_2$ by joining, for each edge $e$ in $\partial_{G_1}(v_1)$, the end of $e$ in $G_1-v_1$ to the end of $\pi(e)$ in $G_2-v_2$. We refer to the graph $(G_1\odot G_2)_{v_1,v_2,\pi}$ as the graph obtained by {\em splicing $G_1$ at $v_1$ and $G_2$ at $v_2$ with respect to the bijection $\pi$}, or simply as a graph obtained by {\em splicing $G_1$~and~$G_2$}. If $G_1$ is $3$-regular and $G_2$ is $K_4$, we simplify the notation to $(G_1\odot K_4)_{v_1}$ since the result does not depend on $v_2$~and~$\pi$. The following fact is easily verified and we shall use it implicitly throughout this paper.

\begin{prp}\label{prp:splicing-with-k4}
Splicing any \cc{3} with~$K_4$ yields a \cc{3}. \qed
\end{prp}

Observe that, if $C$ is a separating cut of a \mcg~$G$, and if $G_1$~and~$G_2$ are the $C$-contractions, then $G$ may be obtained by splicing the \mcg s $G_1$~and~$G_2$. The converse, as stated below, is also verified easily; see~\cite[Proposition~2.1]{clm05}.
\begin{prp}\label{lem:splicing}
Splicing two \mcg s $G_1$ and $G_2$ yields a \mcg~$G$. \qed
\end{prp}

Following Bondy and Murty~\cite{bomu08}, we use the notation $H[A,B]$ to denote a bipartite graph with color classes $A$~and~$B$. 
It is easily proved using Hall's Theorem that a bipartite matchable graph~$H[A,B]$ is \mcg\ if and only if $H-a-b$ is matchable for each pair $a\in A$ and $b\in B$. 
A graph $G$, of order four or more, is {\em bicritical} if $G-u-v$ is matchable for each pair of distinct vertices $u$~and~$v$; clearly, every such graph is nonbipartite. The following facts are easily verified; see \cite{lumu24}. 


\begin{prp}\label{prp:bicritical-proposition}
    Splicing two bipartite graphs yields a bipartite graph, whereas splicing two bicritical graphs yields a bicritical graph. \qed
\end{prp}
We now proceed to define a special class of separating cuts, and related graph classes.

\subsubsection{Bricks}

 A cut~$C$ of a \mcg~$G$ is {\em tight} if $|C\cap M|=1$ for each \pmg~$M$. A \mcg, free of nontrivial tight cuts, is a {\em brace} if it is bipartite, and otherwise a {\em brick}. The following deep result was proved by Edmonds, Lov\'asz and Pulleyblank~\cite{elp82}.

\begin{thm}\label{thm:elp-theorem}
    A graph is a brick if and only if it is $3$-connected and bicritical.
\end{thm}

It is worth noting that their proof relied on linear programming duality; a graph-theoretical proof was given by Szigeti~\cite{szig02}, and more recently, by CLM~\cite{clm18}. Tight cuts, bricks and braces occupy crucial roles in the theory of \mcg s; we refer the interested reader to
\cite{lumu24}. For instance, every \cc{2} free of nontrivial $3$-cuts is either a brick or a brace; see~\cite{kcll20}.
We shall find the following result of Lov\'asz~\cite{lova87} useful; the interested reader may also refer to~\cite[Section 9.1]{lumu24}.

\begin{lem}\label{lem:brick-equivalence-classes}
    For distinct edges $\alpha$~and~$\beta$ of a brick~$G$, if $\alpha \xleftrightarrow{G} \beta$ then $H[A,B]:=G-\alpha-\beta$ is bipartite such that $|A|=|B|$; furthermore, up to relabeling, ends of $\alpha$ lie in~$A$ whereas ends of~$\beta$ lie in~$B$.
\end{lem}

A nonbipartite \mcg~$G$ is \textit{near-bipartite} if it has a removable doubleton~$R$ such that $G-R$ is a bipartite (matching covered) graph. This class of graphs plays a crucial role in the theory of \mcg s; see \cite{lumu24}~and~\cite{koca17, koth17}. We state the following consequence of the above lemma.

\begin{cor}\label{cor:brick-rd-near-bipartite}
    Every brick~$G$, that has a removable doubleton, is near-bipartite. \qed
\end{cor}

In the following section, we shall deduce various implications of the existence of a solitary edge in an $r$-graph, and our proofs of these shall rely heavily on the theory of \mcg s discussed thus far.

\subsection{Implications of the existence of a solitary edge}\label{sec:cuts-and-applications}

Recall that we allow multiple/parallel edges. For a graph~$G$, we use $\mu_{G}(u,v)$ to denote the number of edges joining vertices $u$ and $v$; furthermore, if $e:=uv$ is an edge of~$G$, we let $\mu_{G}(e):=\mu_{G}(u,v)$.
We begin this section with the following easy observation.

\begin{prp}\label{prp:basic-prp}
If $e$ is any solitary edge of a graph~$G$ and $M_e$ is the unique perfect matching containing~$e$, then $\mu_{G}(d)=1$ for each edge $d\in M_e-e$. \qed
\end{prp}

\subsubsection{Uniquely matchable graphs and odd \texorpdfstring{$1$}{}-cuts}

A graph $G$ is {\em uniquely matchable} if it has precisely one \pmg. Any investigation of \ses\ naturally calls for the study of uniquely matchable graphs due to the following observation.

\begin{prp}\label{prp:solitary-calls-for-uniquely-matchable}
An edge $e:=uv$ of a graph $G$ is solitary if and only if $G-u-v$ is uniquely matchable. \hfill\qed
\end{prp}

Following Bondy and Murty \cite{bomu08}, by the {\em null graph}, we mean the graph with no vertices, and we consider this to be uniquely matchable and disconnected.
We now state a result of Kotzig~\cite[Section~$8.6$]{lopl86} pertaining to uniquely matchable graphs and odd $1$-cuts that we shall find very useful in the next section.

\begin{lem}{\sc[Kotzig's Lemma]}
\label{lem:unique-pm-in-a-graph}
\newline Every non-null graph $G$ that has a unique {\pmg}, say $M$, has an odd $1$-cut whose constituent edge, say $f:=xy$, belongs to $M$; furthermore, $M-f$ is the only \pmg~ of $G-x-y$. 
\end{lem}

For a $1$-cut comprising the edge~$f$, we simplify the notation $\{f\}$ to~$f$. In the next section, we discuss the structure of an $r$-graph with respect to a fixed solitary edge. Again, odd cuts --- in particular, odd $r$-cuts --- will play an important role.

\subsubsection{Solitary edges and odd \texorpdfstring{$r$}{}-cuts in \texorpdfstring{$r$}{}-graphs}
We first state and prove a couple of properties of $3$-cuts in \cc{3}s.

\begin{prp}\label{prp:3-cut-matching-laminar}
 {\sc[$3$-cuts in $3$-connected \mbox{$3$-regular} Graphs]}
\newline In a \cc{3}, every nontrivial $3$-cut is a matching, and any two $3$-cuts are laminar.
\end{prp}
\begin{proof}
We invite the reader to prove the first part of the statement.
Now, let $C:=\partial(X)$ and $D:=\partial(Y)$ denote distinct $3$-cuts in a \cc{3}~$G$. Since $C$ is an odd cut, exactly one of the two quadrants $X\cap Y$ and $X\cap \overline{Y}$ is an odd set; adjust notation so that $|X\cap Y|$ is odd; consequently, since $D$ is also an odd cut, $|\overline{X}\cap Y|$ is even whereas $|\overline{X}\cap \overline{Y}|$ is odd.

Suppose to the contrary that all four quadrants are nonempty. Since $G$ is $3$-connected, each nonempty cut has at least three edges; furthermore, since $G$ is \mbox{$3$-regular} each nonempty even cut has at least four edges. Thus, $|\partial(X\cap Y)|+|\partial(X\cap \overline{Y})|+|\partial(\overline{X}\cap Y)|+$ $|\partial(\overline{X}\cap \overline{Y})|\ge 14$. On the other hand, $|\partial(X\cap Y)|+|\partial(X\cap \overline{Y})|+|\partial(\overline{X}\cap Y)|+|\partial(\overline{X}\cap \overline{Y})|=2|C\cup D|\le 12$; a contradiction.
\end{proof}

Now, let $G$ be any \ecc{3} $r$-graph that has order four or more, and let $e:=uv$ denote a solitary edge. By Kotzig's Lemma (\ref{lem:unique-pm-in-a-graph}), $G-u-v$ has an odd $1$-cut, and let $f$ denote any odd $1$-cut. In the following lemma, we try to gain a deeper understanding of the structure of~$G$ with respect to $e$ and $f$; see Figure~\ref{fig:update-fig-with-two-possibilities-r-graphs}.

\begin{figure}[!htb]
    \begin{subfigure}[b]{.45\textwidth}
        \begin{tikzpicture}[scale=0.6]
\draw (-3.5,0) ellipse (2 and 2.8);
\node[draw=none] at (-4.5,0) {$X$};
\draw (3.5,0) ellipse (2 and 2.8);
\node[draw=none] at (4.5,0) {$Y$};
\node[circle,fill=white] (1) at (0,0){};
\node[draw=none] at (0,-0.6) {$v$};
\node[draw=none] at (0.2,1) {$e$};
\node[circle,fill=white] (2) at (0,2){};
\node[draw=none] at (0,2.4) {$u$};
\node[circle,fill=white] (3) at (-2.5,-1.5){};
\node[circle,fill=white] (4) at (2.5,-1.5){};
\node[draw=none] at (0,-1.9) {$f$};
\draw (3) -- (4);
\draw (1) -- (2);
\draw (1) -- (2.5,0);
\draw (1) -- (-2.5,0);
\draw (2) -- (2.5,2);
\draw (2) -- (-2.5,2);
\node[draw=none] at (-1,4.4) {$C$};
\draw (-1,4) -- (-1,-4);
\node[draw=none] at (1,4.4) {$D$};
\draw (1,4) -- (1,-4);
\draw (1) -- (2.5,1);
\draw (2) -- (-2.5,1);
\end{tikzpicture}
    \caption{}
    \label{fig:associated-r-cuts-1}
    \end{subfigure}
    \hspace{25pt}
    \begin{subfigure}[b]{.45\textwidth}
   \begin{tikzpicture}[scale=0.6]
\draw (-3.5,0) ellipse (2 and 2.8);
\node[draw=none] at (-4.5,0) {$X$};
\draw (3.5,0) ellipse (2 and 2.8);
\node[draw=none] at (4.5,0) {$Y$};
\node[circle,fill=white] (1) at (0,0){};
\node[draw=none] at (0,-0.6) {$v$};
\node[draw=none] at (0.2,1) {$e$};
\node[circle,fill=white] (2) at (0,2){};
\node[draw=none] at (0,2.4) {$u$};
\node[circle,fill=white] (3) at (-2.5,-1.5){};
\node[circle,fill=white] (4) at (2.5,-1.5){};
\node[draw=none] at (0,-1.9) {$f$};
\draw (3) -- (4);
\draw (1) -- (2);
\draw (1) -- (2.5,0);
\draw (1) -- (2.5,1);
\draw (2) -- (-2.5,1);
\draw (2) -- (-2.5,2);
\node[draw=none] at (-1,4.4) {$C$};
\draw (-1,4) -- (-1,-4);
\node[draw=none] at (1,4.4) {$D$};
\draw (1,4) -- (1,-4);
\draw (1) -- (2.5,-1);
\draw (2) -- (-2.5,0);
\end{tikzpicture}
    \caption{} 
    \label{fig:associated-r-cuts-2}
    \end{subfigure}
    \caption{Illustration for Lemma \ref{lem:associated-r-cuts}}
    \label{fig:update-fig-with-two-possibilities-r-graphs}
\end{figure}

\begin{lem}{\sc[Structure of an $r$-graph with respect to a fixed solitary edge]}
\label{lem:associated-r-cuts}
\newline Let $e:=uv$ denote a solitary edge of an $r$-graph~$G$, of order four or more, and let $f$ denote an odd $1$-cut of $G-u-v$ that has shores $X$ and $Y$. Then $C:=\partial_{G}(X)$ and $D:=\partial_{G}(Y)$ are (laminar) odd $r$-cuts of $G$ that satisfy: (i) $C\cap D=\{f\}$, (ii)~$X\cap Y=\emptyset$, (iii) $V(G)=X\cup Y\cup \{u,v\}$, (iv)~$|\partial_{G}(u)\cap C|=|\partial_{G}(v)\cap D|$ and $|\partial_{G}(v)\cap C|=|\partial_{G}(u)\cap D|$, and (v)~$G-u-v$ is a connected graph. Also, $\mu_{G}(e')=1$ for each edge~$e'$ in the unique perfect matching containing~$e$. Furthermore, if $G$ is \ecc{3}, then (i)~each of $u$~and~$v$ has at least one neighbour in each of $X$~and~$Y$, and (ii)~$G-u-v$ has a unique $1$-cut, namely~$f$.
\end{lem}
\begin{proof}
We let $G,e:=uv,f,X$ and $Y$ be as defined in the lemma statement. Clearly, $X\cap Y=\emptyset$ and $V(G)=X\cup Y\cup \{u,v\}$. In particular, $C:=\partial_{G}(X)$ and $D:=\partial_{G}(Y)$ are laminar odd cuts of~$G$. Since $f$ is a $1$-cut in~$G-u-v$ with shores $X$~and~$Y$, each edge of $(C\cup D)-f$ has one end in $X\cup Y$ and the other end in $\{u,v\}$; consequently, $C\cap D=\{f\}$ and $(C\cup D)-f=\partial_{G}(X\cup Y)$.

Since $G$ is an $r$-graph, $|C|\ge r$ and $|D|\ge r$. Consequently, $|\partial_{G}(X\cup Y)|\ge 2r-2$. On the other hand, $|\partial_{G}(\{u,v\})|\le 2r-2$. However, $\partial_{G}(\{u,v\})=\partial_{G}(X\cup Y)$. These facts imply that $|\partial_{G}(\{u,v\})|=2r-2$, the multiplicity $\mu_{G}(u,v)=1$, $|C|=|D|=r$, and $|\partial_{G}(u)\cap C|=|\partial_{G}(v)\cap D|$ as well as $|\partial_{G}(v)\cap C|=|\partial_{G}(u)\cap D|$. By Proposition~\ref{prp:r-graphs-basic-prp}, both subgraphs $G[X]$~and~$G[Y]$ are connected; consequently, $G-u-v$ is connected. Since $\mu_{G}(u,v)=1$, using Proposition~\ref{prp:basic-prp}, we conclude that $\mu_{G}(e')=1$ for each edge $e'$ in the unique perfect matching containing~$e$.

Furthermore, the reader may verify that the structure of $G$ with respect to $e$ and $f$ is as shown in one of the two Figures~\ref{fig:associated-r-cuts-1}~and~\ref{fig:associated-r-cuts-2} up to relabeling of $X$ and $Y$, since either each of $u$~and~$v$ has at least one neighbour in each of $X$~and~$Y$, or otherwise one of them has no neighbours in~$X$ whereas the other one has no neighbours in~$Y$; in the latter case, $\{e,f\}$ is a $2$-cut of~$G$.

\begin{figure}[!htb]
    \centering
    \begin{tikzpicture}[scale=0.7]
\centering
\draw (-3.5,0) ellipse (2 and 2.8);
\node[draw=none] at (-3.5,-1.5) {$X$};
\draw (3.5,0) ellipse (2 and 2.8);
\node[draw=none] at (3.5,-1.5) {$Y_1$};
\draw (10.5,0) ellipse (2 and 2.8);
\node[draw=none] at (10.5,-1.5) {$Y_2$};
\node[circle,fill=white] (3) at (-2.5,0){};
\node[circle,fill=white] (4) at (2.5,0){};
\node[draw=none] at (0,-0.4) {$f$};
\draw (3) -- (4);
\node[circle,fill=white] (5) at (4.5,0){};
\node[circle,fill=white] (6) at (9.5,0){};
\node[draw=none] at (7,-0.4) {$f'$};
\draw (5) -- (6);
\end{tikzpicture}
    \caption{An illustration for the proof of Lemma~\ref{lem:associated-r-cuts}}
    \label{fig:3cc-proof}
\end{figure}

Finally, suppose to the contrary that $G$ is \ecc{3} and that $G-u-v$ has a $1$-cut $f'$ that is distinct from $f$. We may adjust notation so that $f'$ has both ends in $Y$, and is thus a $1$-cut of $G[Y]$. Let $Y_1$ and $Y_2$ denote the shores of $f'$ in $G[Y]$; adjust notation so that $f$ has one end in $Y_1$ as shown in Figure~\ref{fig:3cc-proof}. Since $G$ is \ecc{3}, a simple counting argument shows that $\partial_{G}(X\cup Y_1 \cup Y_2)=\partial_{G}(\{u,v\})$ has at least~$2r-1$ edges; however, this contradicts what we have already established earlier. Thus, $f$ is the unique $1$-cut of $G-u-v$. This completes the proof.
\end{proof}

In light of Lemma~\ref{lem:associated-r-cuts}, for a solitary edge $e:=uv$ of an $r$-graph~$G$, of order four or more, we refer to any odd $1$-cut $f$ of $G-u-v$ as a {\em companion of $e$}, and we refer to the laminar (odd) $r$-cuts $C$ and $D$, as defined in the statement of Lemma~\ref{lem:associated-r-cuts}, as the {\em $r$-cuts associated with the solitary-companion pair $(e,f)$}, or simply as the {\em $r$-cuts associated with $(e,f)$}; see Figure~\ref{fig:update-fig-with-two-possibilities-r-graphs}. Furthermore, if $G$ is \ecc{3},
then $f$ is its unique companion; consequently, the $r$-cuts $C$ and $D$ associated with $(e,f)$ are also unique, and we refer to them simply as the {\em $r$-cuts associated with $e$}.
These notions will play a crucial role throughout this paper.

For a subgraph $H$ of a graph~$G$, we simplify the notation $\partial(V(H))$ to $\partial(H)$. The following is a consequence of the preceding two results.
\begin{cor}\label{cor:solitary-leaving-triangle-implies-K4}
If a \cc{3} $G$ has a solitary edge $e:=uv$ that belongs to $\partial(T)$, where $T:=uu_1u_2u$ is a triangle, then $G$ is~$K_4$.
\end{cor}
\begin{proof}
Let $f$ denote the companion of $e$, and $C:=\partial(X)$~and~$D:=\partial(Y)$ denote the $3$-cuts associated with the solitary-companion pair~$(e,f)$ so that $X\cap Y=\emptyset$ and $C\cap D=\{f\}$; see Lemma~\ref{lem:associated-r-cuts}. Since $G$ is $3$-connected, exactly one of $u_1$~and~$u_2$ lies in $X$; we adjust notation so that $u_1\in X$~and~$u_2\in Y$. Since $u_1u_2\in E(G)$, and since $C\cap D=\{f\}$, we infer that $f=u_1u_2$. Consequently, neither $C$ nor $D$ is a matching. By Proposition~\ref{prp:3-cut-matching-laminar}, $C$~and~$D$ are trivial cuts and $G$ is $K_4$.
\end{proof}



In the following section, we shall use Lemma~\ref{lem:associated-r-cuts} and other results to establish an upper bound on the number of solitary classes in any \ecc{3} $r$-graph; see Corollary~\ref{cor:new-cor}. This, combined with Corollary~\ref{cor:lm-cor}, will prove Theorem~\ref{thm:main-thm-in-simple-words}.


\subsubsection{\texorpdfstring{$r$}{}-edge-colorability, rainbow triangles and bricks} \label{sec:r-edge-colorability}


We begin this section with an easy observation.

\begin{prp}
    Let $G$ be an $r$-regular graph that has a triangle~$T$. Then, $|\partial_{G}(T)|=r$ if and only if the induced subgraph~$G[V(T)]$ has precisely $r$ edges. \qed
\end{prp}

Now, let $G$ be an $r$-regular graph that is \mbox{$r$-edge-colorable}. Observe that if $G$ has a triangle~$T$ such that $G[V(T)]$ has precisely $r$ edges then, in any proper \mbox{$r$-edge-coloring}, the induced subgraph~$G[V(T)]$ witnesses each of the $r$ colors (precisely once), and the same holds for $\partial_{G}(T)$.

Inspired by the above observation, for an $r$-graph~$G$, a triangle~$T$ such that~$G[V(T)]$ has precisely $r$ edges is called a {\em rainbow triangle}, or simply an {\em $r$-triangle}. We now prove their existence in \ecc{3} $r$-graphs, of order four or more, that have a solitary edge.



\begin{lem}\label{lem:combined-stat-r-graphs}
Every $r$-graph $G$ of order four or more, that has a solitary edge, is \mbox{$r$-edge-colorable}. Furthermore, if $G$ is \ecc{3} then: (i) $G$ is bicritical, and (ii) for each $w\in V(G)$, there exists an $r$-triangle~$T$ that does not contain~$w$.
\end{lem}
\begin{proof}
We proceed by induction on the order of~$G$. 
Let $e:=uv$ denote a solitary edge and let $f$ denote any companion of $e$, and $C:=\partial(X)$ and $D:=\partial(Y)$ denote the odd $r$-cuts associated with~$(e,f)$ as defined in the statement of Lemma~\ref{lem:associated-r-cuts} so that $X\cap Y=\emptyset$. 
If both $C$~and~$D$ are trivial,
the desired conclusions are easily verified using Proposition~\ref{r-graphs-of-order-four}.
Now, suppose that at least one of them is nontrivial; adjust notation so that $C$ is nontrivial. 
By Proposition~\ref{prp:r-cuts}, both $G_1:=G/(X\rightarrow x)$ and $G_2:= G/(\overline{X}\rightarrow \overline{x})$ are $r$-graphs; clearly, each of them is of order four or more.

Observe that each edge in $\partial_{G_1}(x)-f$ is incident with one of $u$ and $v$; consequently, $e\xrightarrow{G_1} f$. Thus, if $M_1$ is any \pmg\ of~$G_1$ containing~$e$, and $M_2$ is any \pmg\ of~$G_2$ containing $f$, then $M_1 \cup M_2$ is a \pmg\ of~$G$. Since $e$ is solitary in~$G$, these facts imply that $e$ is solitary in~$G_1$ and that $f$ is solitary in~$G_2$.
Therefore, by the induction hypothesis, both $G_1$~and~$G_2$ satisfy the desired conclusions. In particular, both $G_1$~and~$G_2$ are \mbox{$r$-edge-colorable}. Since $G$ is obtained from $G_1$~and~$G_2$ by splicing, we conclude that $G$ is also \mbox{$r$-edge-colorable}.
This proves the first part.

Finally, suppose that $G$ is \ecc{3}. By Proposition~\ref{prp:r-cuts}, both $G_1$~and~$G_2$ are also \ecc{3}. By the induction hypothesis, both $G_1$~and~$G_2$ are bicritical, and $G_1$ has a triangle~$T_1$ such that \mbox{$x\notin V(T_1)$} and $|\partial_{G_1}(T_1)|=r$; likewise, $G_2$ has a triangle~$T_2$ such that $\overline{x} \notin V(T_2)$ and $|\partial_{G_2}(T_2)|=r$. By Proposition~\ref{prp:bicritical-proposition}, $G$ is bicritical. 
Observe that both $T_1$~and~$T_2$ are triangles in~$G$; furthermore, $T_1$ is the desired triangle for each $w\in X$, whereas $T_2$ is the desired triangle for each $w\in \overline{X}$. This completes the proof.
\end{proof}



We are now ready to state and prove two consequences that together comprise a strengthening of Proposition~\ref{prp:marked-c-components-solitary-edges} that we mentioned earlier.
Let $D$ denote any solitary class of an $r$-graph~$G$ that has a $2$-cut~$C$, and let $(M_1,M_2,\dots,M_r)$ be a proper $r$-edge-coloring so that $D\subseteq M_1$. By Proposition~\ref{prp:mcgs-hcycle}, $M_1\cup M_i$ is a hamiltonian cycle for each $i\in \{2,\dots,r\}$; consequently, $C\subseteq M_1$. This proves the following.


\begin{cor}\label{lem:2-cut-decomposition-r-graphs-new}
For each solitary class~$D$ of an $r$-graph~$G$, the unique \pmg\ that contains~$D$ also contains every $2$-cut of~$G$; consequently, if $G$ is not \ecc{3} then $G$ has at most one solitary class. \qed
\end{cor}

Propositions~\ref{prp:r-graph-mc},~\ref{prp:gluing}, \ref{prp:2-bond} and \ref{prp:marked-C-even-2-cut-components-in-r-graphs-are-also-r-graphs} imply the following.

\begin{cor}\label{cor:2-cut-solitary-stronger-version}
    Every $r$-graph~$G$ that has a $2$-cut~$C$ and (precisely) one solitary class~$D$ satisfies the following, where $G_1$~and~$G_2$ denote its marked $C$-components with marker edges $e_1$~and~$e_2$, respectively.
    \begin{enumerate}[(i)]
        \item $C\subseteq D$ if and only if $e_i$ is solitary in~$G_i$ for each $i\in \{1,2\}$, and
        \item for an edge $e\in E(G_1)-e_1$, the edge $e\in D$ if and only if $e$ is solitary in~$G_1$, $e\xrightarrow{G_1}e_1$, and $e_2$ is solitary in~$G_2$.\qed
    \end{enumerate} 
\end{cor}

The following is an immediate consequence of Lemma~\ref{lem:combined-stat-r-graphs} and Theorem~\ref{thm:elp-theorem}.

\begin{cor}\label{cor:3cc-solitary-brick}
    Every $3$-connected $r$-graph of order four or more, that has a solitary edge, is a brick. \qed
\end{cor}

Our next goal is to strengthen the above by weakening the hypothesis. To this end, we prove the following that is reminiscent of an earlier result, namely Lemma~\ref{lem:r-graphs-new-lemma}.

\begin{lem}\label{3cc-subgraph-in-r-graphs}
Let $G$ be a \ecc{3} $r$-graph that has a solitary edge~$e$, and let $(M_1, M_2, $ $ M_3, \dots, M_r)$ denote any proper $r$-edge-coloring wherein $M_1$ is the perfect matching containing~$e$. Then there exist distinct $i,j\in \{2,3,\dots,r\}$ such that the (spanning) subgraph \mbox{$H:= M_1\cup M_i\cup M_j$} is a \cc{3} (and $e$ is a solitary edge of~$H$).
\end{lem}
\begin{proof}
We proceed by induction on the order. If $|V(G)|=2$, the desired conclusion holds trivially. Now suppose that $|V(G)|\ge 4$, and let $u$ denote an end of~$e$. By the second part of Lemma~\ref{lem:combined-stat-r-graphs}, there exists a triangle~$T$ such that $u\notin V(T)$ and $|\partial_{G}(T)|=r$.
By Proposition~\ref{prp:r-cuts}, the smaller graph $G':=G/V(T)$ is also a \ecc{3} $r$-graph, and $\partial_{G}(T)$ is a separating cut of~$G$. By Lemma~\ref{lem:solitary-ME-sepcuts}~(i), $e$ is also solitary in~$G'$. Note that $(M_1', M_2',M_3', \dots, M_r')$ is a proper $r$-edge-coloring of~$G'$, where $M_i'$ is the restriction of $M_i$ to~$G'$. By the induction hypothesis, there exist two distinct
$i,j\in \{2,3,\dots,r\}$ such that the subgraph $H':= M_1' \cup M_i' \cup M_j'$ is a \cc{3}. Observe that $H:=M_1\cup M_i \cup M_j$ is obtained from $H'$ by splicing with~$K_4$; consequently, by Proposition~\ref{prp:splicing-with-k4}, $H$ is also a \cc{3}.
\end{proof}

We are now ready to deduce the aforementioned strengthening of Corollary~\ref{cor:3cc-solitary-brick}. Its statement~(i) is an immediate consequence of Corollary~\ref{cor:3cc-solitary-brick} and Lemma~\ref{3cc-subgraph-in-r-graphs}, whereas statement~(ii) follows from  Corollary~\ref{cor:brick-rd-near-bipartite} and the fact that bricks are \mbox{$3$-connected}; recall Theorem~\ref{thm:elp-theorem}.

\begin{cor}\label{cor:solitary-doubleton-near-bipartite-brick}
    Every \ecc{3} $r$-graph, of order four or more, satisfies the following:
    \begin{enumerate}[(i)]
        \item if $G$ has a solitary edge then $G$ is a brick, and
        \item if $G$ has a solitary doubleton then $G$ is a near-bipartite brick; in particular, deletion of any solitary doubleton~$\{\alpha, \beta\}$ results in a bipartite \mcg~$H[A,B]$ such that $\alpha$ has both ends in~$A$ whereas $\beta$ has both ends in~$B$.
        \qed
    \end{enumerate}
\end{cor}

In light of Theorem~\ref{conj:Seymour} and the discussion ensuing it, we rephrase the first part of Lemma~\ref{lem:combined-stat-r-graphs} as follows.
\begin{cor}\label{cor:solitary-implies-class-I}
    Every $r$-graph, that has a solitary edge, belongs to class I. \qed
\end{cor}

In light of Conjecture~\ref{conj:Seymour-planar}, we remark that an $r$-graph with a solitary edge need not be planar; see Figure~\ref{fig:nonplanar-graph}. We also mention a consequence of the second part of Lemma~\ref{lem:combined-stat-r-graphs} that may be viewed as a refinement of Proposition~\ref{prp:r-graph-mc}.

\begin{cor}\label{cor:simple_r-graph_matching_double_covered}
    Every simple \mbox{$3$-edge-connected} $r$-graph, where $r\ge 4$, is matching double covered. \qed
\end{cor}

The first part of Lemma~\ref{lem:combined-stat-r-graphs}, combined with Corollary~\ref{lem:r-reg-rec}, yields the following.
\begin{cor}\label{cor:r-graphs-has-at-most-r-solitary-classes}
Every $r$-graph has at most $r$ solitary classes. \qed
\end{cor}

It is worth noting that the above follows immediately if one assumes the famous Berge-Fulkerson Conjecture to be true; it states that every $r$-graph has $2r$ (not necessarily distinct) \pmg s so that each edge appears in precisely two of them. To see this implication, note that any two solitary classes are mutually exclusive. On a different note, the first part of Lemma~\ref{lem:combined-stat-r-graphs}, combined with Lemma~\ref{lem:r-graphs-new-lemma}, proves the following.

\begin{cor}\label{cor:r-graphs-3-connected}
Every $r$-graph, that has two or more solitary classes, is \mbox{$3$-connected}. \qed
\end{cor}

Combining the above with Corollary~\ref{lem:solitary-class-unique-pm} yields the following --- that proves Theorem~\ref{thm:at-most-n/2-solitary-edges} for $r$-graphs that are not $3$-edge-connected.
\begin{cor}\label{cor:2-connected-but-not-3-connected}
Every $r$-graph, that is not $3$-edge-connected, has at most $\frac{n}{2}$ \ses. Furthermore, equality holds if and only if the \ses\ comprise a \pmg. \qed
\end{cor}

We also note that the first part of Lemma~\ref{lem:combined-stat-r-graphs} proves Theorem~\ref{thm:main-thm-in-simple-words} for $r$-graphs that are not \mbox{$r$-edge-colorable}; furthermore, along with Corollaries~\ref{cor:lm-cor}~and~\ref{cor:r-graphs-has-at-most-r-solitary-classes}, it proves Theorem~\ref{thm:main-thm-in-simple-words} for \ecc{3} $3$-graphs. To prove Theorem~\ref{thm:main-thm-in-simple-words} for the remaining case --- that is, for \ecc{3} $r$-graphs, where $r\ge 4$, that are \mbox{$r$-edge-colorable} --- we first prove the following lemma.

\begin{lem}\label{lem:r-cut-lemma}
Let $G$ be an $r$-graph that has an $r$-triangle~$T$. Then, $G$ has at most $k$ solitary classes where $k$ denotes the number of edges of~$T$ whose multiplicity (in~$G$) is precisely one.
\end{lem}
\begin{proof}
If $G$ has no solitary classes then we are done. Now, suppose that $G$ has a solitary class. By the first part of Lemma~\ref{lem:combined-stat-r-graphs}, $G$ is \mbox{$r$-edge-colorable}; let $M_1, M_2, \dots, M_r$ denote a proper \mbox{$r$-edge-coloring}. By Corollary~\ref{lem:solitary-class-unique-pm}, each $M_i$ contains at most one solitary class.
Since $T$ is a triangle, any two edges in the induced subgraph~$T':=G[V(T)]$ are adjacent; consequently, $|M_i\cap E(T')|=1$ for each $M_i$. Finally, by Lemma~\ref{lem:associated-r-cuts}, if $e$ is any solitary edge then $\mu(e')=1$ for each edge~$e'$ in the unique perfect matching containing~$e$. Using these facts, we conclude that $G$ has at most $k$ solitary classes where $k$ denotes the number of edges of~$T$ whose multiplicity (in~$G$) is precisely one.
\end{proof}

We are now ready to state and prove the desired stronger versions of Theorem~\ref{thm:main-thm-in-simple-words}. We do this in the following section.

\subsection{Proof of Theorem~\ref{thm:main-thm-in-simple-words}: solitude patterns of \texorpdfstring{\ecc{3}}{} \texorpdfstring{$r$}{}-graphs}\label{sec:LM}

We first invite the reader to observe that the next result is an immediate consequence of Lemmas \ref{lem:combined-stat-r-graphs}~and~\ref{lem:r-cut-lemma}.

\begin{cor}\label{cor:new-cor}
Every \ecc{3} $r$-graph~$G$, of order four or more, satisfies the following:
\begin{enumerate}[(i)]
\item if $r=3$ then $G$ has at most three solitary classes, and
\item if $r\ge 4$ then $G$ has at most two solitary classes. \qed
\end{enumerate}
\end{cor}

Observe that the above, combined with Corollary~\ref{cor:lm-cor}, proves Theorem~\ref{thm:main-thm-in-simple-words}. The following is simply an amalgamation of  Lemma~\ref{lem:combined-stat-r-graphs} and Corollaries~\ref{cor:lm-cor}~and~\ref{cor:new-cor}. 

\begin{cor}\label{cor:main-thm}
Every \ecc{3} $r$-graph~$G$, of order four or more, satisfies the following:
\begin{enumerate}[(i)]
\item if $G$ has a solitary edge then $G$ is \mbox{$r$-edge-colorable},
\item the cardinality of each solitary class is either one or two,
\item if $r=3$ then $G$ has at most three solitary classes, whence at most six solitary edges,
\item and if $r\ge 4$ then $G$ has at most two solitary classes, whence at most four solitary edges. \qed
\end{enumerate}
\end{cor}

Recall from Section~\ref{sec:CLM's-DR} that the solitude pattern of a \mcg\ refers to the (possibly empty) sequence of cardinalities of its solitary classes in nonincreasing order. Now, the following is simply a restatement of Corollary~\ref{cor:main-thm}~(ii), (iii) and (iv).

\begin{cor}\label{lem:solitary-patterns-in-r-graphs}
Every $3$-edge-connected $r$-graph~$G$, of order four or more, satisfies the following:
\begin{enumerate}[(i)]
\item if $r=3$, then $G$ has one of the following ten solitude patterns:  \newline $(2,2,2)$, $(2,2,1)$, $(2,1,1)$, $(1,1,1)$,  $(2,2)$, $(2,1)$, $(1,1)$, $(2)$, $(1)$~or~$()$,
\item whereas if $r\ge 4$, then $G$ has one of the following six solitude patterns: \newline $(2,2)$, $(2,1)$, $(1,1)$, $(2)$, $(1)$~or~$()$. \qed
\end{enumerate}
\end{cor}

We have thus stated and proved our desired stronger versions of Theorem~\ref{thm:main-thm-in-simple-words}. 
In this paper, we provide complete characterizations of \ecc{3} $r$-graphs for six of the ten solitude patterns that are shown in (the first six rows of) Table~\ref{table1}, and of \ecc{3} $3$-graphs for the solitude pattern~$(2)$; these characterizations and their proofs appear in the corresponding sections as shown in the table. Finally, in Section~\ref{sec:largest-EC-n/2}, we establish a characterization of all \mcg s whose largest equivalence class has cardinality~$\frac{n}{2}$, and as a consequence, we obtain a characterization of \mbox{$r$-graphs} with the same property; by Corollary~\ref{cor:at-least-n/2-solitary-edges}, any such equivalence class is solitary.

\begin{table}[!htb]
    \centering
    \makebox[\textwidth][c]{
    \begin{tabular}{|c|c|c|c|}
    \toprule
     \textbf{Solitude pattern} & \textbf{Characterization} & \textbf{Results}& \textbf{Sections}\\
    \midrule
     $(2,2,2)$ & $K_4, \overline{C_6}$ & &\\[2.5pt]
     $(2,2,1)$ & ${R}_{8}$ & {Theorem~\ref{thm:more-than-one-SD}} &{\ref{sec:more-than-one-SD}~and~\ref{sec:more than one solitary doubleton}} \\[2.5pt]
     $(2,2)$ & staircases of order ten or more & & \\[2.5pt]
    \midrule
     $(2,1,1)$ & ${N}_{10}$ & \multirow{2}{*}{Theorem~\ref{thm:SDSS}} &\multirow{2}{*}{\ref{sec:SDSS}~and~\ref{sec:oneSD-oneSS}} \\[2.5pt]
     $(2,1)$ & $3$-staircases of order twelve or more & & \\[2.5pt]
    \midrule
    
     $(1,1,1)$ & family $\mathcal{S}$ & Theorem~\ref{thm:3SS} & \ref{sec:5-special-graphs}~and~\ref{sec:3SS} \\[2.5pt]
     \midrule
	$(2)$\footnote{only for $r=3$} & family $\mathcal{D}$ & Theorem~\ref{thm:SP2} & \ref{sec:SP2} \\[2.5pt]
	
    \bottomrule
    \end{tabular}}
     \caption{Solitude patterns and their characterizations}
\label{table1}
\end{table}

It is worth noting that Goedgebeur, Mazzuoccolo, Renders, Wolf and the second author
\cite{gmmrw24} have obtained characterizations of \ecc{3} $3$-graphs for the six solitude patterns --- all of those shown in Table~\ref{table1} except solitude pattern~$(2)$. As mentioned earlier, their descriptions and approach are inspired by Theorem~\ref{thm:Klee}, whereas ours are inspired by the theory of \mcg s that we have discussed thus far and that we develop further in Section~\ref{sec:further-developing-toolkit}.

The results pertaining to $3$-graphs, that appear in this paper, are included in the unpublished manuscript available here: \cite{dmkg24}. 
\section{Characterizations and easier implications: \newline \texorpdfstring{\ecc{3}}{} \texorpdfstring{$r$}{}-graphs with three or more solitary edges}\label{sec:characterizations}

In this section, we state our main results (Theorems~\ref{thm:more-than-one-SD}, \ref{thm:SDSS}~and~\ref{thm:3SS}) pertaining to the characterizations of \ecc{3} $r$-graphs that have at least three solitary edges, and we also verify their easier implications. Furthermore, in Section~\ref{sec:proof-of-Theorem-at-most-n/2-solitary-edges}, we present a proof of Theorem~\ref{thm:at-most-n/2-solitary-edges} assuming Theorems~\ref{thm:more-than-one-SD}~and~\ref{thm:SDSS}.

We begin by defining a couple of graph families that feature as subgraphs of some classes we intend to characterize. By {\em $k$-path} or {\em $k$-cycle}, we mean a path or a cycle with~$k$ edges, respectively.
A {\em ladder} is a graph that is obtained from two disjoint copies of a path, say $P$~and~$P'$, by adding edges called {\em rungs} --- each of which joins a vertex of~$P$ with its copy
in~$P'$.
The smallest ladders are $K_2$ and $C_4$; in the latter case, we may consider any \pmg\ to be its set of rungs. A rung of a ladder is {\em peripheral} if any (thus each) of its ends has degree at most two. 
The following pertinent fact is easily observed.

\begin{prp}\label{prp:ladder-is-bipartite-mcg}
Each ladder is a bipartite \mcg; furthermore, it is matching double covered if and only if its order is ten or more. \qed
\end{prp}


A {\em $k$-\dumbbell}, where $k$ is an odd positive integer, is a graph $J$ obtained from the disjoint union of two ladders, say $L_1$~and~$L_2$, and a $k$-path $P$ with ends $v_1$~and~$v_2$, by adding four edges as follows: for $i\in \{1,2\}$, the vertex~$v_i$ is joined with both ends of a peripheral rung of~$L_i$. We refer to rungs of $L_1$~and~$L_2$ as {\em rungs of~$J$}, the path~$P$ as the {\em bone of~$J$}, and the two ends of~$P$ as the {\em sockets of~$J$}.
A rung of a $k$-\dumbbell\ is {\em peripheral} if any (thus each) of its ends has degree exactly two. 
Note that every $k$-\dumbbell~$J$ has a unique perfect matching~$M$ which includes all of its rungs; we refer to the graph $J':=J\oplus (t-1)M$, where $t\in \mathbb{Z}^+$, as a {\em $k$-\dumbbell\ of thickness~$t$}, and the sockets of~$J$ as the {\em sockets of~$J'$}. Note that the sockets of~$J'$ are its only cut-vertices that belong to triangles.
We remark that a $k$-\dumbbell\ is simply a $k$-\dumbbell\ of thickness one.

In particular, $1$-\dumbbell s and $3$-\dumbbell s, of thickness $r-2$ will play an important role in  our characterizations of \ecc{3} $r$-graphs that have a solitary doubleton and at least three solitary edges. The following section deals with those members that have at least two solitary doubletons.



\subsection{Solitude patterns \texorpdfstring{$(2,2,2),(2,2,1)$}{}~and~\texorpdfstring{$(2,2)$}{} --- \texorpdfstring{$1$}{}-staircases}\label{sec:more-than-one-SD}

We let $J$ denote a $1$-\dumbbell\ of thickness~$t\in \mathbb{Z}^+$. The unique $2$-connected $(t+2)$-regular graph~$G$ obtained from~$J$ by adding two edges is called a {\em $1$-staircase of thickness~$t$}, or simply a {\em staircase of thickness~$t$}; see Figures~\ref{fig:1-staircases}~and~\ref{fig:1-staircases-thickness-3}. It is easy to see that $G$ is $3$-connected. By a {\em staircase}, we mean a staircase of thickness one. The following fact is easily verified, and we shall find it useful in Section~\ref{sec:distance-between-solitary-edges}.

\begin{prp}\label{prp:staircase-uniquely-determined}
    If $Q$ is an even cycle of order $n\ge 6$, and $u$ and $v$ are vertices at distance two, then there is a unique supergraph of~$Q+uv$ that is isomorphic to the staircase of order~$n$.
\end{prp}

Staircases play an important role in the study of \mcg s; see~\cite{noth07,lumu24}. We now observe the symmetries of the two smallest staircases.

\begin{prp}\label{prp:R8-vertex-orbits}
The triangular prism $\overline{C_6}$ is vertex-transitive, whereas the automorphism group of the bicorn~$R_8$ has three vertex-orbits --- $\{0,3\},\{1,2,4,5\}$ and $\{6,7\}$ as per the labeling shown in Figure~\ref{fig:R8}. \qed
\end{prp}

\begin{figure}[!htb]
        \centering
        \begin{subfigure}[b]{.3 \textwidth}
        \centering
        \begin{tikzpicture}[scale=0.68]
		\node [circle,fill=white] (0) at (0, 1) {};
		\node [circle,fill=white] (1) at (0, 0) {};
		\node [circle,fill=white] (2) at (1, 1) {};
		\node [circle,fill=white] (3) at (1, 0) {};
		\node [circle,fill=white] (4) at (2, 1) {};
		\node [circle,fill=white] (5) at (2, 0) {};
		\node [circle,fill=white] (6) at (3, 0.5) {};
		\node [circle,fill=white] (7) at (4, 0.5) {};
		\node [circle,fill=white] (8) at (5, 1) {};
		\node [circle,fill=white] (9) at (5, 0) {};
		\node [circle,fill=white] (10) at (6, 1) {};
		\node [circle,fill=white] (11) at (6, 0) {};
		
		\draw (0) -- (2);
		\draw (2) -- (4);
		\draw (4) -- (6);
		\draw (6) -- (5);
		\draw (5) -- (3);
		\draw (3) -- (1);
		\draw (1) -- (0);
		\draw (2) -- (3);
		\draw (4) -- (5);
		\draw (6) -- (7);
		\draw (7) -- (8);
		\draw (8) -- (9);
		\draw (9) -- (7);
		\draw (8) -- (10);
		\draw (10) -- (11);
		\draw (9) -- (11);
        \end{tikzpicture}
        \vspace{21pt}
        \caption{A $1$-dumbbell $J$}
\label{fig:1-dumbbell}
    \end{subfigure}
    \begin{subfigure}[b]{.33 \textwidth}
    \centering
    \begin{tikzpicture}[scale=0.68]
      \node [circle,fill=white] (0) at (0, 1) {};
		\node [circle,fill=white] (1) at (0, 0) {};
		\node [circle,fill=white] (2) at (1, 1) {};
		\node [circle,fill=white] (3) at (1, 0) {};
		\node [circle,fill=white] (4) at (2, 1) {};
		\node [circle,fill=white] (5) at (2, 0) {};
		\node [circle,fill=white] (6) at (3, 0.5) {};
		\node [circle,fill=white] (7) at (4, 0.5) {};
		\node [circle,fill=white] (8) at (5, 1) {};
		\node [circle,fill=white] (9) at (5, 0) {};
		\node [circle,fill=white] (10) at (6, 1) {};
		\node [circle,fill=white] (11) at (6, 0) {};
		
		\draw (0) -- (2);
		\draw (2) -- (4);
		\draw (4) -- (6);
		\draw (6) -- (5);
		\draw (5) -- (3);
		\draw (3) -- (1);
		\draw (1) -- (0);
		\draw (2) -- (3);
		\draw (4) -- (5);
		\draw (6) -- (7);
		\draw (7) -- (8);
		\draw (8) -- (9);
		\draw (9) -- (7);
		\draw (8) -- (10);
		\draw (10) -- (11);
		\draw (9) -- (11);
		\draw [bend left] (0) to (10);
		\draw [bend right] (1) to (11);
    \end{tikzpicture}
    \caption{Corresponding staircase $G$}
\label{fig:1-staircase}
    \end{subfigure}
    \begin{subfigure}[b]{.33 \textwidth}
    \centering
    \begin{tikzpicture}[scale=0.68]
    \node [circle,fill=white] (0) at (0, 0.5) {};
		\node [circle,fill=white] (1) at (1, 1) {};
		\node [circle,fill=white] (2) at (1, 0) {};
		\node [circle,fill=white] (3) at (2, 1) {};
		\node [circle,fill=white] (4) at (2, 0) {};
		\node [circle,fill=white] (5) at (3, 1) {};
		\node [circle,fill=white] (6) at (3, 0) {};
		\node [circle,fill=white] (7) at (4, 1) {};
		\node [circle,fill=white] (8) at (4, 0) {};
		\node [circle,fill=white] (9) at (5, 1) {};
		\node [circle,fill=white] (10) at (5, 0) {};
		\node [circle,fill=white] (11) at (6, 0.5) {};
		
		\draw (0) -- (1);
		\draw (1) -- (3);
		\draw (3) -- (5);
		\draw (5) -- (7);
		\draw (7) -- (9);
		\draw (9) -- (11);
		\draw (11) -- (10);
		\draw (10) -- (8);
		\draw (8) -- (6);
		\draw (6) -- (4);
		\draw (4) -- (2);
		\draw (0) -- (2);
		\draw (1) -- (2);
		\draw (3) -- (4);
		\draw (5) -- (6);
		\draw (7) -- (8);
		\draw (9) -- (10);
		\draw [bend left=60] (0) to (11);
    \end{tikzpicture}
    \vspace{22pt}
    \caption{Another drawing of $G$}
\label{fig:staircase}
    \end{subfigure}
\caption{$1$-dumbbells and $1$-staircases of thickness one}
\label{fig:1-staircases}
\end{figure}
\begin{figure}[!htb]
        \centering
        \begin{subfigure}[b]{.3 \textwidth}
        \centering
        \begin{tikzpicture}[scale=0.68]
		\node [circle,fill=white] (0) at (0, 1) {};
		\node [circle,fill=white] (1) at (0, 0) {};
		\node [circle,fill=white] (2) at (1, 1) {};
		\node [circle,fill=white] (3) at (1, 0) {};
		\node [circle,fill=white] (4) at (2, 1) {};
		\node [circle,fill=white] (5) at (2, 0) {};
		\node [circle,fill=white] (6) at (3, 0.5) {};
		\node [circle,fill=white] (7) at (4, 0.5) {};
		\node [circle,fill=white] (8) at (5, 1) {};
		\node [circle,fill=white] (9) at (5, 0) {};
		\node [circle,fill=white] (10) at (6, 1) {};
		\node [circle,fill=white] (11) at (6, 0) {};

        \draw[bend left] (1) to (0);
		\draw[bend right] (1) to (0);
        \draw[bend left] (2) to (3);
		\draw[bend right] (2) to (3);
        \draw[bend left] (4) to (5);
		\draw[bend right] (4) to (5);
        \draw[bend left] (6) to (7);
		\draw[bend right] (6) to (7);
        \draw[bend left] (8) to (9);
		\draw[bend right] (8) to (9);
        \draw[bend left] (10) to (11);
		\draw[bend right] (10) to (11);
        
		\draw (0) -- (2);
		\draw (2) -- (4);
		\draw (4) -- (6);
		\draw (6) -- (5);
		\draw (5) -- (3);
		\draw (3) -- (1);
		\draw (1) -- (0);
		\draw (2) -- (3);
		\draw (4) -- (5);
		\draw (6) -- (7);
		\draw (7) -- (8);
		\draw (8) -- (9);
		\draw (9) -- (7);
		\draw (8) -- (10);
		\draw (10) -- (11);
		\draw (9) -- (11);
        \end{tikzpicture}
        \vspace{7pt}
        \caption{A $1$-dumbbell $J$}
\label{fig:1-dumbbell-thickness-3}
    \end{subfigure}
    \begin{subfigure}[b]{.33 \textwidth}
    \centering
    \begin{tikzpicture}[scale=0.68]
      \node [circle,fill=white] (0) at (0, 1) {};
		\node [circle,fill=white] (1) at (0, 0) {};
		\node [circle,fill=white] (2) at (1, 1) {};
		\node [circle,fill=white] (3) at (1, 0) {};
		\node [circle,fill=white] (4) at (2, 1) {};
		\node [circle,fill=white] (5) at (2, 0) {};
		\node [circle,fill=white] (6) at (3, 0.5) {};
		\node [circle,fill=white] (7) at (4, 0.5) {};
		\node [circle,fill=white] (8) at (5, 1) {};
		\node [circle,fill=white] (9) at (5, 0) {};
		\node [circle,fill=white] (10) at (6, 1) {};
		\node [circle,fill=white] (11) at (6, 0) {};

        \draw[bend left] (1) to (0);
		\draw[bend right] (1) to (0);
        \draw[bend left] (2) to (3);
		\draw[bend right] (2) to (3);
        \draw[bend left] (4) to (5);
		\draw[bend right] (4) to (5);
        \draw[bend left] (6) to (7);
		\draw[bend right] (6) to (7);
        \draw[bend left] (8) to (9);
		\draw[bend right] (8) to (9);
        \draw[bend left] (10) to (11);
		\draw[bend right] (10) to (11);
		
		\draw (0) -- (2);
		\draw (2) -- (4);
		\draw (4) -- (6);
		\draw (6) -- (5);
		\draw (5) -- (3);
		\draw (3) -- (1);
		\draw (1) -- (0);
		\draw (2) -- (3);
		\draw (4) -- (5);
		\draw (6) -- (7);
		\draw (7) -- (8);
		\draw (8) -- (9);
		\draw (9) -- (7);
		\draw (8) -- (10);
		\draw (10) -- (11);
		\draw (9) -- (11);
		\draw [bend left] (0) to (10);
		\draw [bend right] (1) to (11);
    \end{tikzpicture}
    \caption{Corresponding staircase $G$}
\label{fig:1-staircase-thickness-3}
    \end{subfigure}
    \begin{subfigure}[b]{.33 \textwidth}
    \centering
    \begin{tikzpicture}[scale=0.68]
    \node [circle,fill=white] (0) at (0, 0.5) {};
		\node [circle,fill=white] (1) at (1, 1) {};
		\node [circle,fill=white] (2) at (1, 0) {};
		\node [circle,fill=white] (3) at (2, 1) {};
		\node [circle,fill=white] (4) at (2, 0) {};
		\node [circle,fill=white] (5) at (3, 1) {};
		\node [circle,fill=white] (6) at (3, 0) {};
		\node [circle,fill=white] (7) at (4, 1) {};
		\node [circle,fill=white] (8) at (4, 0) {};
		\node [circle,fill=white] (9) at (5, 1) {};
		\node [circle,fill=white] (10) at (5, 0) {};
		\node [circle,fill=white] (11) at (6, 0.5) {};

        \draw[bend left] (1) to (2);
		\draw[bend right] (1) to (2);
        \draw[bend left] (4) to (3);
		\draw[bend right] (4) to (3);
        \draw[bend left] (6) to (5);
		\draw[bend right] (6) to (5);
        \draw[bend left] (8) to (7);
		\draw[bend right] (8) to (7);
        \draw[bend left] (10) to (9);
		\draw[bend right] (10) to (9);
        \draw[bend left=70] (0) to (11);
		\draw[bend left=50] (0) to (11);
		
		\draw (0) -- (1);
		\draw (1) -- (3);
		\draw (3) -- (5);
		\draw (5) -- (7);
		\draw (7) -- (9);
		\draw (9) -- (11);
		\draw (11) -- (10);
		\draw (10) -- (8);
		\draw (8) -- (6);
		\draw (6) -- (4);
		\draw (4) -- (2);
		\draw (0) -- (2);
		\draw (1) -- (2);
		\draw (3) -- (4);
		\draw (5) -- (6);
		\draw (7) -- (8);
		\draw (9) -- (10);
		\draw [bend left=60] (0) to (11);
    \end{tikzpicture}
    \vspace{22pt}
    \caption{Another drawing of $G$}
\label{fig:staircase-thickness-3}
    \end{subfigure}
\caption{$1$-dumbbells and $1$-staircases of thickness three}
\label{fig:1-staircases-thickness-3}
\end{figure}

 Using the drawings shown in Figures~\ref{fig:staircase}~and~\ref{fig:staircase-thickness-3}, it is easily verified that there is a unique staircase of thickness~$t\in \mathbb{Z}^+$ for each even order $n\ge 6$. The following is our characterization of \ecc{3} $r$-graphs that have at least two solitary doubletons.

\begin{thm}\label{thm:more-than-one-SD}
A \ecc{3} $r$-graph~$G$ has more than one solitary doubleton if and only if either $G$ is a member of~\mbb{K}{4}{1} or a staircase of thickness~$r-2$. Furthermore, precisely one of the following holds:
\begin{enumerate}[(i)]
\item either $G\in \{K_4,\overline{C_6}\}$ and has solitude pattern $(2,2,2)$, or
\item $G$ is the bicorn~$R_8$ and has solitude pattern $(2,2,1)$, or otherwise
\item $G$ has solitude pattern $(2,2)$.
\end{enumerate}
\end{thm}


A proof of the forward implication of the above theorem appears in Section~\ref{sec:more than one solitary doubleton}. In the rest of this section, we shall convince the reader of the reverse implication along with statements~(i),~(ii)~and~(iii). To this end, we shall find the following easy observation useful.

\begin{prp}\label{prp:matchable-spanning-subgraph}
Let $H$ be a spanning subgraph of a graph~$G$. If an edge $e$ of~$H$ is solitary in~$G$, then $e$ is either unmatchable or otherwise solitary in~$H$. \qed
\end{prp}

In light of the above, in order to locate the solitary edges of a $1$-staircase, one may first consider locating the solitary as well as unmatchable edges of a $1$-dumbbell. We do this below.

Let $J$ denote a $1$-dumbbell with $uv$ as its bone. Note that $uv$ is its unique odd \mbox{$1$-cut}, say $\partial_{J}(X)$, where $u\in X$. Since $uv$ participates in each \pmg\ of~$J$, each of the four edges adjacent to~$uv$ is unmatchable. Furthermore, all of the remaining edges are matchable (as each ladder is matching covered). Note that, if the ladder~$J[\overline{X}-v]$ is distinct from~$K_2$, then each matchable edge in $J[X+ v]$ 
participates in at least two \pms\ of~$J$; an analogous statement holds for $J[X-u]$ and $J[\overline{X}+ u]$. These observations, along with Proposition~\ref{prp:ladder-is-bipartite-mcg}, imply that if $J$ has a solitary edge then one of the ladders $J[X-u]$ and $J[\overline{X}-v]$ is isomorphic to $K_2$, whereas the other one is of order at most eight. This proves the following except for the easy task of locating the solitary edges in the graphs shown in Figure~\ref{fig:1-dumbbells}. 

\begin{prp}\label{prp:solitary-unmatchable-in-1-dumbbell}
Every $1$-dumbbell $J$ of thickness~$t\in \mathbb{Z}^+$ has exactly four unmatchable edges, each of which is incident with a socket and belongs to a triangle. Also, $J$~is free of solitary edges unless $J$ is one of the four graphs shown in Figure~\ref{fig:1-dumbbells} --- wherein the solitary edges are indicated in red. \qed
\end{prp}

\begin{figure}[!htb]
    \centering
    \begin{subfigure}[b]{.2\textwidth}
        \centering
        \begin{tikzpicture}[scale=0.6]
           \node [circle,fill=white] (0) at (0, 1) {};
		\node [circle,fill=white] (1) at (0, 0) {};
		\node [circle,fill=white] (2) at (1, 0.5) {};
		\node [circle,fill=white] (3) at (2, 0.5) {};
		\node [circle,fill=white] (4) at (3, 1) {};
		\node [circle,fill=white] (5) at (3, 0) {};
		\draw[thick,red] (0) -- (1);
		\draw (0) -- (2);
		\draw (2) -- (1);
		\draw[thick,red] (2) -- (3);
		\draw (3) -- (4);
		\draw[thick,red] (4) -- (5);
		\draw (5) -- (3);
        \end{tikzpicture}
        \caption{}
  \label{fig:1-dumbbells-a}
  
    \end{subfigure}
    \begin{subfigure}[b]{.2\textwidth}
        \centering
        \begin{tikzpicture}[scale=0.6]
            \node [circle,fill=white] (0) at (0, 1) {};
		\node [circle,fill=white] (1) at (0, 0) {};
		\node [circle,fill=white] (2) at (1, 0.5) {};
		\node [circle,fill=white] (3) at (2, 0.5) {};
		\node [circle,fill=white] (4) at (3, 1) {};
		\node [circle,fill=white] (5) at (3, 0) {};
		\node [circle,fill=white] (6) at (4, 1) {};
		\node [circle,fill=white] (7) at (4, 0) {};
		\draw (0) -- (1);
		\draw (0) -- (2);
		\draw (2) -- (1);
		\draw (2) -- (3);
		\draw (3) -- (4);
		\draw[thick,red] (4) -- (5);
		\draw (5) -- (3);
		\draw[thick,red] (4) -- (6);
		\draw[thick,red] (6) -- (7);
		\draw[thick,red] (7) -- (5);
        \end{tikzpicture}
        \caption{}
        \label{fig:1-dumbbells-b}
        
    \end{subfigure}
    \begin{subfigure}[b]{.29\textwidth}
        \centering
        \begin{tikzpicture}[scale=0.6]
            \node [circle,fill=white] (0) at (0, 1) {};
		\node [circle,fill=white] (1) at (0, 0) {};
		\node [circle,fill=white] (2) at (1, 0.5) {};
		\node [circle,fill=white] (3) at (2, 0.5) {};
		\node [circle,fill=white] (4) at (3, 1) {};
		\node [circle,fill=white] (5) at (3, 0) {};
		\node [circle,fill=white] (6) at (4, 1) {};
		\node [circle,fill=white] (7) at (4, 0) {};
		\node [circle,fill=white] (8) at (5, 1) {};
		\node [circle,fill=white] (9) at (5, 0) {};
		\draw (0) -- (1);
		\draw (0) -- (2);
		\draw (2) -- (1);
		\draw (2) -- (3);
		\draw (3) -- (4);
		\draw (4) -- (5);
		\draw (5) -- (3);
		\draw[thick,red] (4) -- (6);
		\draw[thick,red] (6) -- (7);
		\draw[thick,red] (7) -- (5);
		\draw[thick,red] (6) -- (8);
		\draw (8) -- (9);
		\draw[thick,red] (9) -- (7);
        \end{tikzpicture}
        \caption{}
        \label{fig:1-dumbbells-c}
        
    \end{subfigure}
    \begin{subfigure}[b]{.25\textwidth}
        \centering
        \begin{tikzpicture}[scale=0.6]
                     \node [circle,fill=white] (0) at (0, 1) {};
		\node [circle,fill=white] (1) at (0, 0) {};
		\node [circle,fill=white] (2) at (1, 0.5) {};
		\node [circle,fill=white] (3) at (2, 0.5) {};
		\node [circle,fill=white] (4) at (3, 1) {};
		\node [circle,fill=white] (5) at (3, 0) {};
		\node [circle,fill=white] (6) at (4, 1) {};
		\node [circle,fill=white] (7) at (4, 0) {};
		\node [circle,fill=white] (8) at (5, 1) {};
		\node [circle,fill=white] (9) at (5, 0) {};
		\node [circle,fill=white] (10) at (6, 1) {};
		\node [circle,fill=white] (11) at (6, 0) {};
		\draw (0) -- (1);
		\draw (0) -- (2);
		\draw (2) -- (1);
		\draw (2) -- (3);
		\draw (3) -- (4);
		\draw (4) -- (5);
		\draw (5) -- (3);
		\draw (4) -- (6);
		\draw (6) -- (7);
		\draw (7) -- (5);
		\draw[thick,red] (6) -- (8);
		\draw (8) -- (9);
		\draw[thick,red] (9) -- (7);
		\draw (8) -- (10);
		\draw (10) -- (11);
		\draw (11) -- (9);
        \end{tikzpicture}
		 \caption{}
  \label{fig:1-dumbbells-d}
    \end{subfigure}
    \caption{Illustration for Proposition~\ref{prp:solitary-unmatchable-in-1-dumbbell}}
    \label{fig:1-dumbbells}
\end{figure}

Now, let $G$ be a staircase of thickness~$t\in \mathbb{Z}^+$ that is obtained from a $1$-dumbbell~$J$ by adding two edges $e$~and~$f$. Observe that the matchable graph~$H$ obtained from $G$ by deleting both ends of~$e$ (or of $f$) is $2$-connected; consequently, by Kotzig's Lemma~(\ref{lem:unique-pm-in-a-graph}), $H$ has at least two \pms, and thus neither $e$ nor $f$ is solitary in~$G$. It follows that each solitary edge of~$G$ thus belongs to~$J$. In particular, if $J$ is not one of the four graphs shown in Figure~\ref{fig:1-dumbbells} then, by Propositions~\ref{prp:solitary-unmatchable-in-1-dumbbell}~and~\ref{prp:matchable-spanning-subgraph}, $G$~has at most four solitary edges, and these four are precisely the unmatchable edges of~$J$.

Using the drawings of $1$-staircases, shown in the Figures~\ref{fig:staircase}~and~\ref{fig:staircase-thickness-3}, it is easy to see that each of them has at least four solitary edges. These observations prove the reverse implication of Theorem~\ref{thm:more-than-one-SD} as well as statements~(i), (ii)~and~(iii).


The next section deals with those \ecc{3} $r$-graphs that have at least one solitary class of each cardinality --- that is, doubleton and singleton.

\subsection{Solitude patterns \texorpdfstring{$(2,2,1),(2,1,1)$}{}~and~\texorpdfstring{$(2,1)$}{} --- \texorpdfstring{$3$}{}-staircases}\label{sec:SDSS}
We let $J$ denote a $3$-\dumbbell\ of thickness~$t\in \mathbb{Z}^+$. The unique $3$-connected $(t+2)$-regular graph~$G$ obtained from~$J$, by adding three edges, is called a {\em $3$-staircase of thickness~$t$}; see Figures~\ref{fig:3-staircases}~and~\ref{fig:3-staircases-thickness-3}. 
By a {\em $3$-staircase}, we mean a $3$-staircase of thickness one. Using the definition, and the observation that the unique even $1$-cut of~$J$ is a removable edge of~$G$, we deduce the following.



\begin{prp}\label{prp:3-staircase-ss}
Let $G$ be the $3$-staircase of thickness~$t\in \mathbb{Z}^+$, obtained from a \mbox{$3$-dumbbell~$J$}, by adding the edges $d,d'$~and~$f$. Then, two of these three edges, say $d$~and~$d'$, are incident with distinct ends of the unique even $1$-cut of~$J$, say $e:=vw$; whereas the third edge~$f$ is the unique odd $1$-cut of~$G-v-w$. Furthermore, $e$ is a solitary singleton in~$G$, and $f$ is its companion. \qed
\end{prp}

\begin{figure}[!htb]
        \centering
        \begin{subfigure}[b]{.3\textwidth}
        \centering
        \begin{tikzpicture}[scale=0.5]
		\node [circle,fill=white] (0) at (0, 1) {};
		\node [circle,fill=white] (1) at (0, 0) {};
		\node [circle,fill=white] (2) at (1, 1) {};
		\node [circle,fill=white] (3) at (1, 0) {};
		\node [circle,fill=white] (4) at (2, 1) {};
		\node [circle,fill=white] (5) at (2, 0) {};
		\node [circle,fill=white] (6) at (3, 0.5) {};
		\node [circle,fill=white] (7) at (4, 0.5) {};
		\node [circle,fill=white] (8) at (5, 0.5) {};
		\node [circle,fill=white] (9) at (6, 0.5) {};
		\node [circle,fill=white] (10) at (7, 1) {};
		\node [circle,fill=white] (11) at (7, 0) {};
		\node [circle,fill=white] (12) at (8, 1) {};
		\node [circle,fill=white] (13) at (8, 0) {};
		
		\draw (0) -- (1);
		\draw (1) -- (3);
		\draw (3) -- (5);
		\draw (5) -- (4);
		\draw (0) -- (2);
		\draw (2) -- (4);
		\draw (2) -- (3);
		\draw (4) -- (6);
		\draw (6) -- (5);
		\draw (6) -- (7);
		\draw (7) -- (8);
		\draw (8) -- (9);
		\draw (9) -- (10);
		\draw (9) -- (11);
		\draw (11) -- (10);
		\draw (10) -- (12);
		\draw (12) -- (13);
		\draw (11) -- (13);
        \end{tikzpicture}
        \vspace{12pt}
        \caption{A $3$-dumbbell $J$}
\label{fig:3-dumbbell}
    \end{subfigure}
    \begin{subfigure}[b]{.33 \textwidth}
    \centering
    \begin{tikzpicture}[scale=0.5]
      \node [circle,fill=white] (0) at (0, 1) {};
		\node [circle,fill=white] (1) at (0, 0) {};
		\node [circle,fill=white] (2) at (1, 1) {};
		\node [circle,fill=white] (3) at (1, 0) {};
		\node [circle,fill=white] (4) at (2, 1) {};
		\node [circle,fill=white] (5) at (2, 0) {};
		\node [circle,fill=white] (6) at (3, 0.5) {};
		\node [circle,fill=white] (7) at (4, 0.5) {};
		\node [circle,fill=white] (8) at (5, 0.5) {};
		\node [circle,fill=white] (9) at (6, 0.5) {};
		\node [circle,fill=white] (10) at (7, 1) {};
		\node [circle,fill=white] (11) at (7, 0) {};
		\node [circle,fill=white] (12) at (8, 1) {};
		\node [circle,fill=white] (13) at (8, 0) {};
		
		\draw (0) -- (1);
		\draw (1) -- (3);
		\draw (3) -- (5);
		\draw (5) -- (4);
		\draw (0) -- (2);
		\draw (2) -- (4);
		\draw (2) -- (3);
		\draw (4) -- (6);
		\draw (6) -- (5);
		\draw (6) -- (7);
		\draw (7) -- (8);
		\draw (8) -- (9);
		\draw (9) -- (10);
		\draw (9) -- (11);
		\draw (11) -- (10);
		\draw (10) -- (12);
		\draw (12) -- (13);
		\draw (11) -- (13);
		\draw [bend left] (0) to (12);
		\draw [bend right] (1) to (8);
		\draw [bend right=25, looseness=1.50] (7) to (13);
    \end{tikzpicture}
    \caption{Corresponding $3$-staircase~$G$}
\label{fig:3-staircase}
    \end{subfigure}
    \begin{subfigure}[b]{.33 \textwidth}
    \centering
    \begin{tikzpicture}[scale=0.8]
    \node [circle,fill=white] (0) at (0, 1) {};
		\node [circle,fill=white] (1) at (0, -0.25) {};
		\node [circle,fill=white] (2) at (1, 0.5) {};
		\node [circle,fill=white] (3) at (0.5, 0.75) {};
		\node [circle,fill=white] (4) at (0.5, 0.125) {};
		\node [circle,fill=white] (5) at (2, 0.5) {};
		\node [circle,fill=white] (6) at (2, -0.25) {};
		\node [circle,fill=white] (7) at (3.05, 0.5) {};
		\node [circle,fill=white] (8) at (4.05, 1) {};
		\node [circle,fill=white] (9) at (3.4, 0.25) {};
		\node [circle,fill=white] (10) at (4.025, 0.575) {};
		\node [circle,fill=white] (11) at (3.975, -0.25) {};
		\node [circle,fill=white] (12) at (3.675, 0) {};
		\node [circle,fill=white] (13) at (4, 0.175) {};
		
		\draw (0) -- (3);
		\draw (3) -- (2);
		\draw (2) -- (4);
		\draw (4) -- (1);
		\draw (1) -- (0);
		\draw (3) -- (4);
		\draw (2) -- (5);
		\draw (5) -- (6);
		\draw (1) -- (6);
		\draw (5) -- (7);
		\draw (7) -- (8);
		\draw (0) -- (8);
		\draw (7) -- (9);
		\draw (9) -- (10);
		\draw (10) -- (13);
		\draw (9) -- (12);
		\draw (12) -- (13);
		\draw (13) -- (11);
		\draw (12) -- (11);
		\draw (6) -- (11);
		\draw (8) -- (10);
    \end{tikzpicture}
    \vspace{10pt}
    \caption{Another drawing of $G$}
\label{fig:3-staircase-isomorphic}
    \end{subfigure}
\caption{$3$-dumbbells and $3$-staircases}
\label{fig:3-staircases}
\end{figure}

\begin{figure}[!htb]
        \centering
        \begin{subfigure}[b]{.3\textwidth}
        \centering
        \begin{tikzpicture}[scale=0.5]
		\node [circle,fill=white] (0) at (0, 1) {};
		\node [circle,fill=white] (1) at (0, 0) {};
		\node [circle,fill=white] (2) at (1, 1) {};
		\node [circle,fill=white] (3) at (1, 0) {};
		\node [circle,fill=white] (4) at (2, 1) {};
		\node [circle,fill=white] (5) at (2, 0) {};
		\node [circle,fill=white] (6) at (3, 0.5) {};
		\node [circle,fill=white] (7) at (4, 0.5) {};
		\node [circle,fill=white] (8) at (5, 0.5) {};
		\node [circle,fill=white] (9) at (6, 0.5) {};
		\node [circle,fill=white] (10) at (7, 1) {};
		\node [circle,fill=white] (11) at (7, 0) {};
		\node [circle,fill=white] (12) at (8, 1) {};
		\node [circle,fill=white] (13) at (8, 0) {};

        \draw[bend left] (0) to (1);
        \draw[bend right] (0) to (1);
        \draw[bend left] (2) to (3);
        \draw[bend right] (2) to (3);
        \draw[bend left] (4) to (5);
        \draw[bend right] (4) to (5);
        \draw[bend left] (6) to (7);
        \draw[bend right] (6) to (7);
        \draw[bend left] (8) to (9);
        \draw[bend right] (8) to (9);
        \draw[bend left] (10) to (11);
        \draw[bend right] (10) to (11);
        \draw[bend left] (12) to (13);
        \draw[bend right] (12) to (13);
		
		\draw (1) -- (3);
		\draw (3) -- (5);
		\draw (0) -- (2);
		\draw (2) -- (4);
		\draw (4) -- (6);
		\draw (6) -- (5);
		\draw (7) -- (8);
		\draw (9) -- (10);
		\draw (9) -- (11);
		\draw (10) -- (12);
		\draw (11) -- (13);
        \end{tikzpicture}
        \vspace{12pt}
        \caption{A $3$-dumbbell $J$}
\label{fig:3-dumbbell-thickness-3}
    \end{subfigure}
    \begin{subfigure}[b]{.33 \textwidth}
    \centering
    \begin{tikzpicture}[scale=0.5]
      \node [circle,fill=white] (0) at (0, 1) {};
		\node [circle,fill=white] (1) at (0, 0) {};
		\node [circle,fill=white] (2) at (1, 1) {};
		\node [circle,fill=white] (3) at (1, 0) {};
		\node [circle,fill=white] (4) at (2, 1) {};
		\node [circle,fill=white] (5) at (2, 0) {};
		\node [circle,fill=white] (6) at (3, 0.5) {};
		\node [circle,fill=white] (7) at (4, 0.5) {};
		\node [circle,fill=white] (8) at (5, 0.5) {};
		\node [circle,fill=white] (9) at (6, 0.5) {};
		\node [circle,fill=white] (10) at (7, 1) {};
		\node [circle,fill=white] (11) at (7, 0) {};
		\node [circle,fill=white] (12) at (8, 1) {};
		\node [circle,fill=white] (13) at (8, 0) {};

        \draw[bend left] (0) to (1);
        \draw[bend right] (0) to (1);
        \draw[bend left] (2) to (3);
        \draw[bend right] (2) to (3);
        \draw[bend left] (4) to (5);
        \draw[bend right] (4) to (5);
        \draw[bend left] (6) to (7);
        \draw[bend right] (6) to (7);
        \draw[bend left] (8) to (9);
        \draw[bend right] (8) to (9);
        \draw[bend left] (10) to (11);
        \draw[bend right] (10) to (11);
        \draw[bend left] (12) to (13);
        \draw[bend right] (12) to (13);
		
		\draw (1) -- (3);
		\draw (3) -- (5);
		\draw (0) -- (2);
		\draw (2) -- (4);
		\draw (4) -- (6);
		\draw (6) -- (5);
		\draw (7) -- (8);
		\draw (9) -- (10);
		\draw (9) -- (11);
		\draw (10) -- (12);
		\draw (11) -- (13);
		\draw [bend left] (0) to (12);
		\draw [bend right] (1) to (8);
		\draw [bend right=25, looseness=1.50] (7) to (13);
    \end{tikzpicture}
    \caption{Corresponding $3$-staircase~$G$}
\label{fig:3-staircase-thickness-3}
    \end{subfigure}
    \begin{subfigure}[b]{.33 \textwidth}
    \centering
    \begin{tikzpicture}[scale=0.8]
    \node [circle,fill=white] (0) at (0, 1) {};
		\node [circle,fill=white] (1) at (0, -0.25) {};
		\node [circle,fill=white] (2) at (1, 0.5) {};
		\node [circle,fill=white] (3) at (0.5, 0.75) {};
		\node [circle,fill=white] (4) at (0.5, 0.125) {};
		\node [circle,fill=white] (5) at (2, 0.5) {};
		\node [circle,fill=white] (6) at (2, -0.25) {};
		\node [circle,fill=white] (7) at (3.05, 0.5) {};
		\node [circle,fill=white] (8) at (4.05, 1) {};
		\node [circle,fill=white] (9) at (3.4, 0.25) {};
		\node [circle,fill=white] (10) at (4.025, 0.575) {};
		\node [circle,fill=white] (11) at (3.975, -0.25) {};
		\node [circle,fill=white] (12) at (3.675, 0) {};
		\node [circle,fill=white] (13) at (4, 0.175) {};

        \draw[bend left=15] (0) to (1);
        \draw[bend right=15] (0) to (1);
        \draw[bend left] (3) to (4);
        \draw[bend right] (3) to (4);
        \draw[bend left=15] (2) to (5);
        \draw[bend right=15] (2) to (5);
        \draw[bend left=15] (7) to (8);
        \draw[bend right=15] (7) to (8);
        \draw[bend left] (9) to (10);
        \draw[bend right] (9) to (10);
        \draw[bend left] (12) to (13);
        \draw[bend right] (12) to (13);
        \draw[bend left=13] (6) to (11);
        \draw[bend right=13] (6) to (11);
		
		\draw (0) -- (3);
		\draw (3) -- (2);
		\draw (2) -- (4);
		\draw (4) -- (1);
		\draw (5) -- (6);
		\draw (1) -- (6);
		\draw (5) -- (7);
		\draw (0) -- (8);
		\draw (7) -- (9);
		\draw (10) -- (13);
		\draw (9) -- (12);
		\draw (13) -- (11);
		\draw (12) -- (11);
		\draw (8) -- (10);
    \end{tikzpicture}
    \vspace{10pt}
    \caption{Another drawing of $G$}
\label{fig:3-staircase-isomorphic-thickness-3}
    \end{subfigure}
\caption{$3$-dumbbells and $3$-staircases of thickness two}
\label{fig:3-staircases-thickness-3}
\end{figure}

The two smallest $3$-staircases are the bicorn~$R_8$, and the graph~$N_{10}$ shown in Figure~\ref{fig:N10}; the reader may easily observe the symmetries of $N_{10}$ that we shall find useful later.

\begin{figure}[!htb]
        \centering
        \begin{subfigure}[b]{.48 \textwidth}
        \centering
        \begin{tikzpicture}[scale=1]
      \node [circle,fill=white] (0) at (-0.5, 0) {};
        		\node [circle,fill=white] (1) at (2.5, 0) {};
        		\node [draw=none] at (-0.5,0.3) {$1$};
        		\node [draw=none] at (2.5,0.3) {$4$};
		\node [draw=none] at (1.5,1.3) {$8$};
		\node [draw=none] at (-1.5,1.3) {$0$};
		\node [draw=none] at (3.5,1.3) {$3$};
		\node [draw=none] at (-1.5,-1.3) {$2$};
		\node [draw=none] at (3.5,-1.3) {$5$};
		\node [draw=none] at (0.5,-1.3) {$7$};
		\node [draw=none] at (0.5,0.3) {$6$};
		\node [draw=none] at (1.5,-0.3) {$9$};
        		\node [circle,fill=white] (2) at (3.5, 1) {};
        		\node [circle,fill=white] (3) at (3.5, -1) {};
        		\node [circle,fill=white] (4) at (-1.5, 1) {};
        		\node [circle,fill=white] (5) at (-1.5, -1) {};
        		\node [circle,fill=white] (6) at (0.5, 0) {};
        		\node [circle,fill=white] (7) at (0.5, -1) {};
        		\node [circle,fill=white] (8) at (1.5, 0) {};
        		\node [circle,fill=white] (9) at (1.5, 1) {};
                \draw (4) -- (0);
        		\draw (0) -- (5);
        		\draw (5) -- (4);
        		\draw (4) -- (9);
        		\draw (9) -- (8);
        		\draw (8) -- (6);
        		\draw (0) -- (6);
        		\draw (6) -- (7);
        		\draw (7) -- (5);
        		\draw (8) -- (1);
        		\draw (1) -- (2);
        		\draw (2) -- (3);
        		\draw (3) -- (7);
        		\draw (1) -- (3);
        		\draw (9) -- (2);
        \end{tikzpicture}
        \caption{}
\label{fig:N10a}
    \end{subfigure}
    \begin{subfigure}[b]{.48 \textwidth}
    \centering
    \begin{tikzpicture}[scale=1]
      \node [circle,fill=white] (0) at (0, 1) {};
		\node [circle,fill=white] (1) at (0, 0) {};
		\node [circle,fill=white] (2) at (1, 0.5) {};
		\node [circle,fill=white] (3) at (2, 0.5) {};
		\node [circle,fill=white] (4) at (3, 0.5) {};
		\node [circle,fill=white] (5) at (4, 0.5) {};
		\node [circle,fill=white] (6) at (5, 1) {};
		\node [circle,fill=white] (7) at (5, 0) {};
		\node [circle,fill=white] (8) at (6, 1) {};
		\node [circle,fill=white] (9) at (6, 0) {};
		\draw (0) -- (1);
		\draw (0) -- (2);
		\draw (2) -- (1);
		\draw (2) -- (3);
		\draw (3) -- (4);
		\draw (4) -- (5);
		\draw (5) -- (6);
		\draw (6) -- (7);
		\draw (7) -- (5);
		\draw (6) -- (8);
		\draw (8) -- (9);
		\draw (9) -- (7);
		\draw [bend right] (1) to (9);
		\draw [bend left] (0) to (4);
		\draw [bend left] (3) to (8);
    \end{tikzpicture}
    \caption{}
\label{}
    \end{subfigure}
\caption{Two drawings of the $3$-staircase $N_{10}$}
\label{fig:N10}
\end{figure}

\begin{prp}\label{prp:N10-vertex-orbits}
The automorphism group of the graph~$N_{10}$ has five vertex-orbits --- 
\newline$\{0,5\},\{1,4\},\{2,3\}, \{6,9\}$ and $\{7,8\}$ as per the labeling shown in Figure~\ref{fig:N10a}. \qed
\end{prp}

Unlike the $1$-staircases, for each even order $n\ge 12$, the number of $3$-staircases is at least two and grows as $n$ increases. Figure~\ref{fig:st12} shows the only two $3$-staircases of order twelve; each of them has solitude pattern~$(2,1)$, and to observe that they are nonisomorphic, one may consider the distance between their only two triangles.
The following is our characterization of \ecc{3} \mbox{$r$-graphs} that have two solitary classes of different cardinalities.

\begin{thm}\label{thm:SDSS}
A \ecc{3} $r$-graph~$G$ has a solitary doubleton as well as a solitary singleton if and only if $G$ is 
a $3$-staircase of thickness~$r-2$. Furthermore, precisely one of the following holds:
\begin{enumerate}[(i)]
\item $G$ is the bicorn~$R_8$ and has solitude pattern $(2,2,1)$, or
\item $G$ is $N_{10}$ and has solitude pattern $(2,1,1)$, or otherwise
\item $G$ has solitude pattern $(2,1)$.
\end{enumerate}
\end{thm}


\begin{figure}[!htb]
        \centering
        \begin{subfigure}[b]{.48 \textwidth}
        \centering
        \begin{tikzpicture}[scale=1]
      \node [circle,fill=white] (0) at (-0.5, 0) {};
        		\node [circle,fill=white] (1) at (2.5, 0) {};
        		\node [circle,fill=white] (2) at (3.5, 1) {};
        		\node [circle,fill=white] (3) at (3.5, -1) {};
        		\node [circle,fill=white] (4) at (-1.5, 1) {};
        		\node [circle,fill=white] (5) at (-1.5, -1) {};
        		\node [circle,fill=white] (6) at (0.5, 0) {};
        		\node [circle,fill=white] (7) at (0.5, -1) {};
        		\node [circle,fill=white] (8) at (1.5, 0) {};
        		\node [circle,fill=white] (9) at (1.5, 1) {};
        		\node [circle,fill=white] (10) at (-1.5, 0) {};
        		\node [circle,fill=white] (11) at (-1, 0.5) {};
                \draw (4) -- (11) -- (0);
        		\draw (0) -- (5);
        		\draw (10) -- (11);
        		\draw (4) -- (10) -- (5);
        		\draw (4) -- (9);
        		\draw (9) -- (8);
        		\draw (8) -- (6);
        		\draw (0) -- (6);
        		\draw (6) -- (7);
        		\draw (7) -- (5);
        		\draw (8) -- (1);
        		\draw (1) -- (2);
        		\draw (2) -- (3);
        		\draw (3) -- (7);
        		\draw (1) -- (3);
        		\draw (9) -- (2);
        \end{tikzpicture}
        \caption{}
\label{fig:st12a}
    \end{subfigure}
    \begin{subfigure}[b]{.48 \textwidth}
    \centering
    \begin{tikzpicture}[scale=1]
      \node [circle,fill=white] (0) at (-0.5, 0) {};
        		\node [circle,fill=white] (1) at (2.5, 0) {};
        		\node [circle,fill=white] (2) at (3.5, 1) {};
        		\node [circle,fill=white] (3) at (3.5, -1) {};
        		\node [circle,fill=white] (4) at (-1.5, 1) {};
        		\node [circle,fill=white] (5) at (-1.5, -1) {};
        		\node [circle,fill=white] (6) at (0.5, 0) {};
        		\node [circle,fill=white] (7) at (0.5, -1) {};
        		\node [circle,fill=white] (8) at (1.5, 0) {};
        		\node [circle,fill=white] (9) at (1.5, 1) {};
		\node [circle,fill=white] (11) at (-1, 0.5) {};
		\node [circle,fill=white] (10) at (-1, -0.5) {};
                \draw (4) -- (11) -- (0);
        		\draw (0) -- (10) -- (5);
        		\draw (11) -- (10);
        		\draw (5) -- (4);
        		\draw (4) -- (9);
        		\draw (9) -- (8);
        		\draw (8) -- (6);
        		\draw (0) -- (6);
        		\draw (6) -- (7);
        		\draw (7) -- (5);
        		\draw (8) -- (1);
        		\draw (1) -- (2);
        		\draw (2) -- (3);
        		\draw (3) -- (7);
        		\draw (1) -- (3);
        		\draw (9) -- (2);
    \end{tikzpicture}
    \caption{}
\label{st12b}
    \end{subfigure}
\caption{Smallest \ecc{3} $3$-graphs with solitude pattern~$(2,1)$}
\label{fig:st12}
\end{figure}

A proof of the forward implication of the above theorem appears in Section~\ref{sec:oneSD-oneSS}. We now proceed to convince the reader of the reverse implication along with statements~(i),~(ii)~and~(iii). In light of Proposition~\ref{prp:matchable-spanning-subgraph}, in order to locate the solitary edges of a $3$-staircase, one may first consider locating the solitary as well as unmatchable edges of a $3$-dumbbell.

Let $J$ denote a $3$-dumbbell with $uvwx$ as its bone. Note that $vw$ is its unique even $1$-cut, say $\partial_{J}(Y)$, where $u,v\in Y$. Since $uv$ and $wx$ are odd $1$-cuts of~$J$, they participate in each \pmg\ of~$J$; consequently, the five edges that are adjacent with either of them are unmatchable. Furthermore, each of the remaining edges of~$J$ is matchable (since each ladder is matching covered). 
Note that, if the ladder~$J[\overline{Y}-w-x]$ is distinct from $K_2$, then each matchable edge in $J[Y+w+x]$ participates in at least two \pms\ of~$J$; an analogous statement holds for $J[Y-u-v]$ and $J[\overline{Y}+u+v]$. These observations, along with Proposition~\ref{prp:ladder-is-bipartite-mcg}, imply that if $J$ has a solitary edge then one of the ladders $J[Y-u-v]$ and $J[\overline{Y}-w-x]$ is isomorphic to $K_2$, whereas the other one is of order at most eight. This proves the following except for the easy task of locating the solitary edges in the graphs shown in Figure~\ref{fig:3-dumbbells}.

\begin{prp}\label{prp:solitary-unmatchable-in-3-dumbbell}
Every $3$-dumbbell $J$ of thickness~$t \in \mathbb{Z}^+$ has exactly five unmatchable edges; one of them is its unique even $1$-cut, and each of the remaining four edges is incident with a socket and belongs to a triangle. Also, $J$~is free of solitary edges unless $J$ is one of the four graphs shown in Figure~\ref{fig:3-dumbbells} --- wherein the solitary edges are indicated in red. \qed
\end{prp}

\begin{figure}[!htb]
    \centering
    \begin{subfigure}[b]{.2\textwidth}
        \begin{tikzpicture}[scale=0.5]
        \node [circle,fill=white] (0) at (0, 1) {};
		\node [circle,fill=white] (1) at (0, 0) {};
		\node [circle,fill=white] (2) at (1, 0.5) {};
		\node [circle,fill=white] (6) at (2, 0.5) {};
		\node [circle,fill=white] (7) at (3, 0.5) {};
		\node [circle,fill=white] (3) at (2+2, 0.5) {};
		\node [circle,fill=white] (4) at (3+2, 1) {};
		\node [circle,fill=white] (5) at (3+2, 0) {};
		\draw[thick,red] (0) -- (1);
		\draw (0) -- (2);
		\draw (2) -- (1);
		\draw[thick,red] (2) -- (6);
		\draw (6) -- (7);
		\draw[thick,red] (7) -- (3);
		\draw (3) -- (4);
		\draw[thick,red] (4) -- (5);
		\draw (5) -- (3);
        \end{tikzpicture}
    \end{subfigure}
    \begin{subfigure}[b]{.22\textwidth}
        \begin{tikzpicture}[scale=0.5]
        \node [circle,fill=white] (0) at (0, 1) {};
		\node [circle,fill=white] (1) at (0, 0) {};
		\node [circle,fill=white] (2) at (1, 0.5) {};
		\node [circle,fill=white] (3) at (2+2, 0.5) {};
		\node [circle,fill=white] (8) at (2, 0.5) {};
		\node [circle,fill=white] (9) at (3, 0.5) {};
		\node [circle,fill=white] (4) at (3+2, 1) {};
		\node [circle,fill=white] (5) at (3+2, 0) {};
		\node [circle,fill=white] (6) at (4+2, 1) {};
		\node [circle,fill=white] (7) at (4+2, 0) {};
		\draw (0) -- (1);
		\draw (0) -- (2);
		\draw (2) -- (1);
		\draw (2) -- (8);
		\draw (8) -- (9);
		\draw (9) -- (3);
		\draw (3) -- (4);
		\draw[thick,red] (4) -- (5);
		\draw (5) -- (3);
		\draw[thick,red] (4) -- (6);
		\draw[thick,red] (6) -- (7);
		\draw[thick,red] (7) -- (5);
        \end{tikzpicture}
    \end{subfigure}
    \begin{subfigure}[b]{.28\textwidth}
        \centering
        \begin{tikzpicture}[scale=0.5]
        \node [circle,fill=white] (0) at (0, 1) {};
		\node [circle,fill=white] (1) at (0, 0) {};
		\node [circle,fill=white] (2) at (1, 0.5) {};
		\node [circle,fill=white] (10) at (2, 0.5) {};
		\node [circle,fill=white] (11) at (3, 0.5) {};
		\node [circle,fill=white] (3) at (2+2, 0.5) {};
		\node [circle,fill=white] (4) at (3+2, 1) {};
		\node [circle,fill=white] (5) at (3+2, 0) {};
		\node [circle,fill=white] (6) at (4+2, 1) {};
		\node [circle,fill=white] (7) at (4+2, 0) {};
		\node [circle,fill=white] (8) at (5+2, 1) {};
		\node [circle,fill=white] (9) at (5+2, 0) {};
		\draw (0) -- (1);
		\draw (0) -- (2);
		\draw (2) -- (1);
		\draw (2) -- (10);
		\draw (10) -- (11);
		\draw (11) -- (3);
		\draw (3) -- (4);
		\draw (4) -- (5);
		\draw (5) -- (3);
		\draw[thick,red] (4) -- (6);
		\draw[thick,red] (6) -- (7);
		\draw[thick,red] (7) -- (5);
		\draw[thick,red] (6) -- (8);
		\draw (8) -- (9);
		\draw[thick,red] (9) -- (7);
        \end{tikzpicture}
    \end{subfigure}
    \begin{subfigure}[b]{.25\textwidth}
        \centering
        \begin{tikzpicture}[scale=0.5]
        \node [circle,fill=white] (0) at (0, 1) {};
		\node [circle,fill=white] (1) at (0, 0) {};
		\node [circle,fill=white] (2) at (1, 0.5) {};
		\node [circle,fill=white] (12) at (2, 0.5) {};
		\node [circle,fill=white] (13) at (3, 0.5) {};
		\node [circle,fill=white] (3) at (2+2, 0.5) {};
		\node [circle,fill=white] (4) at (3+2, 1) {};
		\node [circle,fill=white] (5) at (3+2, 0) {};
		\node [circle,fill=white] (6) at (4+2, 1) {};
		\node [circle,fill=white] (7) at (4+2, 0) {};
		\node [circle,fill=white] (8) at (5+2, 1) {};
		\node [circle,fill=white] (9) at (5+2, 0) {};
		\node [circle,fill=white] (10) at (6+2, 1) {};
		\node [circle,fill=white] (11) at (6+2, 0) {};
		\draw (0) -- (1);
		\draw (0) -- (2);
		\draw (2) -- (1);
		\draw (2) -- (12) -- (13) -- (3);
		\draw (3) -- (4);
		\draw (4) -- (5);
		\draw (5) -- (3);
		\draw (4) -- (6);
		\draw (6) -- (7);
		\draw (7) -- (5);
		\draw[thick,red] (6) -- (8);
		\draw (8) -- (9);
		\draw[thick,red] (9) -- (7);
		\draw (8) -- (10);
		\draw (10) -- (11);
		\draw (11) -- (9);
        \end{tikzpicture}
    \end{subfigure}
    \caption{Illustration for Proposition~\ref{prp:solitary-unmatchable-in-3-dumbbell}}
    \label{fig:3-dumbbells}
\end{figure}

Now, let $G$ be a $3$-staircase of thickness~$t\in \mathbb{Z}^+$ that is obtained from a $3$-dumbbell~$J$ by adding three edges $d,d'$~and~$f$; adjust notation as per Proposition~\ref{prp:3-staircase-ss}.
Observe that the matchable graph~$H$ obtained from $G$ by deleting both ends of~$d$ (or of~$d'$, or of~$f$) is $2$-connected; consequently, by Kotzig's Lemma~(\ref{lem:unique-pm-in-a-graph}), $H$ has at least two \pms, and thus none of $d,d'$ and $f$ is solitary in~$G$. It follows that each solitary edge of~$G$ thus belongs to~$J$. In particular, if $J$ is not one of the four graphs shown in Figure~\ref{fig:3-dumbbells} then, by Propositions~\ref{prp:solitary-unmatchable-in-3-dumbbell}~and~\ref{prp:matchable-spanning-subgraph}, $G$~has at most five solitary edges, and these five are precisely the unmatchable edges of~$J$. We shall now focus on these five edges when $J$ is not one of the four graphs shown in Figure~\ref{fig:3-dumbbells}.

As noted in Proposition~\ref{prp:3-staircase-ss}, the unique even $1$-cut of~$J$ is a solitary singleton of~$G$. Now, let $\alpha,\beta,\alpha'$~and~$\beta'$ denote the remaining four unmatchable edges of~$J$ so that $\alpha$~and~$\alpha'$ are adjacent. We invite the reader to observe that two of them, say~$\alpha$ and $\beta$, depend on each other, as well as on~$d$~and~$d'$, whereas the remaining two $\alpha'$~and~$\beta'$ depend on~$f$. Using these observations, and Kotzig's Lemma (\ref{lem:unique-pm-in-a-graph}), it is easy to see that $\{\alpha,\beta\}$ is a solitary doubleton, whereas neither $\alpha'$ nor $\beta'$ is solitary.

The above observations, along with checking the small $3$-staircases that are obtained from the \mbox{$3$-dumbbells} shown in Figure~\ref{fig:3-dumbbells}, 
prove the reverse implication of Theorem~\ref{thm:SDSS}, as well as statements~(i), (ii)~and~(iii). Our next goal is to count the number of \ecc{3} $r$-graphs for a fixed value of~$r$, of order~$n$, that have solitary classes of different cardinalities; assuming Theorem~\ref{thm:SDSS}, it suffices to count the number of \mbox{$3$-staircases}.

It follows from our above discussion that each $3$-staircase~$G$, of order~$n\ge 12$, has a unique solitary singleton~$e$ and a unique solitary doubleton~$\{\alpha, \beta\}$; consequently, the three edges $d,d'$~and~$f$ are uniquely determined, and their deletion results in a $3$-\dumbbell. On the other hand, for $n\in \{8,10\}$, there is precisely one $3$-dumbbell, and thus precisely one $3$-staircase. These facts prove the following.

\begin{prp}\label{prp:3-staircase-3-dumbbell-isomorphism}
Let $G$~and~$G'$ denote $3$-staircases that are obtained from $3$-dumbbells $J$~and~$J'$, respectively, by adding three edges. Then, $G$ is isomorphic to $G'$ if and only if $J$ is isomorphic to~$J'$. \qed
\end{prp}

By the above proposition, it suffices to count the number of $3$-dumbbells of order~$n$, where $n$ is even and at least eight. The reader may observe that this number is the same as the number of partitions of $x:=\frac{n-8}{2}$ into exactly two parts where parts are allowed to be zero --- that is, $\lceil \frac{x+1}{2} \rceil$. This proves the following.
\begin{prp}
The number of $3$-staircases of thickness~$t\in \mathbb{Z}^+$, of even order~$n\ge 8$, is precisely $\lceil \frac{n-6}{4} \rceil$. \qed
\end{prp}

The above, combined with Theorem~\ref{thm:SDSS}, gives us a precise formula for the number of \ecc{3} $r$-graphs for a fixed value of~$r$, of order~$n$, that have solitary classes of different cardinalities. In the following section, we shall assume Theorems~\ref{thm:more-than-one-SD}~and~\ref{thm:SDSS}, and present a proof of Theorem~\ref{thm:at-most-n/2-solitary-edges}.

\subsection{Proof of Theorem~\ref{thm:at-most-n/2-solitary-edges}: \texorpdfstring{$r$}{}-graphs have at most \texorpdfstring{$\frac{n}{2}$}{} solitary edges} \label{sec:proof-of-Theorem-at-most-n/2-solitary-edges}
Let us first recall the statement we intend to prove.

\begin{reptheorem}{thm:at-most-n/2-solitary-edges}
For any $r$-graph~$G$, precisely one of the following holds:
    \begin{enumerate}[(i)]
        \item either $G$ has at most $\frac{n}{2}$ \ses, 
        \item or otherwise $G\in $ \mbb{\theta}{}{1}   $\cup$~\mbb{K}{4}{1}~ $\cup~{\overline{C_6}}^1 \cup \{R_8\}$.
    \end{enumerate}
\end{reptheorem}
\begin{proof}
    Note that if $G\in $ \mbb{\theta}{}{1}   $\cup$~\mbb{K}{4}{1}~ $\cup~{\overline{C_6}}^1 \cup \{R_8\}$, then $G$ is $3$-connected and $G$ has at least $\frac{n}{2} +1$ solitary edges. If $n\le 4$ then, by Proposition~\ref{prp:solitary-patterns-of-small-r-graphs}, it is easy to verify that the desired conclusion holds. If $G$ is not $3$-edge-connected then, by Corollary~\ref{cor:2-connected-but-not-3-connected}, we are done. Now, we consider the remaining case: $G$ is \ecc{3} and $n\ge 6$. Clearly, if the number of solitary edges is at most three then we are done. Now suppose that $G$ has four or more solitary edges. By Corollary~\ref{cor:main-thm}, either $G$ has at least two solitary doubletons, or it has a solitary doubleton as well as a solitary singleton. In the former case, we invoke Theorem~\ref{thm:more-than-one-SD}, whereas in the latter case, we invoke Theorem~\ref{thm:SDSS}, to observe that the desired conclusion holds.
\end{proof}


In the final part of this section, we describe those \ecc{3} $r$-graphs that have precisely three solitary singletons.

\subsection{Solitude pattern \texorpdfstring{$(1,1,1)$}{} --- the finite family \texorpdfstring{$\mathcal{S}$}{}} \label{sec:5-special-graphs}
Clearly, the theta graph $\theta$, shown in Figure~\ref{fig:theta-graph}, has three solitary singletons. It turns out that apart from $\theta$, there are precisely five special \ecc{3} $r$-graphs, shown in Figure~\ref{fig: 3cc-with-three-distinct-solitary-singletons}, that also possess precisely three solitary singletons --- labeled as $e_1,e_2$~and~$e_3$. We use $\mathcal{S}$ to denote the set comprising these five graphs, and state our characterization as follows.

\begin{figure}[!htb]
    \centering
    \begin{subfigure}[b]{.33\textwidth}
        \centering
        \begin{tikzpicture}[scale=0.8]
            \node[circle,fill=white] (1) at (90:1.25){};
            \node[circle,fill=white] (2) at (210:1.25){};
            \node[circle,fill=white] (3) at (330:1.25){};
            \node[circle,fill=white] (4) at (80-5-5:2+0.5){};
            \node[circle,fill=white] (5) at (100+5+5:2+0.5){};
            \node[circle,fill=white] (6) at (200-5-5:2+0.5){};
            \node[circle,fill=white] (7) at (220+5+5:2+0.5){};
            \node[circle,fill=white] (8) at (320-5-5:2+0.5){};
            \node[circle,fill=white] (9) at (340+5+5:2+0.5){};
            \node[circle,fill=white] (0) at (0:0){};
            \node [draw=none] at (90:2.6) {$e_1$};
		  \node [draw=none] at (210:2.6) {$e_2$};
		  \node [draw=none] at (330:2.6) {$e_3$};
            \draw (0) to (1);
            \draw (0) to (2);
            \draw (0) to (3);
            \draw (4) to (1);
            \draw (5) to (1);
            \draw (6) to (2);
            \draw (7) to (2);
            \draw (8) to (3);
            \draw (9) to (3);
            \draw (4) -- (5) -- (6) -- (7) -- (8) -- (9) -- (4);
        \end{tikzpicture}
        \caption{}
  \label{fig:tricorn}
    \end{subfigure}
    \begin{subfigure}[b]{.33\textwidth}
        \centering
        \begin{tikzpicture}[scale=0.8]
            \node[circle,fill=white] (1) at (90:1.25){};
            \node[circle,fill=white] (2) at (210:1.25){};
            \node[circle,fill=white] (3) at (330:1.25){};
            \node[circle,fill=white] (4) at (80-5-5:2+0.5){};
            \node[circle,fill=white] (5) at (100+5+5:2+0.5){};
            \node[circle,fill=white] (6) at (200-5-5:2+0.5){};
            \node[circle,fill=white] (7) at (220+5+5:2+0.5){};
            \node[circle,fill=white] (8) at (320-5-5:2+0.5){};
            \node[circle,fill=white] (9) at (340+5+5:2+0.5){};
            \node[circle,fill=white] (0) at (0:0){};
            \node[circle,fill=white] (10) at (80-4.5:1+0.95){};
            \node[circle,fill=white] (11) at (100+4.5:1+0.95){};
            \node [draw=none] at (90:2.6) {$e_1$};
		  \node [draw=none] at (210:2.6) {$e_2$};
		  \node [draw=none] at (330:2.6) {$e_3$};
            \draw (10) -- (11);
            \draw (0) to (1);
            \draw (0) to (2);
            \draw (0) to (3);
            \draw (4) to (1);
            \draw (5) to (1);
            \draw (6) to (2);
            \draw (7) to (2);
            \draw (8) to (3);
            \draw (9) to (3);
            \draw (4) -- (5) -- (6) -- (7) -- (8) -- (9) -- (4);
            \node[circle,fill=white] (1) at (90:1.25){};
            \node[circle,fill=white] (2) at (210:1.25){};
            \node[circle,fill=white] (3) at (330:1.25){};
            \node[circle,fill=white] (4) at (80-5-5:2+0.5){};
            \node[circle,fill=white] (5) at (100+5+5:2+0.5){};
            \node[circle,fill=white] (6) at (200-5-5:2+0.5){};
            \node[circle,fill=white] (7) at (220+5+5:2+0.5){};
            \node[circle,fill=white] (8) at (320-5-5:2+0.5){};
            \node[circle,fill=white] (9) at (340+5+5:2+0.5){};
            \node[circle,fill=white] (0) at (0:0){};
            \node[circle,fill=white] (10) at (80-4.5:1+0.95){};
            \node[circle,fill=white] (11) at (100+4.5:1+0.95){};
        \end{tikzpicture}
        \caption{}
  \label{fig:tricorn+2vertices}
    \end{subfigure}
    \begin{subfigure}[b]{.3\textwidth}
        \centering
        \begin{tikzpicture}[scale=0.8]
            \node[circle,fill=white] (1) at (90:1.25){};
            \node[circle,fill=white] (2) at (210:1.25){};
            \node[circle,fill=white] (3) at (330:1.25){};
            \node[circle,fill=white] (4) at (80-5-5:2+0.5){};
            \node[circle,fill=white] (5) at (100+5+5:2+0.5){};
            \node[circle,fill=white] (6) at (200-5-5:2+0.5){};
            \node[circle,fill=white] (7) at (220+5+5:2+0.5){};
            \node[circle,fill=white] (8) at (320-5-5:2+0.5){};
            \node[circle,fill=white] (9) at (340+5+5:2+0.5){};
            \node[circle,fill=white] (0) at (0:0){};
            \node[circle,fill=white] (10) at (320-4.5:1+0.95){};
            \node[circle,fill=white] (11) at (340+4.5:1+0.95){};
            \node[circle,fill=white] (12) at (200-4.5:1+0.95){};
            \node[circle,fill=white] (13) at (220+4.5:1+0.95){};
            \node [draw=none] at (90:2.6) {$e_1$};
		  \node [draw=none] at (210:2.6) {$e_2$};
		  \node [draw=none] at (330:2.6) {$e_3$};
            \draw (10) -- (11);
            \draw (13) -- (12);
            \draw (0) to (1);
            \draw (0) to (2);
            \draw (0) to (3);
            \draw (4) to (1);
            \draw (5) to (1);
            \draw (6) to (2);
            \draw (7) to (2);
            \draw (8) to (3);
            \draw (9) to (3);
            \draw (4) -- (5) -- (6) -- (7) -- (8) -- (9) -- (4);
            \node[circle,fill=white] (1) at (90:1.25){};
            \node[circle,fill=white] (2) at (210:1.25){};
            \node[circle,fill=white] (3) at (330:1.25){};
            \node[circle,fill=white] (4) at (80-5-5:2+0.5){};
            \node[circle,fill=white] (5) at (100+5+5:2+0.5){};
            \node[circle,fill=white] (6) at (200-5-5:2+0.5){};
            \node[circle,fill=white] (7) at (220+5+5:2+0.5){};
            \node[circle,fill=white] (8) at (320-5-5:2+0.5){};
            \node[circle,fill=white] (9) at (340+5+5:2+0.5){};
            \node[circle,fill=white] (0) at (0:0){};
            \node[circle,fill=white] (10) at (320-4.5:1+0.95){};
            \node[circle,fill=white] (11) at (340+4.5:1+0.95){};
            \node[circle,fill=white] (12) at (200-4.5:1+0.95){};
            \node[circle,fill=white] (13) at (220+4.5:1+0.95){};
        \end{tikzpicture}
        \caption{}
  \label{fig:tricorn+4vertices}
    \end{subfigure}
    \begin{subfigure}[b]{.53\textwidth}
        \centering
        \begin{tikzpicture}[scale=0.8]
            \node[circle,fill=white] (1) at (90:1.25){};
            \node[circle,fill=white] (2) at (210:1.25){};
            \node[circle,fill=white] (3) at (330:1.25){};
            \node[circle,fill=white] (4) at (80-5-5:2+0.5){};
            \node[circle,fill=white] (5) at (100+5+5:2+0.5){};
            \node[circle,fill=white] (6) at (200-5-5:2+0.5){};
            \node[circle,fill=white] (7) at (220+5+5:2+0.5){};
            \node[circle,fill=white] (8) at (320-5-5:2+0.5){};
            \node[circle,fill=white] (9) at (340+5+5:2+0.5){};
            \node[circle,fill=white] (0) at (0:0){};
            \node[circle,fill=white] (10) at (80-4.5:1+0.95){};
            \node[circle,fill=white] (11) at (100+4.5:1+0.95){};
            \node[circle,fill=white] (12) at (200-4.5:1+0.95){};
            \node[circle,fill=white] (13) at (220+4.5:1+0.95){};
            \node[circle,fill=white] (14) at (320-4.5:1+0.95){};
            \node[circle,fill=white] (15) at (340+4.5:1+0.95){};
            \node [draw=none] at (90:2.6) {$e_1$};
		  \node [draw=none] at (210:2.6) {$e_2$};
		  \node [draw=none] at (330:2.6) {$e_3$};
            \draw (10) -- (11);
            \draw (13) -- (12);
            \draw (14) -- (15);
            \draw (0) to (1);
            \draw (0) to (2);
            \draw (0) to (3);
            \draw (4) to (1);
            \draw (5) to (1);
            \draw (6) to (2);
            \draw (7) to (2);
            \draw (8) to (3);
            \draw (9) to (3);
            \draw (4) -- (5) -- (6) -- (7) -- (8) -- (9) -- (4);
            \node[circle,fill=white] (1) at (90:1.25){};
            \node[circle,fill=white] (2) at (210:1.25){};
            \node[circle,fill=white] (3) at (330:1.25){};
            \node[circle,fill=white] (4) at (80-5-5:2+0.5){};
            \node[circle,fill=white] (5) at (100+5+5:2+0.5){};
            \node[circle,fill=white] (6) at (200-5-5:2+0.5){};
            \node[circle,fill=white] (7) at (220+5+5:2+0.5){};
            \node[circle,fill=white] (8) at (320-5-5:2+0.5){};
            \node[circle,fill=white] (9) at (340+5+5:2+0.5){};
            \node[circle,fill=white] (0) at (0:0){};
            \node[circle,fill=white] (10) at (80-4.5:1+0.95){};
            \node[circle,fill=white] (11) at (100+4.5:1+0.95){};
            \node[circle,fill=white] (12) at (200-4.5:1+0.95){};
            \node[circle,fill=white] (13) at (220+4.5:1+0.95){};
            \node[circle,fill=white] (14) at (320-4.5:1+0.95){};
            \node[circle,fill=white] (15) at (340+4.5:1+0.95){};
        \end{tikzpicture}
        \caption{}
  \label{fig:tricorn+6vertices}
    \end{subfigure}
    \begin{subfigure}[b]{.44\textwidth}
        \centering
        \begin{tikzpicture}[scale=0.8]
            \node[circle,fill=white] (0) at (0,0){};
            \node[circle,fill=white] (1) at (1,0){};
            \node[circle,fill=white] (2) at (-1,0){};
            \node[circle,fill=white] (3) at (2,0){};
            \node[circle,fill=white] (4) at (-2,0){};
            \node[circle,fill=white] (5) at (0,1){};
            \node[circle,fill=white] (6) at (-1,-1){};
            \node[circle,fill=white] (7) at (1,-1){};
            \node[circle,fill=white] (8) at (-3,-1){};
            \node[circle,fill=white] (9) at (3,-1){};
            \node[circle,fill=white] (10) at (-3,1){};
            \node[circle,fill=white] (11) at (3,1){};
            \node [draw=none] at (-1.3,-0.5) {$e_1$};
		  \node [draw=none] at (1.3,-0.5) {$e_2$};
		  \node [draw=none] at (0.3,0.5) {$e_3$};
            \draw (0) -- (1) -- (3) -- (11) -- (9) -- (7) -- (6) -- (2) -- (4) -- (8) -- (10) -- (5) -- (0);
            \draw (8) -- (6);
            \draw (4) -- (10);
            \draw (0) -- (2);
            \draw (5) -- (11);
            \draw (1) -- (7);
            \draw (3) -- (9);
        \end{tikzpicture}
        \caption{}
  \label{fig:different-1+1+1-graph}
    \end{subfigure}
    \caption{The family $\mathcal{S}$}
    \label{fig: 3cc-with-three-distinct-solitary-singletons}
\end{figure}


\begin{thm}\label{thm:3SS}
	A \ecc{3} $r$-graph~$G$ has 
solitude pattern $(1,1,1)$ if and only if $G$ is one of the six graphs in~$\mathcal{S}\cup \{\theta\}$.
\end{thm}


A proof of the forward implication of the above theorem appears in Section~\ref{sec:3SS}, and we leave the reverse implication as an exercise for the reader. We conclude this section with the following immediate implication of the above, Corollary~\ref{lem:solitary-patterns-in-r-graphs} and Proposition~\ref{r-graphs-of-order-two}.

\begin{cor}
    A \ecc{3} $r$-graph~$G$ has three or more solitary singletons if and only if $G\in \mathcal{S}~\cup$ \mbb{\theta}{}{1}. \qed
\end{cor}

\section{Distance between solitary classes}
\label{sec:further-developing-toolkit}

For distinct edges $e_1$ and $e_2$ of a graph $G$, we use \distg{G}{e_1}{e_2} to denote the distance between them --- that is, the length of a shortest path that starts at an end of $e_1$ and ends at an end of $e_2$. We extend this notion to equivalence classes $D_1$~and~$D_2$: the distance between them is $\min\{$\distg{G}{e_1}{e_2}$~|~e_1\in~D_1~\rm{and}~e_2\in D_2\}$. 
In this section, we shall establish that the distance between any two solitary classes, in a \ecc{3} \mbox{$r$-graph}, is at most one. In order to do so, we shall need several technical results, and some of them will also be used in Section~\ref{sec:proofs-of-main-results} to prove the forward implications of Theorems~\ref{thm:more-than-one-SD},~\ref{thm:SDSS}~and~\ref{thm:3SS}. Let us recall Proposition~\ref{prp:solitary-calls-for-uniquely-matchable} which implies that any study of solitary edges is intrinsically related to the study of uniquely matchable graphs.

\subsection{Uniquely matchable bipartite graphs} \label{sec:results-pertaining-to-bipartite-graphs}

For any edge~$e:=ab$ of a bipartite graph~$H[A,B]$, where $a\in A$ and $b\in B$, we refer to $a$ as its {\em $A$-end} and $b$ as its {\em $B$-end}.
The following characterization is well known; see~\cite[Lemma~$4.3.2$]{lopl86}. 

\begin{prp} 
\label{thm:uniquely-matchable-bipartite-graph}
A \pmg~$M$ in a bipartite graph~$H[A,B]$ is the only \pmg\ of~$H$ if and only if there exists an ordering $e_1,e_2,\dots,e_{|M|}$ of the edges of~$M$ such that each edge in $E(H)-M$ has its $B$-end in $e_j$ and its $A$-end in $e_i$ for some $1\le j<i\le |M|$.
\end{prp}

One may easily deduce the above using Kotzig's Lemma~(\ref{lem:unique-pm-in-a-graph}). We now state an immediate consequence that is useful to us.

\begin{cor}\label{cor:degree-one}
Every nonempty uniquely matchable bipartite graph $H[A,B]$ has a vertex of degree one in each color class. \qed
\end{cor}

In light of Lemma~\ref{3cc-subgraph-in-r-graphs}, for the most part, it shall suffice to focus on \cc{3}s. It is for this reason that, in the next section, we prove several results pertaining to this class.

\subsection{Some results pertaining to \texorpdfstring{\cc{3}}{}s} \label{sec:results-pertaining-to-3ccs}

We begin this section with the following easy observation; its
reverse implication may also be viewed as a special case of Lemma~\ref{lem:solitary-ME-sepcuts}~(iii).

\begin{prp}\label{prp:solitary-splice-with-K4}
Let $e$ be an edge of a \cc{2}~$G$, let \mbox{$G':=(G\odot K_4)_{v}$}, where $v$ is an end of~$e$, and let $e'$ denote the unique edge of $E(G')-E(G)$ that is nonadjacent with~$e$. Then, $e$ is solitary in~$G$ if and only if $e'$ is solitary in~$G'$. \qed
\end{prp}

In the rest of this section, we focus on \cc{3}s. The following fact is easily verified.

\begin{prp}\label{prp:edge-in-at-most-one-triangle}
In a \cc{3}, of order six or more, any edge participates in at most one triangle. \qed
\end{prp}




Let us recall Proposition~\ref{prp:splicing-with-k4}. Next, we prove that splicing any \cc{3} with $K_4$, at a vertex incident with a solitary doubleton, results in a graph that also has a solitary doubleton.

\begin{cor}\label{cor:splicing-at-a-SD-vertex}
Let $G$ be a \cc{3} that has a solitary doubleton $\{\alpha, \beta\}$, let $G'$ be $(G\odot K_4)_{a}$, where $a$ is an end of $\alpha$, and let $\alpha'$ denote the unique edge in $E(G')-E(G)$ that is nonadjacent with $\alpha$. Then, $\{\alpha', \beta\}$ is a solitary doubleton in~$G'$.
\end{cor}
\begin{proof}
Since $\{\alpha, \beta\}$ is a solitary doubleton in~$G$, by Corollary~\ref{cor:solitary-doubleton-near-bipartite-brick}~(ii), $G-\alpha-\beta$ is bipartite and matching covered; let $A$ and $B$ denote its color classes so that $a\in A$. We let $\alpha':=a_1a_2$, and let $b$ denote the end of~$\alpha$ in~$G'$ that is adjacent with ends of~$\alpha'$. The reader may observe that $G'-\alpha'-\beta$ is also bipartite with equicardinal color classes $A-a+a_1+a_2$ and $B+b$. Clearly, $\alpha' \xleftrightarrow{G'} \beta$.
Since $\alpha$ is solitary in~$G$, by Proposition~\ref{prp:solitary-splice-with-K4}, $\alpha'$ is solitary in~$G'$. By  Corollaries~\ref{prp:solitary-dependence-clasee}~and~\ref{cor:main-thm}~(ii), $\{\alpha', \beta\}$ is a solitary doubleton in~$G'$.
\end{proof}

In order to prove the converse of the above, we need another consequence of Corollary~\ref{cor:solitary-doubleton-near-bipartite-brick}~(ii).


\begin{cor}\label{lem:vertex-disjoint-triangle-lemma}
Every \cc{3} $G$, of order at least six, with a solitary doubleton $\{\alpha,\beta\}$, has precisely two triangles; furthermore, they are vertex-disjoint, one of them contains $\alpha$ whereas the other contains~$\beta$.
\end{cor}
\begin{proof}
From Corollary~\ref{cor:solitary-doubleton-near-bipartite-brick}~(ii), we infer that $G'[A',B']:=G-\alpha-\beta$ is bipartite; furthermore, both ends of $\alpha:=a_1a_2$ lie in~$A'$ whereas both ends of $\beta:=b_1b_2$ lie in~$B'$. Let $H[A,B]:=G-a_1-a_2-b_1-b_2$, where $A:=A'-a_1-a_2$ and $B:=B'-b_1-b_2$. Since $\{\alpha,\beta\}$ is a solitary doubleton in~$G$ and since $|V(G)|\ge 6$, the bipartite graph $H$ is uniquely matchable and nonempty. By Corollary~\ref{cor:degree-one}, $H$ has degree one vertices $a\in A$ and $b \in B$. Since $G$ is \mbox{$3$-regular}, $ab_1,ab_2,ba_1,ba_2\in E(G)$. Consequently, $T_1:=ab_1b_2a$ and $T_2:=ba_1a_2b$ are the desired vertex-disjoint triangles. Using the fact that $G'$ is bipartite, and Proposition~\ref{prp:edge-in-at-most-one-triangle}, we conclude that $T_1$~and~$T_2$ are the only triangles in~$G$.
\end{proof}

The above lemma will help us in locating either a $1$-dumbbell or a $3$-dumbbell in Sections~\ref{sec:more than one solitary doubleton}~and~\ref{sec:oneSD-oneSS}, respectively.
We are now ready to state and prove the converse of Corollary~\ref{cor:splicing-at-a-SD-vertex}; its statement relies on the above corollary.

\begin{cor}\label{cor:SD-shrinking-triangle}
Let $G$ be a \cc{3}, of order six or more, that has a solitary doubleton~$\{\alpha,\beta\}$, let $T$ denote the triangle containing~$\alpha$, and let $\alpha'$ denote the unique edge in~$\partial_{G}(T)$ that is nonadjacent with~$\alpha$. Then, $\{\alpha',\beta\}$ is a solitary doubleton in~$G':=G/T$.
\end{cor}
\begin{proof}
By Lemma~\ref{lem:solitary-ME-sepcuts}~(i), $\beta$ is solitary in~$G'$. By Proposition~\ref{prp:solitary-splice-with-K4}, $\alpha'$ is solitary in~$G'$.
Now, let $M$ denote the \pmg\ of~$G$ that contains $\alpha$~and~$\beta$. Observe that $M-\alpha$ is a \pmg\ of~$G'$ containing both $\alpha'$~and~$\beta$. By Corollary~\ref{lem:solitary-class-unique-pm} and Corollary~\ref{cor:main-thm}~$(ii)$, we conclude that $\{\alpha',\beta\}$ is a solitary doubleton in~$G'$.
\end{proof}



 In the following section, we heavily exploit Lemma~\ref{lem:associated-r-cuts} as well as some other results stated earlier in order to prove facts pertaining to solitary edges with respect to triangles.
\subsection{Solitude versus triangles in \texorpdfstring{\cc{3}}{}s}\label{sec:solitude-versus-triangles}

We begin by considering a solitary edge that participates in a triangle; recall Lemma~\ref{lem:associated-r-cuts}.

\begin{cor}\label{cor:triangle-companion}
    If $e:=uv$ is a solitary edge of a \cc{3} $G$ that participates in a triangle $T$, then the companion of $e$ is the unique edge $f\in \partial(T)$ that is nonadjacent with~$e$.
\end{cor}
\begin{proof}
Observe that $f$ is an odd $1$-cut in $G-u-v$. It follows from the final statement of Lemma~\ref{lem:associated-r-cuts} that $f$ is indeed the companion of $e$.
\end{proof}



We now use the above result, along with earlier results, to establish the surprising fact that the companion of a solitary edge participates in at least two \pmg s unless the graph is $K_4$.

\begin{cor}\label{lem:K_4-lemma}
If the companion $f$ of a solitary edge $e:=uv$ in a \cc{3}~$G$ is also solitary then $G$ is~$K_4$.
\end{cor}
\begin{proof}
We proceed by contradiction, and we let $G$ denote a smallest counterexample. We let $C:=\partial(X)$~and~$D:=\partial(Y)$ denote the $3$-cuts associated with the solitary-companion pair~$(e,f)$ so that $X\cap Y=\emptyset$; see Lemma~\ref{lem:associated-r-cuts}. Since $G$ is not $K_4$, at least one of $C$ and $D$ is nontrivial; adjust notation so that $C$ is nontrivial.

By Proposition~\ref{prp:r-cuts}, the smaller \mbox{$3$-regular} graph $G':=G/X$ is also $3$-connected. By Lemma~\ref{lem:solitary-ME-sepcuts}~(i), $e$ and $f$ are solitary in~$G'$. Observe that, by Corollary~\ref{cor:triangle-companion}, $f$ is the companion of $e$ in~$G'$. Thus, it follows from our choice of $G$ that $G'$ is~$K_4$. Consequently, the cut $D$ is trivial, and we let $Y:=\{y\}$. Observe that $T:=uvyu$ is a triangle in~$G$ and $f\in \partial(T)$. By Corollary~\ref{cor:solitary-leaving-triangle-implies-K4}, $G$ is~$K_4$; a contradiction. This completes the proof.
\end{proof}

The reader may recall the definition of distance between two edges stated at the beginning of Section \ref{sec:further-developing-toolkit}. The next result establishes that, in any \cc{3}, any two adjacent solitary edges participate in a triangle unless the graph is~$\theta$.
 
\begin{cor}\label{lem:triangle-lemma}
Let $G$ be a \cc{3}, distinct from $\theta$, that has two distinct solitary edges $e_1$~and~$e_2$. If \distg{G}{e_1}{e_2} $=0$ then there exists a triangle containing $e_1$~and~$e_2$.
\end{cor}
\begin{proof}
We invoke Lemma~\ref{lem:associated-r-cuts}, and let
 $C:=\partial(X)$ and $D:=\partial(Y)$ denote the $3$-cuts associated with $e_1:=uv$ and its companion so that $X\cap Y=\emptyset$. We let  $e_2:=uw$ and adjust notation so that $w\in Y$. Note that it suffices to prove that $Y=\{w\}$.

By Proposition~\ref{prp:r-cuts}, we conclude that $G':=G/(X\rightarrow x)$ is $3$-connected. By Lemma~\ref{lem:solitary-ME-sepcuts}~(i), both $e_1$~and~$e_2$ are also solitary in~$G'$. Since $uvxu$ is a triangle in~$G'$, we infer that $e_3\xrightarrow{G'} e_2$ where $e_3:=vx$. By Lemma~\ref{lem:solitary-source-clasee}~and~Corollary~\ref{cor:main-thm}~(ii), we conclude that $\{e_2,e_3\}$ is a solitary doubleton in~$G'$. Note that $e_2$ is an odd $1$-cut in $G'-v-x$ and thus the companion of $e_3$ in~$G'$. By Corollary~\ref{lem:K_4-lemma}, $G'$ is~$K_4$. Consequently, $Y=\{w\}$, and this completes the proof as noted above.
\end{proof}

We conclude this section with a characterization of those \cc{3}s that have three solitary edges incident at a common vertex.
\begin{cor}\label{lem:three-pairwise-adjacent-solitaryedges}
If a \cc{3} $G$ has a vertex $v$ such that each member of $\partial(v)$ is solitary, then $G$ is either $\theta$ or~$K_4$.
\end{cor}
\begin{proof}
Assume $G$ is distinct from $\theta$, and let $\partial(v):=\{ vv_1,vv_2,vv_3\}$. By Corollary~\ref{lem:triangle-lemma}, any two members of $\{vv_1,vv_2,vv_3\}$ participate in a triangle; in other words, $v_1v_2,v_2v_3,v_1v_3\in E(G)$. Since $G$ is \mbox{$3$-regular}, we conclude that $G$ is~$K_4$.
\end{proof}


In the following part, we prove the main results of this section pertaining to distance between solitary classes.



\subsection{Distance between solitary edges} \label{sec:distance-between-solitary-edges}
We begin this section with an observation that is easily verified using both statements of Lemma~\ref{lem:basic-facts-about-parity}.
\begin{lem}\label{lem:conformal-lemma}
In a $3$-edge-colorable \mbox{$3$-regular} graph $G$, if $C$ is either a $2$-cut or a $4$-cut, then the subgraph induced by either shore of $C$ is matchable. \qed
\end{lem}

We first establish a distance result that is applicable to \cc{3}s.

\begin{thm}\label{thm:distance-at-most-1}
In a \cc{3}, any two mutually exclusive \ses\ are at distance at most one.
\end{thm}
\begin{proof}
    We proceed by contradiction, and we let $G$ denote a smallest \cc{3} that has a pair of mutually exclusive \ses, say $e_1:=u_1v_1$ and $e_2:=u_2v_2$, so that \distg{G}{e_1}{e_2} is at least two.
     Let $f_i$ denote the companion of $e_i$ for each $i\in \{1,2\}$. Note that, since $G$ is $3$-edge-colorable and since $e_1$ and $e_2$ are mutually exclusive, all four edges $e_1,e_2,f_1$ and $f_2$ are pairwise distinct.

    \begin{sta}\label{sta:triangle}
        Each of $e_1$ and $e_2$ lies in a triangle.
    \end{sta}
    \begin{proof}
    By symmetry, it suffices to prove that $e_1$ belongs to a triangle.
    Let $C_1:=\partial(X)$ and $D_1:=\partial(Y)$ denote the $3$-cuts associated with $(e_1,f_1)$; we adjust notation as in the statement of Lemma \ref{lem:associated-r-cuts} so that $X\cap Y=\emptyset$. Note that each edge in $C_1\cup D_1-f_1$ is adjacent with $e_1$. It follows from our choice of $G$ that $e_2$ lies in $G[X]$ or otherwise in $G[Y]$. 
        Adjust notation so that $e_2$ is contained in $G[Y]$.

        Now, we claim that $|X|=1$. Suppose to the contrary, and let us consider the smaller \mbox{$3$-regular} graph $H:=G/X$. It follows from Proposition~\ref{prp:r-cuts} that $H$ is $3$-connected. By Lemma~\ref{lem:solitary-ME-sepcuts}~(i)~and~(ii), the edges $e_1$ and $e_2$ are solitary as well as mutually exclusive in~$H$.
         Thus, by our choice of $G$, \distg{H}{e_1}{e_2} is at most one. Consequently, in $H$, a shortest $e_1e_2$-path does not meet the cut $C_1$. This implies that 
         \distg{G}{e_1}{e_2} is also at most one; this contradicts our choice of $G$. Thus $|X|=1$. It follows that $e_1$ lies in a triangle.
    \end{proof}
    
    Let $T_1:=u_1v_1w_1$ and $T_2:=u_2v_2w_2$ denote the triangles containing $e_1$ and $e_2$, respectively, as per \ref{sta:triangle}. Note that, since $G$ is \mbox{$3$-regular} and \distg{G}{e_1}{e_2} $\ge 2$, the triangles
    $T_1$ and $T_2$ are vertex-disjoint as shown in Figure \ref{fig:triangles}.
    By Corollary \ref{cor:triangle-companion}, for each $i\in \{1,2\}$, the companion $f_i$ (of $e_i$) is the unique member of $\partial(w_i)\cap \partial(T_i)$.
   \begin{figure}[!htb]
        \centering
        \begin{tikzpicture}
        \node[circle,fill=white] (1) at (-1,1){};
        \node[draw=none] at (-1,1.4) {$u_1$};
        \node[circle,fill=white] (2) at (0,0){};
        \node[draw=none] at (0,0.4) {$w_1$};
        \node[draw=none] at (1,0.4) {$f_1$};
        \node[circle,fill=white] (3) at (-1,-1){};
        \node[draw=none] at (-1,-1.4) {$v_1$};
        \node[circle,fill=white] (4) at (7,0){};
        \node[draw=none] at (7,0.4) {$w_2$};
        \node[draw=none] at (6,0.4) {$f_2$};
        \node[circle,fill=white] (5) at (8,-1){};
        \node[draw=none] at (8,-1.4) {$v_2$};
        \node[circle,fill=white] (6) at (8,1){};
        \node[draw=none] at (8,1.4) {$u_2$};
        \draw (1) -- (2) -- (3) -- (1);
        \draw (4) -- (5) -- (6) -- (4);
        \draw (1) -- (2,1);
        \draw (2) -- (2,0);
        \draw (3) -- (2,-1);
        \draw (4) -- (5,0);
        \draw (5) -- (5,-1);
        \draw (6) -- (5,1);
        \node[draw=none] at (-1.4,0) {$e_1$};
        \node[draw=none] at (8.4,0) {$e_2$};
        \end{tikzpicture}
            \caption{The vertex-disjoint triangles $T_1$ and $T_2$}
            \label{fig:triangles}
        \end{figure}
    \begin{sta}
        \distgl{G}{f_1}{e_2} and \distgl{G}{f_2}{e_1}.
    \end{sta}
    \begin{proof}
    By symmetry, it suffices to prove that \distgl{G}{f_1}{e_2}. To this end, we consider the smaller \mbox{$3$-regular} graph $H:=G/V(T_1)$. By Proposition~\ref{prp:r-cuts}, $H$ is $3$-connected. Using Lemma~\ref{lem:solitary-ME-sepcuts}~(i),(iii)~and~(iv), we infer that $f_1$ and $e_2$ are solitary as well as mutually exclusive in~$H$. Thus, by choice of~$G$, \distgl{H}{f_1}{e_2}. Observe that, since \distg{G}{e_1}{e_2} is at least two, a shortest $f_1e_2$-path in $H$ does not meet the cut $\partial(T_1)$. Consequently, \distg{G}{f_1}{e_2} $=$ \distgl{H}{f_1}{e_2}.
    \end{proof}
    
    We now consider cases depending on \distg{G}{f_1}{e_2} and \distg{G}{f_2}{e_1}.

        Let us first suppose that \distg{G}{f_1}{e_2} $=0$ and \distg{G}{f_2}{e_1} $=0$. Adjust notation so that $f_1=w_1v_2$ and $f_2=w_2v_1$. Observe that $|V(G)|=6$ since otherwise $\{u_1,u_2\}$ is a $2$-vertex-cut of $G$. Consequently, $u_1u_2\in E(G)$, and $G$ is in fact the triangular prism $\overline{C_6}$; in particular, \distg{G}{e_1}{e_2}~$=1$ and this contradicts our choice of $G$.

        Now suppose that at least one of \distg{G}{f_2}{e_1} and \distg{G}{f_1}{e_2} equals one, and adjust notation so that \distg{G}{f_2}{e_1}~$=1$. Let $x_2$ denote the end of $f_2$ that is distinct from $w_2$. Observe that, since \distg{G}{f_2}{e_1}~$=1$ and $f_2$ is distinct from $f_1$, the vertex $x_2\notin V(T_1)$ and it is adjacent with at least one of $u_1$ and $v_1$; adjust notation so that $x_2v_1\in E(G)$. Now we consider the following two subcases.\\

        First suppose that \distg{G}{f_1}{e_2} $=0$. Adjust notation so that $f_1:=w_1v_2$. We let $S:=V(T_1)\cup V(T_2)\cup \{x_2\}$. Observe that $\partial(S)$ is a (possibly trivial) $3$-cut. The reader may note that the graph $G/\overline{S}$ is the Bicorn~$R_8$, and that the edge $e_1$ is not solitary in $G/\overline{S}$. Consequently, by Lemma~\ref{lem:solitary-ME-sepcuts}~(i), $e_1$ is not solitary in $G$, contrary to our choice of $e_1$.

        Finally, suppose that \distg{G}{f_1}{e_2} $=1$. Let $x_1$ denote the end of $f_1$ that is distinct from~$w_1$. Observe that, since \distg{G}{f_1}{e_2}~$=1$ and $G$ is \mbox{$3$-regular}, the vertex $x_1\notin V(T_2)\cup \{x_2\}$ and it is adjacent with at least one of $u_2$ and $v_2$; adjust notation so that $x_1v_2\in E(G)$. We let \mbox{$S:=V(T_1)\cup V(T_2)\cup \{x_1,x_2\}$}. Observe that, since $G$ is $3$-connected, either $S=V(G)$ or otherwise $\partial(S)$ is a $4$-cut. In either case, the subgraph $G[\overline{S}]$ is matchable; see Lemma~\ref{lem:conformal-lemma}. We let $M$ denote a \pmg\ of $G[\overline{S}]$. Note that $M\cup \{e_1,e_2,f_1,f_2\}$ is a \pmg\ of $G$, contrary to our choice of $e_1$ and $e_2$ being mutually exclusive.
\end{proof}

The following is the main result of this section, and is an immediate consequence of the above theorem along with Lemmas~\ref{lem:r-graphs-new-lemma}~and~\ref{lem:combined-stat-r-graphs}.
\begin{cor}\label{cor:dist-at-most-one-in-r-graphs}
    In a \ecc{3} $r$-graph, any two mutually exclusive \ses\ are at distance at most one; consequently, any two solitary classes are at distance at most one. \qed
\end{cor}

The above, combined with Corollary~\ref{cor:main-thm}~(ii), yields the following.

\begin{cor}\label{cor:dist-at-most-3-unless-pattern(2)}
In a \ecc{3} $r$-graph~$G$, any two solitary edges are at distance at most three unless $G$ has solitude pattern~$(2)$. \qed
\end{cor}

We remark that there exist \ecc{3} $r$-graphs, for each value of $r$, that have two solitary edges at arbitrarily large distance. See Figure~\ref{fig:a-(2)-3cc-graph} for an example with $r=3$ and distance greater than three; one may multiply appropriately chosen \pmg s to obtain examples with $r\ge 4$.
It follows from the above that every such graph has solitude pattern~$(2)$; we provide a recursive description of this class for $r=3$ in Section~\ref{sec:SP2}.
Our next goal is to prove a strengthening of Theorem~\ref{thm:distance-at-most-1} when one of the two solitary edges is a solitary singleton.
However, to do so, we need another technical lemma regarding triangles.

\begin{lem}\label{lem:sd-lemma}
If a \cc{3} $G$ has a triangle~$T$ containing a solitary singleton~$e$, then the unique edge $f\in \partial(T)$ that is nonadjacent with~$e$ is a solitary singleton in~$G':=G/(T\rightarrow t)$.
\end{lem}
\begin{proof}
We let $T:=uvwu$ and $e:=uv$; thus $f\in \partial(w)$.
If the $3$-cut $\partial(T)$ is trivial then $G$ is~$K_4$ and devoid of solitary singletons, contrary to our hypothesis. Consequently, $\partial(T)$ is nontrivial and, by Proposition~\ref{prp:3-cut-matching-laminar}, a matching. By Proposition~\ref{prp:r-cuts}, we conclude that the smaller graph $G'$ is also $3$-connected. By Lemma~\ref{lem:solitary-ME-sepcuts}~(iii), the edge~$f$ is solitary in~$G'$. If $f$ is a solitary singleton in $G'$ then we are done. Otherwise, by Corollary~\ref{cor:main-thm}~(ii), the edge~$f$ participates in a solitary doubleton in~$G'$, say $\{d,f\}$. Our aim is to arrive at a contradiction.

By Corollary~\ref{cor:solitary-doubleton-near-bipartite-brick}~(ii), we infer that $H[A,B]:=G'-d-f$ is bipartite and matching covered; furthermore, we adjust notation so that ends of $f$ lie $A$ and ends of $d$ lie in $B$. We invite the reader to observe that $G-d-e$ is also bipartite with equicardinal color classes $A-t+u+v$ and $B+w$. Consequently, $d\xleftrightarrow{G} e$; however, this contradicts our hypothesis that $e$ is a solitary singleton in~$G$, and completes the proof.
\end{proof}

We are now ready to state and prove the aforementioned stronger result.

\begin{cor}\label{thm:distance-one}
In a \cc{3} distinct from $\theta$, any two solitary edges, of which at least one is a solitary singleton, are at distance one.
\end{cor}
\begin{proof}
    We proceed by contradiction, and we let $G$ denote a smallest counterexample; we adjust notation so that $e_1$ is a solitary singleton. By Theorem~\ref{thm:distance-at-most-1}, \distg{G}{e_1}{e_2}~$=0$. 
By Corollary~\ref{lem:triangle-lemma}, there exists a triangle $T:=e_1e_2d_1$. Note that if the $3$-cut $\partial(T)$ is trivial then $G$ is~$K_4$ and thus devoid of solitary singletons, contrary to our hypothesis. Thus, $\partial(T)$ is nontrivial and the \mbox{$3$-regular} graph $H:=G/T$ is not $\theta$. By Proposition~\ref{prp:r-cuts}, $H$ is $3$-connected. We intend to argue that $H$ is a smaller counterexample.

Now, let $f_i$ denote the companion of $e_i$ for each $i\in \{1,2\}$, and let us make some observations. Firstly, since $e_i\xrightarrow{G} f_i$, by Lemma~\ref{lem:solitary-ME-sepcuts}~(iii), we conclude that $f_i$ is solitary in~$H$.
Secondly, by Corollary~\ref{cor:triangle-companion}, we infer that $f_i$ is the unique edge of $\partial_{G}(T)$ that is nonadjacent with $e_i$. Thus, in $H$, each $f_i$ is incident with the contraction vertex. Consequently, \distg{H}{f_1}{f_2}~$=0$; also, $f_1$ and $f_2$ are mutually exclusive in~$H$. Finally, since $e_1$ is a solitary singleton in~$G$, we invoke Lemma~\ref{lem:sd-lemma} to conclude that $f_1$ is a solitary singleton in~$H$. Thus $H$ is a smaller counterexample contrary to our choice of~$G$, and this completes the proof.
\end{proof}

\begin{cor}\label{thm:distance-one-r-graphs}
In a \ecc{3} $r$-graph of order four or more, any two solitary edges, of which at least one is a solitary singleton, are at distance one.
\end{cor}
\begin{proof}
We let $e_1$~and~$e_2$ denote distinct solitary edges in a \ecc{3} \mbox{$r$-graph}~$G$, of order four or more, such that at least one of them is a solitary singleton. By Proposition~\ref{prp:solitary-patterns-of-small-r-graphs}, order of~$G$ is at least six.
By the first part of Lemma~\ref{lem:combined-stat-r-graphs}, $G$ is \mbox{$r$-edge-colorable}; we let $(M_1, M_2, M_3, \dots, M_r)$ denote a proper \mbox{$r$-edge-coloring} wherein $e_1\in M_1$ and $e_2\in M_2$. Since $e_1$~and~$e_2$ are mutually exclusive, by Lemma~\ref{lem:r-graphs-new-lemma}, we infer that for each $i\in \{3,\dots,r\}$, the (spanning) subgraph $G_i:=M_1\cup M_2\cup M_i$ is a $3$-connected (\mbox{$3$-regular}) graph. Furthermore, $e_1$~and~$e_2$ are mutually exclusive solitary edges in~$G_i$. By Corollary~\ref{cor:main-thm}, in each~$G_i$, the edge $e_1$ is either a solitary singleton or otherwise participates in a solitary doubleton; an analogous statement holds for~$e_2$.

Observe that if \distg{G_i}{e_1}{e_2}~$= 1$ for some $i\in \{3,\dots,r\}$ then \distg{G}{e_1}{e_2}~$= 1$, and we are done.
Consequently, by invoking Corollary~\ref{thm:distance-one}, it remains to consider the case in which both $e_1$~and~$e_2$ participate in (distinct) solitary doubletons in each~$G_i$. By Theorem~\ref{thm:more-than-one-SD}, we infer that every~$G_i$ is a staircase. Since any two solitary edges in a staircase are at distance at most one, it remains to consider the case in which \distg{G_i}{e_1}{e_2}~$= 0$ for each $i\in \{3,\dots,r\}$. It follows from the definition of staircases that $e_1$~and~$e_2$ participate in a triangle in~$G_i$ for each $i\in \{3,\dots,r\}$; furthermore, the perfect matching~$M_i$ contains the third edge of this triangle. Now, by Proposition~\ref{prp:staircase-uniquely-determined}, we conclude that $G$ is a staircase of thickness~$r-2$, and both $e_1$~and~$e_2$ participate in (distinct) solitary doubletons in~$G$; a contradiction.
\end{proof}

In the following section, we shall rely heavily on results of this section as well as Section~\ref{sec:introduction} to establish our characterizations of the solitude patterns claimed in Table~\ref{table1}.

\section{Forward implications: \newline\texorpdfstring{\ecc{3}}{} \texorpdfstring{$r$}{}-graphs with three or more solitary edges}
\label{sec:proofs-of-main-results}
In this section, our main goal is to prove the forward implications of Theorems~\ref{thm:more-than-one-SD}, \ref{thm:SDSS}~and~\ref{thm:3SS}; their proofs appear in Sections \ref{sec:more than one solitary doubleton}, \ref{sec:oneSD-oneSS} and \ref{sec:3SS}, respectively. We remark that their proofs do not rely on each other. In Section~\ref{sec:SP2}, we provide a recursive description of \cc{3}s that have solitude pattern~$(2)$; its proof relies on Theorems~\ref{thm:more-than-one-SD} as well as \ref{thm:SDSS}.
\subsection{Two solitary doubletons}\label{sec:more than one solitary doubleton}

In order to prove that every \ecc{3} $r$-graph, of order six or more, with two or more solitary doubletons, is a $1$-staircase of thickness~$r-2$, we first establish the existence of a $1$-dumbbell of thickness~$r-2$. First, we do this for $r=3$.

\begin{lem}\label{lem:1-dumbbell-existence}
If $\{\alpha,\beta\}$ and $\{\alpha',\beta'\}$ are distinct solitary doubletons in a \cc{3} $G$, distinct from $K_4$, then $G$ has a subgraph $J$ that is a $1$-\dumbbell, and whose unique odd $1$-cut is adjacent with each of the four edges in $\{\alpha,\beta,\alpha',\beta'\}$.
\end{lem}
\begin{proof}
Using Corollaries~\ref{lem:vertex-disjoint-triangle-lemma}~and~\ref{cor:solitary-doubleton-near-bipartite-brick}~(ii), we conclude that $G$ has vertex-disjoint triangles $T_1$~and~$T_2$ such that $\alpha,\alpha'\in E(T_1)$ whereas $\beta,\beta'\in E(T_2)$. We may adjust notation so that $\alpha:=aa_1, \alpha':=aa_2,\beta:=bb_1$ and $\beta':=bb_2$.

Since $T_1$ and $T_2$ are vertex-disjoint, by Theorem~\ref{thm:distance-at-most-1}, ${\sf dist}_G(\alpha,\beta')=1$ and ${\sf dist}_G(\alpha',\beta)=1$. As $G$ is \mbox{$3$-regular}, $\alpha$ depends on the unique edge in $\partial_{G}(T_1)\cap\partial_{G}(a_2)$. Since $\alpha\leftrightarrow \beta$, we infer that $a_2b_1,a_2b\notin E(G)$. An analogous argument implies that $a_1b_2,a_1b\notin E(G)$. Using all of these observations, we infer that $ab\in E(G)$. Thus $J:=T_1+T_2+ab$ is the desired $1$-\dumbbell.
\end{proof}


Now, we extend the above to all \ecc{3} $r$-graphs.

\begin{lem}\label{lem:r1-dumbbell-existence-r-graphs}
If $\{\alpha_1,\beta_1\}$~and~$\{\alpha_2,\beta_2\}$ are distinct solitary doubletons in a \ecc{3} \mbox{$r$-graph}~$G$, of order six or more, then $G$ has a subgraph~$J$ that is a $1$-\dumbbell\ of order six and thickness~$r-2$; furthermore, $J$ contains $\alpha_1,\beta_1,\alpha_2$~and~$\beta_2$ as its (only) unmatchable edges.
\end{lem}
\begin{proof}
By the first part of Lemma~\ref{lem:combined-stat-r-graphs}, $G$ is $r$-edge-colorable; we let $(M_1, M_2, M_3, \dots, M_r)$ denote a proper \mbox{$r$-edge-coloring} wherein $\{\alpha_i, \beta_i \}\subset M_i$ for each $i\in \{1,2\}$. Since $\alpha_1$~and~$\alpha_2$ are mutually exclusive, the following is an immediate consequence of Lemma~\ref{lem:r-graphs-new-lemma}.
\begin{sta}
For each $i\in \{3,\dots,r\}$, the (spanning) subgraph $H_i:=M_1\cup M_2\cup M_i$ is a $3$-connected (\mbox{$3$-regular}) graph. \qed
\end{sta}
It is easy to see that $\alpha_1,\beta_1,\alpha_2$~and~$\beta_2$ are solitary edges of~$H_i$, and that $\alpha_1 \xleftrightarrow{H_i} \beta_1$ and $\alpha_2 \xleftrightarrow{H_i} \beta_2$ for each $i\in \{3,\dots,r\}$; by Corollary~\ref{cor:main-thm}~(ii), $\{\alpha_1, \beta_1\}$ and $\{\alpha_2, \beta_2\}$ are solitary doubletons of~$H_i$.
We now invoke Lemma~\ref{lem:1-dumbbell-existence} to deduce the following.
\begin{sta}
For each $i\in \{3,\dots,r\}$, the graph~$H_i$ has a subgraph~$J_i$ that is a $1$-\dumbbell\ of order six and thickness one; furthermore, $J_i$ contains $\alpha_1,\beta_1,\alpha_2$~and~$\beta_2$ as its (only) unmatchable edges, or equivalently, its unique perfect matching is a subset of~$M_i$. \qed
\end{sta}

We invoke the above for $i=3$, and adjust notation so that $\alpha_1$~and~$\alpha_2$ are adjacent; consequently, $\beta_1$~and~$\beta_2$ are adjacent. The reader may now invoke the above for the remaining values of~$i$ to infer that the subgraph~$J$ formed by the union of $J_3$ up to~$J_r$ is indeed the desired $1$-\dumbbell.
\end{proof}

A subgraph~$H$ of a graph~$G$ is {\em conformal} if $G-V(H)$ is matchable; note that every spanning subgraph is conformal. Furthermore, for a \pmg~$M$ of a graph~$G$, a (matchable) subgraph~$H$ of~$G$ is {\em $M$-conformal} if $M\cap E(H)$ is a \pmg\ of~$H$.

We are now ready to prove the forward implication of Theorem~\ref{thm:more-than-one-SD}. In its proof, we shall invoke the above lemma and consider a maximal $1$-dumbbell with certain properties and then use results from earlier sections to conclude that the chosen $1$-dumbbell is a spanning subgraph.

\begin{thm}\label{thm:(2,2)-graphs-r-graphs}
Every \ecc{3} $r$-graph, of order six or more, with at least two solitary doubletons is a $1$-staircase of thickness~$r-2$.
\end{thm}
\begin{proof}
Let $G$ be a \ecc{3} $r$-graph, of order six or more, and let $\{\alpha,\beta\}$~and~$\{\alpha',\beta'\}$ denote distinct solitary doubletons. By Lemma~\ref{lem:r1-dumbbell-existence-r-graphs}, $G$ has a subgraph that is a \mbox{$1$-\dumbbell}\ of thickness~$r-2$ whose (only) unmatchable edges are $\alpha,\beta,\alpha'$~and~$\beta'$; let $J$ denote a maximal $1$-\dumbbell\ in~$G$ with these properties. Let $u_1, w_1, u_2$~and~$w_2$ denote the four vertices of degree~$r-1$ in~$J$ so that $u_1$~and~$w_1$ are adjacent in~$J$.

\begin{sta}\label{sta:no-edge-joining-peripheral-rungs-r-graphs}
If $G$ has any edge that has one end in $\{u_1,w_1\}$ and the other end in $\{u_2,w_2\}$ then $G$ is a $1$-staircase of thickness~$r-2$.
\end{sta}
\begin{proof}
Adjust notation so that $u_1w_2\in E(G)$. Since $G$ is $3$-edge-connected, \mbox{$V(G)=V(J)$} and $E(G)=E(J)+u_1w_2+u_2w_1$; whence $G$ is a $1$-staircase of thickness~$r-2$.
\end{proof}

Henceforth, we assume that $G$ has no edge with one end in $\{u_1,w_1\}$ and the other end in $\{u_2,w_2\}$. Our aim is to arrive at a contradiction.

We let $a_1$ and $b_1$ denote the (only) cut-vertices of~$J$ so that $a_1$~and~$u_1$ lie in the same block. Note that the four edges in~$\partial_J(\{a_1,b_1\})$ are the (only) unmatchable edges of~$J$; consequently, these are $\alpha, \beta, \alpha'$ and $\beta'$.
Since $\{\alpha,\beta\}$ is a matching, we let $\alpha:=a_1a_2$ and $\beta:=b_1b_2$. By Corollary~\ref{cor:solitary-doubleton-near-bipartite-brick}~(ii), the connected graph $G-\alpha-\beta$ is bipartite; we let $A'$ and $B'$ denote its color classes so that $a_1,a_2\in A'$ and $b_1,b_2\in B'$. Adjust notation so that $w_1,w_2\in A'$~and~$u_1,u_2\in B'$.

Now, let $M$ denote the unique \pmg\ of~$G$ that contains the edges $\alpha$~and~$\beta$, and let $d_1$ denote the $M$-edge incident at $u_1$. 
We let $d_1:=u_1w$. By our assumption stated after~\ref{sta:no-edge-joining-peripheral-rungs-r-graphs}, and by paying attention to the color classes of $G-\alpha-\beta$, it follows that \mbox{$w\in A'\cap \overline{V(J)}$} and that $J':=J+d_1-w_2$ is an $M$-conformal subgraph.

We let $H[A,B]:=G-a_1-a_2-b_1-b_2$ where $A:=A'-a_1-a_2$ and $B:=B'-b_1-b_2$. Since $N:=M-\alpha-\beta$ is the only \pmg\ of the bipartite graph~$H$, we infer that $N\cap E(G-V(J'))$ is the only \pmg\ of the (nonempty) bipartite graph~$G-V(J')$. By Corollary~\ref{cor:degree-one}, $G-V(J')$ has a degree one vertex in each of its color classes, and we let $b\in B$ denote a degree one vertex. It follows that, in~$G$, there are $r-1$ edges incident at~$b$ that have their other end in~$V(J')\cap A'$. By paying attention to the degrees of vertices in~$J'$, we conclude that either $\mu_{G}(b,w)=r-1$, or otherwise $\mu_{G}(b,w)=r-2$~and~$\mu_{G}(b,w_1)=1$. In the former case, $\partial_{G}(\{b,w\})$ is a $2$-cut; a contradiction since~$G$ is \mbox{$3$-edge-connected}. In the latter case, observe that $J+d_1+bw_1+G[b,w]$ is a (larger) \mbox{$1$-\dumbbell}\ of thickness~$r-2$ whose (only) unmatchable edges are $\alpha,\beta,\alpha'$~and~$\beta'$; this contradicts the maximality of~$J$.

This completes the proof of Theorem \ref{thm:(2,2)-graphs-r-graphs}.
\end{proof}

The above, along with our discussion in Section~\ref{sec:more-than-one-SD}, proves Theorem~\ref{thm:more-than-one-SD}.
In the next section, we shall follow the same recipe, albeit more difficult to execute, to prove the forward implication of Theorem~\ref{thm:SDSS}.

\subsection{One solitary singleton plus one solitary doubleton}\label{sec:oneSD-oneSS}

In order to prove that every \ecc{3} $r$-graph, with solitary classes of different cardinalities, is a $3$-staircase, we first establish the existence of a $3$-dumbbell for $r=3$.

\begin{lem}\label{lem:dumbbell-existence}
If $e$ is a solitary singleton and $\{\alpha,\beta\}$ is a solitary doubleton in a \cc{3} $G$, then $G$ has a subgraph $J$ that is a $3$-\dumbbell\ with $e$ as its unique even $1$-cut such that $e$ is at distance exactly one from each of $\alpha$~and~$\beta$ in~$J$.
\end{lem}
\begin{proof}
Since $G$ has a solitary singleton as well as a solitary doubleton, we infer that $G$ is neither $\theta$ nor~$K_4$, or equivalently, $G$ has order six or more. We invoke Corollary~\ref{lem:vertex-disjoint-triangle-lemma} to conclude that $G$ has two vertex-disjoint triangles, say $T_1$ and $T_2$, so that $\alpha:=a_1a_2$ lies in~$T_1$ and $\beta:=b_1b_2$ lies in~$T_2$.

By Corollary~\ref{lem:solitary-class-unique-pm}, the solitary classes $\{e\}$ and $\{\alpha,\beta\}$ are mutually exclusive. Thus, by Corollary~\ref{thm:distance-one}, \distg{G}{e}{\alpha}$=1=$\distg{G}{e}{\beta}. Furthermore, the edge $e$ is unmatchable in the matchable graph $H:=G-a_1-a_2-b_1-b_2$; this immediately implies that $e$ is not an odd $1$-cut in~$H$. It follows from these observations that $e\notin \partial_{G}(T_1)\cup \partial_{G}(T_2)$.

We let $e:=ab$. Since \distg{G}{e}{\alpha}$=1=$\distg{G}{e}{\beta}, at least one of $a$~and~$b$ is adjacent with an end of $\alpha$; likewise, at least one of them is adjacent with an end of $\beta$. Up to symmetry, there are two possibilities. Note that if a particular end of $e$ is adjacent with an end of each of $\alpha$ and $\beta$ then $e$ is an odd $1$-cut in $H$, contradicting our observation in the preceding paragraph. Consequently, up to relabeling, $a_1b,ab_1\in E(G)$. Observe that $T_1\cup T_2\cup P$, where $P$ is the path $a_1bab_1$, is the desired $3$-\dumbbell\ in~$G$.
\end{proof}

For \pmg s~$M$ and $N$ of a graph~$G$, a path or a cycle~$Q$ is \textit{$M$-$N$-alternating} if its edges, on traversal, alternate between being in~$M$ and being in~$N$.
Now, we extend the above to all \ecc{3} $r$-graphs.
\begin{lem}\label{lem:r3-dumbbell-existence-r-graphs}
If $e$ is a solitary singleton and $\{\alpha,\beta\}$ is a solitary doubleton in a \mbox{\ecc{3}} $r$-graph~$G$, then $G$ has a subgraph~$J$ that is a 
$3$-\dumbbell\ of order eight and thickness~$r-2$; furthermore, $J$ contains~$e$ as its (only) even $1$-cut, and two of its remaining (four) unmatchable edges are $\alpha$~and~$\beta$.
\end{lem}
\begin{proof}
Since $\{\alpha, \beta\}$ is a matching and $e$ is a solitary singleton in~$G$, by Proposition~\ref{prp:solitary-patterns-of-small-r-graphs}, we infer that $|V(G)|\ge 6$. We let $\alpha:=a_1a_2$, $\beta:=b_1b_2$ and $e:=ab$. We invoke Corollary~\ref{cor:solitary-doubleton-near-bipartite-brick}~(ii), and we let $H[A,B]$ denote the bipartite matching covered graph~$G-\alpha-\beta$ so that \mbox{$a_1,a_2,a\in A$} and $b_1,b_2,b\in B$. By the first part of Lemma~\ref{lem:combined-stat-r-graphs}, $G$ is \mbox{$r$-edge-colorable}; we let $(M_1, M_2, M_3, \dots, M_r)$ denote a proper \mbox{$r$-edge-coloring} wherein $e\in M_1$ and $\{\alpha, \beta \}\subset M_2$. Since $\alpha$~and~$e$ are mutually exclusive, the first part of the following is an immediate consequence of Lemma~\ref{lem:r-graphs-new-lemma}, whereas the second part relies on easy observations, Corollaries~\ref{cor:main-thm}~(ii)~and~\ref{cor:solitary-doubleton-near-bipartite-brick}~(ii).
\begin{sta}\label{sta:3cc}
For each $i\in \{3,\dots,r\}$, the (spanning) subgraph $G_i:=M_1\cup M_2\cup M_i$ is a $3$-connected (\mbox{$3$-regular}) graph. Furthermore, $\alpha,\beta$~and~$e$ are solitary edges and $\alpha \xleftrightarrow{G_i} \beta$; consequently, $\{\alpha, \beta\}$ is a solitary doubleton and $H_i[A,B]:=G_i-\alpha-\beta$ is a bipartite matching covered graph. \qed
\end{sta}

On the other hand, in any~$G_i$, the edge~$e$ may be a solitary singleton or it may participate in a solitary doubleton. Our next goal is to argue that $e$ is a solitary singleton in each~$G_i$. To this end, we shall find the following observation useful.

\begin{sta}\label{sta:conformal-four-cycle}
Suppose there exist $u,w\in V(G)$ and distinct $i,j\in \{3,\dots,r\}$ such that both $G_i$~and~$G_j$ contain $M_1$-alternating $uw$-paths of length two, say $Q_i$~and~$Q_j$, respectively. If neither $Q_i$ nor $Q_j$ contains~$e$, then their $M_1$-edges are the same, or equivalently, $V(Q_i)=V(Q_j)$. 
\end{sta}
\begin{proof}
Suppose not. Observe that $Q_i+ Q_j$ is an $M_1$-alternating $4$-cycle in~$G$ that does not contain~$e$; a contradiction to Proposition~\ref{prp:mcgs-hcycle}.
\end{proof}

\begin{sta}\label{sta:one-implies-many}
Suppose there exists $i\in \{3,\dots,r\}$ such that $e$ participates in a solitary doubleton in~$G_i$, and let~$e'$ denote its solitary partner. Then, $G_i$ contains two vertex-disjoint \mbox{$M_1$-$M_i$-alternating} paths of length two, one of which has ends $a_1$~and~$a_2$ whereas the other one has ends $b_1$~and~$b_2$; furthermore, one of them contains~$e$ whereas the other one contains~$e'$.
\end{sta}
\begin{proof}
Assume without loss of generality that $\{e,e'\}$ is a solitary doubleton in~$G_3$. Clearly, $e'\in M_1$. Since $\{\alpha, \beta\}$~and~$\{e,e'\}$ are solitary doubletons in~$G_3$, we invoke Corollary~\ref{lem:vertex-disjoint-triangle-lemma} twice, once for $\{\alpha, \beta\}$ and once for $\{e, e'\}$, 
to arrive at the desired conclusion.
\end{proof}

\begin{sta}\label{sta:same-solitary-partner}
If $e$ participates in a solitary doubleton in~$G_i$ for some $i\in\{3, \dots, r\}$, then the same holds for each $i\in \{3, \dots, r\}$; furthermore, the solitary partner of~$e$ is the same in every~$G_i$.
\end{sta}
\begin{proof}
Suppose that $e$ participates in a solitary doubleton in~$G_i$ for some $i\in \{3, \dots, r\}$. It follows from~\ref{sta:one-implies-many} that $e$ is adjacent with one of $\alpha$~and~$\beta$. Consequently, by~\ref{sta:3cc} and Corollary~\ref{thm:distance-one}, $e$ participates in a solitary doubleton in~$G_i$ for each $i\in\{3, \dots, r\}$.

Now, let $e'$ and $e''$ denote the solitary partners of~$e$ in~$G_3$ and $G_4$, respectively. For each $i\in \{3, 4\}$, by \ref{sta:one-implies-many}, $G_i$ contains two vertex-disjoint $M_1$-$M_i$-alternating paths of length two, say $Q_i$~and~$Q'_i$; adjust notation so that $Q_i$ is an $a_1a_2$-path and $Q'_i$ is a $b_1b_2$-path, and $Q_3$ contains~$e$ whereas $Q'_3$ contains~$e'$. Since $e$~and~$\beta$ are nonadjacent, we infer that $Q_4$ contains~$e$, whence $Q'_4$ contains~$e''$. We invoke \ref{sta:conformal-four-cycle} to conclude that $e'=e''$.
\end{proof}

\begin{sta}
For each $i\in \{3,\dots,r\}$, the edge~$e$ is a solitary singleton in the graph~$G_i$.
\end{sta}
\begin{proof}
Suppose not. By~\ref{sta:same-solitary-partner}, we conclude that $\{e,e'\}$ is a solitary doubleton in each~$G_i$ for a fixed $e'\in M_1$. We invoke Corollary~\ref{lem:vertex-disjoint-triangle-lemma} twice for the \cc{3}~$G_3$, once for $\{\alpha, \beta\}$ and once for $\{e,e'\}$, and adjust notation so that $a_1$ is the common end of $\alpha$~and~$e$, whereas $b_1$ is the common end of $\beta$~and~$e'$. Since the four edges $e,e'\in M_1$ and $\alpha, \beta\in M_2$ exist in $G_i$, for each $i\in \{3,\dots,r\}$, we invoke the last part of Lemma~\ref{lem:1-dumbbell-existence} to deduce that $G_i$ has an edge joining $a_1$~and~$b_1$ that lies in~$M_i$. Consequently, $\mu_{G}(a_1,b_1)=r-2$. Since $\alpha \xleftrightarrow{G} \beta$, we infer that $e\xleftrightarrow{G} e'$ and thus belong to the same equivalence class; however, this contradicts the hypothesis that $e$ is a solitary singleton in~$G$.
\end{proof}

In summary, for each $i\in \{3,\dots,r\}$, the graph~$G_i=M_1\cup M_2\cup M_i$ is a \cc{3} with $e$ as a solitary singleton and $\{\alpha, \beta\}$ as a solitary doubleton. We now invoke Lemma~\ref{lem:dumbbell-existence} to deduce the first part of the following; the second part is easy to see by paying attention to the colors.
\begin{sta}\label{sta:cor}
For each $i\in \{3,\dots,r \}$, the graph~$G_i$ has a subgraph~$J_i$ that is a $3$-\dumbbell\ of order eight and thickness one; furthermore, $J_i$ contains~$e$ as its (only) even $1$-cut, and two of its remaining (four) unmatchable edges are $\alpha$~and~$\beta$. Also, the unique perfect matching of~$J_i$ is a subset of~$M_i$. \qed
\end{sta}

 Now, we invoke the above for $i=3$, and adjust notation so that $a_2a_1bab_1b_2$ is a path in~$J_3$; we let $u$~and~$w$ denote the
common neighbours of $a_1$~and~$a_2$, and of $b_1$~and~$b_2$, respectively. Observe that $a_1u\in M_1$, whence $Q_3:=a_1ua_2$ is an $M_1$-$M_3$-alternating $a_1a_2-$path in~$J_3$.
\begin{sta}
For each $i\in \{3,\dots,r\}$, in~$J_i$, the vertex~$u$ is the
common neighbour of $a_1$~and~$a_2$, whereas $w$ is the common neighbour of $b_1$~and~$b_2$; furthermore, $a_2a_1bab_1b_2$ is a path in~$J_i$.
\end{sta}
\begin{proof}
Clearly, we only need to argue for $i\in \{4,\dots,r\}$; by symmetry, it suffices to prove for $i=4$. To this end, in~$J_4$, let $u'$~and~$w'$ denote the common neighbours of $a_1$~and~$a_2$, and of $b_1$~and~$b_2$, respectively. Note that $Q_4:=a_1u'a_2$ is an $M_1$-$M_4$-alternating $a_1a_2$-path in~$J_4$. Since neither $Q_3$ nor $Q_4$ contains~$e$, we invoke \ref{sta:conformal-four-cycle} to conclude that $u=u'$. An analogous argument shows that $w=w'$. This proves the first part.

By \ref{sta:cor}, $e$ is the (only) even $1$-cut of~$J_4$, one triangle of~$J_4$ contains~$\alpha$, whereas the other one contains~$\beta$; also, the unique perfect matching of~$J_4$ is a subset of~$M_4$. Since $a_1u\in M_1$, these facts are enough to deduce that $a_1$ is in fact the cubic end of $\alpha$ in~$J_4$; an analogous argument shows that $b_1$ is the cubic end of~$\beta$.
Finally, since $H_4[A,B]=G_4-\alpha-\beta$ is bipartite, we infer that $a_2a_1bab_1b_2$ is a path in~$J_4$.
\end{proof}

We invoke \ref{sta:cor} for all values of~$i$ to infer that the subgraph~$J$ formed by the union of $J_3$ up to~$J_r$ is indeed the desired $3$-\dumbbell\ of order eight and thickness~$r-2$.
\end{proof}

We are now ready to prove the forward implication of Theorem~\ref{thm:SDSS}. In its proof, we shall invoke the above lemma and consider a maximal $3$-dumbbell with certain properties and then use results from earlier sections to conclude that the chosen $3$-dumbbell is a spanning subgraph.

\begin{thm}\label{thm:(2,1)-graphs-r-graphs}
Every \ecc{3} $r$-graph, with at least one solitary singleton and at least one solitary doubleton, is a $3$-staircase of thickness~$r-2$.
\end{thm} 
\begin{proof}
Let~$G$ be a \ecc{3} $r$-graph, and let $e:=ab$ denote a solitary singleton and $\{\alpha,\beta\}$ denote a solitary doubleton. By Lemma~\ref{lem:r3-dumbbell-existence-r-graphs}, $G$ has a subgraph that is a \mbox{$3$-\dumbbell} of thickness~$r-2$ whose (only) even~$1$-cut is~$e$, and two of its remaining (four) unmatchable edges are $\alpha$~and~$\beta$; let $J$ denote a maximal \mbox{$3$-\dumbbell} in~$G$ with these properties. Let $u_1, w_1, u_2$~and~$w_2$ denote the four vertices of degree~$r-1$ in~$J$, distinct from the ends of~$e$, so that $u_1$~and~$w_1$ are adjacent in~$J$.

We let $a_1$~and~$b_1$ denote the cut-vertices of~$J$, distinct from the ends of~$e$, so that $a_1$~and~$u_1$ lie in the same block. Since $\{\alpha,\beta\}$ is a matching, we let $\alpha:=a_1a_2$ and $\beta:=b_1b_2$, and we adjust notation so that $a_2a_1bab_1b_2$ is a path in~$J$. By Corollary~\ref{cor:solitary-doubleton-near-bipartite-brick}~(ii), 
the connected graph $G-\alpha-\beta$ is bipartite; we let $A'$~and~$B'$ denote its color classes so that $a_1,a_2\in A'$ and $b_1,b_2\in B'$, consequently, $a\in A'$ and $b\in B'$. Adjust notation so that $w_1,w_2\in A'$~and~$u_1,u_2\in B'$.

Now, let $M$ denote the unique \pmg\ of~$G$ that contains the edges $\alpha$~and~$\beta$, and let $d_1:=u_1w$~and~$d_2:=w_2u$ denote the (not necessarily distinct) $M$-edges incident at $u_1$~and~$w_2$, respectively. Clearly, $w\in A'$ and~$u\in B'$. Furthermore, by paying attention to the color classes, we infer that $d_1,d_2\notin E(J)$. Also, $e\notin M$ since $e$ is a solitary singleton and thus mutually exclusive with $\{\alpha,\beta\}$. Observe that if $w=a$~and~$u=b$ then, since $G$ is \mbox{\ecc{3}}, $w_1u_2\in E(G)$ and $G$ is indeed a $3$-staircase of thickness~$r-2$.

Henceforth, either $w\ne a$ or $u\ne b$, possibly both. Consequently, we invite the reader to observe that either $d_1=d_2$, or otherwise at least one of $w$~and~$u$ is not in~$V(J)$; in the latter case, by symmetry, we may assume without loss of generality that $w\notin V(J)$. In each case, we intend to arrive at a contradiction. To this end, we shall find the $M$-conformal subgraph $J':=(J+d_1-a-b)+d_2$ quite useful. It is worth noting that, in each case, $a\notin V(J')$. Consequently, using the fact that $e\notin M$, we infer that if $b\notin V(J')$ then $b$ has degree at least two in $G-V(J')$.

We let $H[A,B]:=G-a_1-a_2-b_1-b_2$ where $A:=A'-a_1-a_2$ and $B:=B'-b_1-b_2$. Since $N:=M-\alpha-\beta$ is the only \pmg\ of~$H$, we infer that $N\cap E(G-V(J'))$ is the only \pmg\ of the (nonempty) bipartite graph~$G-V(J')$. By Corollary~\ref{cor:degree-one}, $G-V(J')$ has a degree one vertex in each of its color classes, and we let $y\in B$ denote such a vertex. It follows from the preceding paragraph that $y\ne b$; consequently, $a_1y\notin E(G)$. We now consider the two cases --- either $d_1=d_2$, or otherwise $w\notin V(J)$ --- separately.

First suppose that $d_1=d_2$; consequently, $w=w_2$~and~$u=u_1$. Note that, in~$J'$, the degree of each vertex in~$A'-a_1-w_1$ is~$r$. Since $a_1y\notin E(G)$ and degree of~$w_1$ in~$J'$ is~$r-1$, the vertex~$y$ has degree at most two in~$G$; a contradiction.

Now suppose that $w\notin V(J)$. By paying attention to the degrees of vertices in~$J'$ and using the fact that $y\ne b$, we conclude that either $\mu_{G}(y,w)=r-1$, or otherwise $\mu_{G}(y,w)=r-2$~and~$\mu_{G}(y,w_1)=1$. In the former case, $\partial_{G}(\{y,w\})$ is a $2$-cut; a contradiction since $G$ is $3$-edge-connected. In the latter case, observe that $J+d_1+yw_1+G[y,w]$ is a (larger) $3$-\dumbbell\ of thickness~$r-2$ whose (only) even~$1$-cut is~$e$, and two of its remaining (four) unmatchable edges are $\alpha$~and~$\beta$; this contradicts the maximality of~$J$.

This completes the proof of Theorem \ref{thm:(2,1)-graphs-r-graphs}.
\end{proof}

The above, along with our discussion in Section~\ref{sec:SDSS}, proves Theorem~\ref{thm:SDSS}. In the following section, we shall use both Theorems~\ref{thm:more-than-one-SD}~and~\ref{thm:SDSS}.

\subsection{The recursive family \texorpdfstring{$\mathcal{D}$}{}: 
 \texorpdfstring{\cc{3}s}{} with solitude pattern \texorpdfstring{$(2)$}{}}\label{sec:SP2}
We shall find it useful to consider a \cc{3}~$G$ with a solitary doubleton~$\{\alpha,\beta\}$ and analyze the solitude pattern of the graph~$G'$ that is obtained from~$G$ by splicing with~$K_4$ at a vertex~$v$ incident with one of $\alpha$~and~$\beta$; for convenience, we say that $v$ is {\em incident with }$\{\alpha,\beta\}$.

For a vertex $v$ and edge $e$ of a graph~$G$, we use \distg{G}{v}{e} to denote the length of a shortest path that starts at~$v$ and ends at any end of~$e$. For instance, recall the definition of $3$-staircases from Section~\ref{sec:SDSS}, and observe the following. If $G$ is any $3$-staircase of order twelve or more, and $\{\alpha,\beta\}$ is its solitary doubleton, then precisely one of the ends of $\alpha$ is at distance one from its solitary singleton~$e$ whereas the other end is at distance two from~$e$; an analogous statement holds for~$\beta$; see Figure~\ref{fig:a-(2,1)-3cc-graph} for an example.
The following proposition deals with the solitude pattern of the resultant \cc{3} when~$G$ is spliced with~$K_4$ at an end $\alpha$ or of~$\beta$.
\begin{prp}\label{prp:two-one-and-two-from-two-one}
Let $G$ be a $3$-staircase of order twelve or more, let $e$ denote its solitary singleton, and let $v$ denote any vertex that is incident with its unique solitary doubleton~\mbox{$\{\alpha,\beta\}$}. The following statements hold for $G':=(G\odot K_4)_v$:
\begin{enumerate}[(i)]
\item if \distg{G}{v}{e}~$=1$ then $G'$ is a larger $3$-staircase, whereas
\item if \distg{G}{v}{e}~$=2$ then $G'$ has solitude pattern~$(2)$.
\end{enumerate}
\end{prp}
\begin{proof}
We invite the reader to observe~$(i)$. Now, suppose that \distg{G}{v}{e}~$=2$, and
adjust notation so that $v$ is an end of~$\alpha$. Let $T$ denote the triangle of $G'$ formed by the edges in~$E(G')-E(G)$, and let $\alpha'$ denote its unique edge that is nonadjacent with~$\alpha$. By Lemma~\ref{lem:solitary-ME-sepcuts}~$(i)$, each edge in $E(G')-E(T)-e-\alpha-\beta$ is not solitary in~$G'$; we proceed to examine the remaining six edges.  Since $\alpha\in \partial_{G'}(T)$, by Corollary~\ref{cor:solitary-leaving-triangle-implies-K4}, $\alpha$ is not solitary in~$G'$. Since $\alpha$ is the only solitary edge of~$G$ in~$\partial_{G}(v)$, by Proposition~\ref{prp:solitary-splice-with-K4}, the edges in $E(T)-\alpha'$ are not solitary in~$G'$. Also, by Corollary~\ref{cor:splicing-at-a-SD-vertex}, $\{\alpha',\beta\}$ is a solitary doubleton in~$G'$. Finally, since \distg{G}{v}{e}~$=2$, we infer that \distg{G'}{\alpha'}{e}~$\ge 2$; by Theorem~\ref{thm:distance-at-most-1}, $e$~is not solitary in~$G'$. Consequently, $\{\alpha',\beta\}$ is the only solitary class in~$G'$.
%
\end{proof}

\begin{figure}[!htb]
    \centering
   \begin{subfigure}[b]{.5\textwidth}
       \centering
        \begin{tikzpicture}[scale=0.8]
\node [circle,fill=white] (0) at (1, 1) {};
		\node [circle,fill=white] (1) at (1, 0) {};
		\node [circle,fill=white] (2) at (2, 1) {};
		\node [circle,fill=white] (3) at (2, 0) {};
		\node [circle,fill=white] (4) at (3, 0.5) {};
		\node [circle,fill=white] (5) at (4, 0.5) {};
		\node [circle,fill=white] (6) at (5, 0.5) {};
		\node [circle,fill=white] (7) at (6, 0.5) {};
		\node [circle,fill=white] (8) at (7, 1) {};
		\node [circle,fill=white] (9) at (7, 0) {};
		\node [circle,fill=white] (10) at (8, 1) {};
		\node [circle,fill=white] (11) at (8, 0) {};
		\node [circle,fill=white] (12) at (0, 1) {};
		\node [circle,fill=white] (13) at (0, 0) {};

\draw (0) to (2);
		\draw (2) to (3);
		\draw (3) to (1);
		\draw (1) to (0);
		\draw (2) to (4);
		\draw[color=red] (4) to (3);
		\draw (4) to (5);
		\draw[color=blue] (5) to (6);
		\draw (6) to (7);
		\draw[color=red] (7) to (8);
		\draw (7) to (9);
		\draw (9) to (8);
		\draw (8) to (10);
		\draw (10) to (11);
		\draw (11) to (9);
		\draw (12) to (0);
		\draw (12) to (13);
		\draw (13) to (1);
		\draw [bend left] (12) to (6);
		\draw [bend left=45] (5) to (10);
		\draw [bend right=15] (13) to (11);
        \end{tikzpicture}
         \caption{A $3$-staircase}
         \label{fig:a-(2,1)-3cc-graph}
   \end{subfigure}
       \begin{subfigure}[b]{.6\textwidth}
       \centering
       \begin{tikzpicture}[scale=0.8]
           \node [circle,fill=white] (0) at (1, 1) {};
		\node [circle,fill=white] (1) at (1, 0) {};
		\node [circle,fill=white] (2) at (2, 1) {};
		\node [circle,fill=white] (3) at (2, 0) {};
		\node [circle,fill=white] (4) at (3, 0.5) {};
		\node [circle,fill=white] (5) at (4, 0.5) {};
		\node [circle,fill=white] (6) at (5, 0.5) {};
		\node [circle,fill=white] (7) at (6, 0.5) {};
		\node [circle,fill=white] (8) at (7, 1) {};
		\node [circle,fill=white] (9) at (7, 0) {};
		\node [circle,fill=white] (10) at (8, 1) {};
		\node [circle,fill=white] (11) at (8, 0) {};
		\node [circle,fill=white] (12) at (0, 1) {};
		\node [circle,fill=white] (13) at (0, 0) {};
		\node [circle,fill=white] (14) at (2, 0.5) {};
		\node [circle,fill=white] (15) at (2.5, 0.25) {};

\draw (0) to (2);
		\draw (3) to (1);
		\draw (1) to (0);
		\draw (2) to (4);
		\draw (4) to (5);
		\draw (5) to (6);
		\draw (6) to (7);
		\draw[color=red] (7) to (8);
		\draw (7) to (9);
		\draw (9) to (8);
		\draw (8) to (10);
		\draw (10) to (11);
		\draw (11) to (9);
		\draw (12) to (0);
		\draw (12) to (13);
		\draw (13) to (1);
		\draw [bend left] (12) to (6);
		\draw [bend left=45] (5) to (10);
		\draw [bend right=15] (13) to (11);
		\draw (14) to (15);
		\draw (2) to (14);
		\draw[color=red] (14) to (3);
		\draw (3) to (15);
		\draw (15) to (4);
       \end{tikzpicture}
       \caption{A \cc{3} with solitude pattern~$(2)$}
         \label{fig:a-(2)-3cc-graph}
   \end{subfigure}
   \caption{An illustration for Proposition~\ref{prp:two-one-and-two-from-two-one}}
   \label{fig:distance-two}
\end{figure}

Figure~\ref{fig:a-(2)-3cc-graph} shows an example of a graph~$G'$ that is obtained from the graph~$G$ shown in Figure~\ref{fig:a-(2,1)-3cc-graph} by splicing as per statement~(ii) of the above. Inspired by Proposition~\ref{prp:two-one-and-two-from-two-one}, we define the family $\mathcal{D}_0$ as follows: $G'\in \mathcal{D}_0$ if and only if $G':=(G\odot K_4)_v$, where $G$ is any $3$-staircase of order twelve or more, and $v$ is incident with its solitary doubleton and is at distance two from its solitary singleton. The following is a restatement of~(ii) of the above proposition.

\begin{cor}\label{cor:family-D0}
Each member of $\mathcal{D}_0$ is a $3$-connected \mbox{$3$-regular} graph that has solitude pattern~$(2)$.~\qed
\end{cor}

Next, we consider the case in which $G$ is any \cc{3} with solitude pattern~$(2)$.
\begin{prp}\label{prp:two-two-from-two-two}
Let $G$ be a \cc{3} that has solitude pattern~$(2)$, and let $v$ denote any vertex that is incident with its unique solitary doubleton~$\{\alpha,\beta\}$. Then, the solitude pattern of the \cc{3}~$G':=(G\odot K_4)_v$ is also~$(2)$.
\end{prp}
\begin{proof}
Adjust notation so that $v$ is an end of~$\alpha$. Let $T$ denote the triangle of $G'$ formed by the edges in~$E(G')-E(G)$, and let $\alpha'$ denote its unique edge that is nonadjacent with~$\alpha$. By Lemma~\ref{lem:solitary-ME-sepcuts}~$(i)$, each edge in $E(G')-E(T)-\alpha-\beta$ is not solitary in~$G'$; we proceed to examine the remaining five edges. Since $\alpha\in \partial_{G'}(T)$, by Corollary~\ref{cor:solitary-leaving-triangle-implies-K4}, $\alpha$ is not solitary in~$G'$. Since $\alpha$ is the only solitary edge of~$G$ in~$\partial_{G}(v)$, by Proposition~\ref{prp:solitary-splice-with-K4}, the edges in $E(T)-\alpha'$ are not solitary in~$G'$. Finally, by Corollary~\ref{cor:splicing-at-a-SD-vertex}, $\{\alpha',\beta\}$ is a solitary doubleton in~$G'$. Consequently, $\{\alpha',\beta\}$ is the only solitary class in~$G'$.
\end{proof}

We define the family $\mathcal{D}$ as follows. Firstly, $\mathcal{D}_0\subset \mathcal{D}$. Secondly, if $G\in \mathcal{D}$ then \mbox{$G':=(G\odot K_4)_{v}$} belongs to~$\mathcal{D}$, where $v$ is any vertex that is incident with a solitary doubleton of~$G$.
The following is an immediate consequence of Corollary~\ref{cor:family-D0} and Proposition~\ref{prp:two-two-from-two-two}.
\begin{cor}\label{cor:easy-direction}
Each member of~$\mathcal{D}$ is a $3$-connected \mbox{$3$-regular} graph that has solitude pattern~$(2)$.~\qed
\end{cor}

%


Now, we recall the definition of $1$-staircases from Section~\ref{sec:more-than-one-SD}, and observe the following. Each $1$-staircase of order eight or more has precisely two vertices that are incident with both solitary doubletons, and has four vertices that are incident with precisely one solitary doubleton. The following proposition deals with the case in which $G$ is a $1$-staircase of order ten or more.

\begin{prp}\label{prp:two-one-from-two-two}
Let $G$ be a $1$-staircase of order ten or more, and let $v$ denote any vertex that is incident with at least one of its solitary doubletons, say $\{\alpha_1,\beta_1\}$~and~$\{\alpha_2,\beta_2\}$. The following statements hold for $G':=(G\odot K_4)_v$:
\begin{enumerate}[(i)]
\item if $v$ is incident with both solitary doubletons of~$G$, then $G'$ is a larger $1$-staircase, 
\item or otherwise, $G'$ is a $3$-staircase.
\end{enumerate}
\end{prp}
\begin{proof}
We invite the reader to observe~$(i)$. Now, suppose that $v$ is incident with precisely one solitary doubleton, and
adjust notation so that $\alpha_1:=uv$~and~$\alpha_2:=uw$. Let $d$ be the unique edge in~$\partial_{G}(w)-\partial_{G}(\{u,v\})$, and let $f$ be the unique edge in~$\partial_{G}(v)-\partial_{G}(\{u,w\})$. We invite the reader to observe that $G'-\alpha_{1}-d-f$ is a (spanning) $3$-dumbbell in~$G'$; since $G'$ is $3$-connected, by definition, $G'$ is a $3$-staircase.
\end{proof}

We are now ready to state and prove our main result of this section.
\begin{thm}\label{thm:SP2}
A \cc{3}~$G$ has solitude pattern~$(2)$ if and only if $G\in \mathcal{D}$.
\end{thm}
\begin{proof}
By Corollary~\ref{cor:easy-direction}, it remains to prove the forward implication.
Let $G$ be a \cc{3} that has solitude pattern~$(2)$, and let $\{\alpha,\beta\}$ denote its solitary doubleton. Clearly, the order of $G$ is at least six. We proceed by induction on the order of~$G$.

We let $T$ denote the triangle containing $\alpha$; see Corollary~\ref{lem:vertex-disjoint-triangle-lemma}. Let $G'$ denote the \cc{3}~$G/(T\rightarrow t)$, and let $\alpha'$ denote the unique edge in $\partial_{G}(T)$ that is nonadjacent with~$\alpha$. By Corollary~\ref{cor:SD-shrinking-triangle}, $\{\alpha',\beta\}$ is a solitary doubleton in~$G'$; thus, by Corollary~\ref{cor:main-thm}, $G'$ has one of the following solitude patterns:~$(2,2,2)$, $(2,2,1)$, $(2,2)$, $(2,1,1)$, $(2,1)$~and~$(2)$. Observe that $G:=(G'\odot K_4)_{t}$.

If $G'$ has solitude pattern~$(2,2,2),(2,2,1)$~or~$(2,1,1)$ then, by Theorems~\ref{thm:more-than-one-SD}~and~\ref{thm:SDSS}, $G'\in \{K_4,\overline{C_6}, R_8, N_{10}\}$; using Propositions~\ref{prp:N10-vertex-orbits}~and~\ref{prp:R8-vertex-orbits}, the reader may verify that either $G\in \{\overline{C_6},R_8,$ $N_{10}\}$ or $G$ is the $1$-staircase of order ten (shown in Figure~\ref{fig:R10}) or $G$ is one of the two $3$-staircases of order twelve (shown in Figure~\ref{fig:st12}); in each case, we arrive at a contradiction to the hypothesis that $G$ has solitude pattern~$(2)$. On the other hand, if $G'$ has solitude pattern~$(2,2)$ then, by Theorem~\ref{thm:more-than-one-SD}, $G'$ is a $1$-staircase of order ten or more; consequently, by Proposition~\ref{prp:two-one-from-two-two}, we conclude that the solitude pattern of~$G$ is not $(2)$; a contradiction.

If $G'$ has solitude pattern~$(2,1)$ then, by Theorem~\ref{thm:SDSS}, $G'$ is a $3$-staircase of order twelve or more; since $G$ has solitude pattern~$(2)$, by Proposition~\ref{prp:two-one-and-two-from-two-one}, $G\in \mathcal{D}_0$ which is a subset of~$\mathcal{D}$. Finally, if $G'$ has solitude pattern~$(2)$ then, by the induction hypothesis, $G'\in \mathcal{D}$; by the definition of $\mathcal{D}$, we conclude that $G\in \mathcal{D}$.
%
%
%
\end{proof}

		

It is worth noting that every \ecc{3} $r$-graph, that has a solitude pattern shown in the first six rows of Table~\ref{table1}, is planar; see Theorems~\ref{thm:more-than-one-SD}, \ref{thm:SDSS}~and~\ref{thm:3SS}. On the other hand, there exist \ecc{3} $r$-graphs, where $r\ge 4$, with solitude pattern~$(2)$ that are nonplanar; see Figure~\ref{fig:nonplanar-graph} for one such example. This perhaps indicates that it may be more difficult to characterize \ecc{3} \mbox{$r$-graphs} that have solitude pattern~$(2)$.

In the final part of this section, we prove the forward implication of Theorem~\ref{thm:3SS}.

\subsection{Three solitary singletons}\label{sec:3SS}

We begin by recalling the family $\mathcal{S}$ shown in Figure~\ref{fig: 3cc-with-three-distinct-solitary-singletons}, and by restating Theorem~\ref{thm:3SS}.

\begin{reptheorem}{thm:3SS}
A \ecc{3} $r$-graph~$G$ has 
solitude pattern $(1,1,1)$ if and only if $G$ is one of the six graphs in~$\mathcal{S}\cup \{\theta\}$.
\end{reptheorem}
\begin{proof}
By Corollary~\ref{lem:solitary-patterns-in-r-graphs}, every \ecc{3} $r$-graph, of order four or more, with solitude pattern $(1,1,1)$ is a \cc{3}.
  Let $G$ denote a \cc{3}, distinct from $\theta$, that has three solitary singletons, say $e_i:=u_iv_i$ for $i\in \{1,2,3\}$.
  We begin with an observation that will come in handy multiple times.

\begin{sta}\label{sta:all-three-solitary-edges-in-one-shore}
If $\partial(W)$ is a $3$-cut such that $e_1,e_2$ and $e_3$ lie in $G[W]$ then $G/W$ is either $\theta$ or~$K_4$.
\end{sta}
\begin{proof}
Since $\{e_1\},\{e_2\}$ and $\{e_3\}$ are solitary classes that are pairwise mutually exclusive, it follows that, for each $d\in \partial(W)$, precisely one of $e_1,e_2$ and $e_3$ depends on~$d$. By Lemma~\ref{lem:solitary-ME-sepcuts}~(iii), each edge of $\partial(W)$ is solitary in $G/W$. By Corollary~\ref{lem:three-pairwise-adjacent-solitaryedges}, we conclude that $G/W$ is either $\theta$ or~$K_4$.
\end{proof}

For distinct $i,j\in \{1,2,3\}$, we invoke Corollary~\ref{thm:distance-one} to infer that 
  \distg{G}{e_i}{e_j}~$=1$, and we let $f_{ij}$ denote an edge that is adjacent with each of $e_i$~and~$e_j$.

    We let $H$ denote the subgraph of $G$ formed by the edges 
  $e_1,e_2,e_3,f_{12},f_{23},f_{13}$ and their ends.
    Observe that $H$ is a simple connected graph with six vertices and six edges; hence it contains exactly one cycle, say $Q$. Furthermore, $H-f_{ij}$ is a tree for distinct $i,j\in \{1,2,3\}$; consequently, each $f_{ij}\in Q$. 
    \begin{figure}[!htb]
    \begin{subfigure}[b]{.24\textwidth}
\begin{tikzpicture}[scale=0.35]
\node[draw=none] at (30:3.5) {$e_1$};
\node[draw=none] at (150:3.6) {$e_2$};
\node[draw=none] at (270:3.5) {$e_3$};
\node[draw=none] at (90:3.6) {$f_{12}$};
\node[draw=none] at (210:3.6) {$f_{23}$};
\node[draw=none] at (330:3.6) {$f_{13}$};
\node[circle,fill=white] (1) at (0:3.5){};
\node[circle,fill=white] (2) at (60:3.5){};
\node[circle,fill=white] (3) at (120:3.5){};
\node[circle,fill=white] (4) at (180:3.5){};
\node[circle,fill=white] (5) at (240:3.5){};
\node[circle,fill=white] (6) at (300:3.5){};
 \node[draw=none] at (30-36:4.4) {$v_1$};
 \node[draw=none] at (30+36:4.4) {$u_1$};
  \node[draw=none] at (150-36:4.4) {$v_2$};
 \node[draw=none] at (150+36:4.4) {$u_2$};
 \node[draw=none] at (270-36:4.4) {$v_3$};
 \node[draw=none] at (270+36:4.4) {$u_3$};
\draw (1) -- (2) -- (3) -- (4) -- (5) -- (6)-- (1);
\node[circle,fill=white] (1) at (0:3.5){};
\node[circle,fill=white] (2) at (60:3.5){};
\node[circle,fill=white] (3) at (120:3.5){};
\node[circle,fill=white] (4) at (180:3.5){};
\node[circle,fill=white] (5) at (240:3.5){};
\node[circle,fill=white] (6) at (300:3.5){};
\end{tikzpicture}
\caption{}
  \label{fig:a}
\end{subfigure}
\begin{subfigure}[b]{.24\textwidth}
\begin{tikzpicture}[scale=0.35]
\node[draw=none] at (198:3.5) {$e_1$};
\node[draw=none] at (342:3.5) {$e_2$};
\node[draw=none] at (270:3.5) {$f_{12}$};
\node[draw=none] at (126:3.6) {$f_{13}$};
\node[draw=none] at (54:3.6) {$f_{23}$};
\node[draw=none] (7) at (83:5) {$e_3$};
\node[draw=none] at (90:7.8){$v_3$};
\node[draw=none] at (90:2.3){$u_3$};
\node[draw=none] at (378:4.4){$u_2$};
\node[draw=none] at (342-36:4.4){$v_2$};
\node[draw=none] at (198-36:4.4){$u_1$};
\node[draw=none] at (198+36:4.4){$v_1$};
\node[circle,fill=white] (1) at (18:3.5){};
\node[circle,fill=white] (2) at (90:3.5){};
\node[circle,fill=white] (3) at (162:3.5){};
\node[circle,fill=white] (4) at (234:3.5){};
\node[circle,fill=white] (5) at (306:3.5){};
\node[circle,fill=white] (6) at (90:7){};
\draw (6) -- (90:3.5);
\draw (1) -- (2) -- (3) -- (4) -- (5) -- (1);
\node[circle,fill=white] (1) at (18:3.5){};
\node[circle,fill=white] (2) at (90:3.5){};
\node[circle,fill=white] (3) at (162:3.5){};
\node[circle,fill=white] (4) at (234:3.5){};
\node[circle,fill=white] (5) at (306:3.5){};
\node[circle,fill=white] (6) at (90:7){};
\end{tikzpicture}
\caption{}
  \label{fig:b}
\end{subfigure}
\begin{subfigure}[b]{.24\textwidth}
\begin{tikzpicture}[scale=0.35]
\node[draw=none] (1) at (0:3.3) {$f_{13}$};
\node[draw=none] (2) at (270:3.3) {$f_{23}$};
\node[draw=none] (3) at (180:3.3) {$f_{12}$};
\node[draw=none] (4) at (90:3.3) {$e_1$};
\node[draw=none] (7) at (236:5.5) {$e_2$};
\node[draw=none] (8) at (235+70:5.5) {$e_3$};
\node[draw=none] at (45:4.45) {$v_1$};
\node[draw=none] at (45+90:4.45) {$u_1$};
\node[draw=none] at (45+180:4.45) {$v_2$};
\node[draw=none] at (45+270:4.45) {$u_3$};
\node[draw=none] at (270+22:7.85) {$v_3$};
\node[draw=none] at (180+70:7.85) {$u_2$};
\node[circle,fill=white] (4) at (0+45:3.5){};
\node[circle,fill=white] (1) at (90+45:3.5){};
\node[circle,fill=white] (2) at (180+45:3.5){};
\node[circle,fill=white] (3) at (270+45:3.5){};
\draw (1) -- (2) -- (3) -- (4) -- (1);
\node[circle,fill=white] (5) at (270+22:7){};
\node[circle,fill=white] (6) at (180+70:7){};
\draw (6) -- (180+45:3.5);
\draw (5) -- (270+45:3.5);

\node[circle,fill=white] (4) at (0+45:3.5){};
\node[circle,fill=white] (1) at (90+45:3.5){};
\node[circle,fill=white] (2) at (180+45:3.5){};
\node[circle,fill=white] (3) at (270+45:3.5){};
\node[circle,fill=white] (5) at (270+22:7){};
\node[circle,fill=white] (6) at (180+70:7){};
\end{tikzpicture}
\caption{}
  \label{fig:c}
\end{subfigure}
\begin{subfigure}[b]{.24\textwidth}
\begin{tikzpicture}[scale=0.25]
\node[draw=none] at (33:3) {$f_{12}$};
\node[draw=none] at (155:3) {$f_{23}$};
\node[draw=none] at (275:3) {$f_{13}$};
\node[draw=none] (7) at (80:5.5) {$e_2$};
\node[draw=none] (8) at (200:5.5) {$e_3$};
\node[draw=none] (9) at (340:5.5) {$e_1$};
\node[circle,fill=white] (1) at (90:3.5){};
\node[circle,fill=white] (2) at (210:3.5){};
\node[circle,fill=white] (3) at (330:3.5){};
\draw (1) -- (2) -- (3) -- (1);
\node[circle,fill=white] (4) at (90:7){};
\node[circle,fill=white] (5) at (210:7){};
\node[circle,fill=white] (6) at (330:7){};
\draw (4) -- (90:3.5);
\draw (5) -- (210:3.5);
\draw (6) -- (330:3.5);

\node[circle,fill=white] (1) at (90:3.5){};
\node[circle,fill=white] (2) at (210:3.5){};
\node[circle,fill=white] (3) at (330:3.5){};
\node[circle,fill=white] (4) at (90:7){};
\node[circle,fill=white] (5) at (210:7){};
\node[circle,fill=white] (6) at (330:7){};
\end{tikzpicture}
\caption{}
  \label{fig:d}
\end{subfigure}
\caption{Possibilities for the subgraph $H$ as per \ref{sta:4-possibilities-of-H}}
\label{possible-H}
\end{figure}

    \begin{sta}\label{sta:4-possibilities-of-H}
        Up to relabeling, the subgraph $H$ is one of the four labeled graphs shown in Figure~\ref{possible-H}.
    \end{sta} 
    \begin{proof}
        Clearly, if $Q$ is a $6$-cycle then $H=Q$ with edge labels as shown in Figure \ref{fig:a};
         whereas if $Q$ is a $5$-cycle then $H$ is $Q$ along with a pendant edge and edge labels as shown in Figure~\ref{fig:b}. If $Q$ is a $4$-cycle then (up to relabeling) $f_{12}f_{23}f_{13}$ is a path in $Q$; observe that this determines the labeled subgraph $H$ as shown in Figure \ref{fig:c}.
         If $Q$ is a $3$-cycle then $Q=f_{12}f_{23}f_{13}$ and the graph~$H$ is~$Q$ along with three pendant edges and labels as shown in Figure \ref{fig:d}. 
    \end{proof}
    
        In each of the four possibilities, we adopt the edge labeling of the subgraph $H$ as shown in Figure \ref{possible-H}. We now proceed to eliminate the last two of these possibilities.

For each $i\in \{1,2,3\}$, we let $f_i$ denote the companion of $e_i$, and we let $C_i:=\partial(X_i)$ and $D_i:=\partial(Y_i)$ denote the $3$-cuts associated with $(e_i,f_i)$ so that $X_i\cap Y_i=\emptyset$; see Lemma \ref{lem:associated-r-cuts} and Figure \ref{fig:associated-r-cuts-1}.
        \begin{sta}
            The labeled subgraph $H$ is one of the two graphs shown in Figures \ref{fig:a} and \ref{fig:b}.
        \end{sta}
        \begin{proof}
            First suppose that $H$ is the subgraph shown in Figure \ref{fig:d}. Since $Q$ is a triangle, observe that $f_{ij}\rightarrow e_k$ for distinct $i,j,k\in \{1,2,3\}$; consequently, the solitary singleton $e_k$ is not removable and this contradicts Corollary \ref{cor:main-thm}~(ii).

Now suppose that $H$ is the labeled subgraph shown in Figure~\ref{fig:c}, and adjust notation so that $f_{12}\in C_2$ and $f_{23}\in D_2$. Since \mbox{\distg{G}{e_1}{e_2}~$=1$}, and since $e_1$~and~$e_2$ are mutually exclusive, we conclude that $e_1$ lies in~$G[X_2]$; consequently, $f_2=f_{13}$ and $D_2$ is a trivial cut. However, this implies that $e_3\in D_2$ and that $e_2$ and $e_3$ are adjacent; this contradicts the already established fact that \distg{G}{e_2}{e_3}~$=1$.
        \end{proof}

        \begin{figure}[!htb]
    \centering
    \begin{tikzpicture}[scale=0.7]
\centering
\draw (-4.5,0) ellipse (2.5 and 4);
\draw (4.5,0) ellipse (2.5 and 4);
\node[circle,fill=white] (1) at (0,0){};
\node[draw=none] at (0,-0.6) {$u_3$};
\node[draw=none] at (0.4,1) {$e_3$};
\node[circle,fill=white] (2) at (0,2){};
\node[draw=none] at (0,2.6) {$v_3$};
\node[circle,fill=white] (3) at (-3.5,-2){};
\node[draw=none] at (-3.2,-2.6) {$v_1$};
\node[draw=none] at (-3.2,-1) {$e_1$};
\node[draw=none] at (3.8,-1) {$e_2$};
\node[circle,fill=white] (4) at (3.5,-2){};
\node[draw=none] at (3.5,-2.6) {$v_2$};
\node[draw=none] at (0,-2.6) {$f_3=f_{12}$};
\node[circle,fill=white] (5) at (-3.5,0){};
\node[draw=none] at (-3.2,0.6) {$u_1$};
\node[circle,fill=white] (6) at (3.5,0){};
\node[draw=none] at (3.5,0.6) {$u_2$};
\node[draw=none] (7) at (5,2){};
\node[draw=none] (8) at (-5,2){};
\draw (5) -- (-5,0);
\draw (6) -- (5,0);
\draw (3) -- (-5,-2);
\draw (4) -- (5,-2);
\draw (3) -- (5);
\draw (4) -- (6);
\draw (3) -- (4);
\draw (1) -- (2);
\draw (1) -- (6);
\draw (1) -- (5);
\draw (2) -- (7);
\draw (2) -- (8);
\node[draw=none] at (-1.4,4.3) {$C_3$};
\draw (-1.5,4) -- (-1.5,-4);
\node[draw=none] at (1.5,4.3) {$D_3$};
\draw (1.5,4) -- (1.5,-4);
\node[draw=none] at (-2.8,4.4) {$C_1$};
\draw (-3,4) to [bend right=38, looseness=1.25] (-3,-4);
\node[draw=none] at (3,4.4) {$D_2$};
\draw (3,-4) to [bend right=38, looseness=1.25] (3,4);
\node[draw=none] at (-4.2,3.4) {$X_1$};
\node[draw=none] at (4.2,3.4) {$Y_2$};
\node[draw=none] at (-2,3.4) {$X_3$};
\node[draw=none] at (1.9,3.4) {$Y_3$};
\node[draw=none] at (-1,-0.5) {$f_{13}$};
\node[draw=none] at (1,-0.5) {$f_{23}$};
\end{tikzpicture}
    \caption{Illustration for \ref{sta:different-1+1+1-graph}}
    \label{fig:illustration-for-different-1+1+1-graph}
\end{figure}
        \begin{sta}\label{sta:different-1+1+1-graph}
            If the labeled subgraph $H$ is as shown in Figure \ref{fig:b} then $G$ is the graph shown in Figure \ref{fig:different-1+1+1-graph}.
        \end{sta}
        \begin{proof}
          Adjust notation so that $f_{13}\in C_3$ and $f_{23}\in D_3$; see Figure~\ref{fig:illustration-for-different-1+1+1-graph}. Observe that $Q$ is a $5$-cycle in $G-v_3$. Since every cycle meets every cut in an even number of edges, it follows that $f_3\in Q$. Also, since $e_3\rightarrow f_3$, the edge $f_3$ is distinct from each of $e_1$ and $e_2$. Thus, $f_3=f_{12}$.

        Observe that the unique edge of $C_3\cap \partial(v_3)$ is an odd $1$-cut in $G-u_1-v_1$; thus, by Lemma~\ref{lem:associated-r-cuts}, this edge is the companion $f_1$ of $e_1$. In particular, $C_3=\{f_3,f_{13},f_1\}$. Observe that $C_3$ is the same as one of $C_1$ and $D_1$; adjust notation so that $C_3=D_1$ and $Y_1=\overline{X_3}$. Note that $X_1=X_3-u_1-v_1$ and that the edges $e_1,e_2$ and $e_3$ lie in~$G[\overline{X_1}]$. Thus, by \ref{sta:all-three-solitary-edges-in-one-shore}, the graph 
        $G_1:=G/\overline{X_1}$ is either $\theta$ or $K_4$. Using an analogous argument, we may adjust notation so that $Y_2=Y_3-u_2-v_2$, and infer that the graph $G_2:=G/\overline{Y_2}$ is either $\theta$ or $K_4$.

        \begin{figure}[!htb]
    \centering
    \begin{tikzpicture}[scale=0.7]
\centering

\node[draw=none] at (0,-0.6) {$u_3$};
\node[draw=none] at (0.4,1) {$e_3$};
\node[circle,fill=white] (2) at (0,2){};
\node[circle,fill=white] (1) at (0,0){};
\node[draw=none] at (0,2.6) {$v_3$};
\draw (1) -- (2);

\node[draw=none] at (-3.2,-2.6) {$v_1$};
\node[draw=none] at (-2.5,-1) {$e_1$};
\node[draw=none] at (-3,0.4) {$u_1$};
\node[circle,fill=white] (5) at (-3,0){};
\node[circle,fill=white] (3) at (-3,-2){};

\node[draw=none] at (3.5,-1) {$e_2$};
\node[draw=none] at (3.2,-2.6) {$v_2$};
\node[draw=none] at (3.2,0.6) {$u_2$};
\node[circle,fill=white] (4) at (3,-2){};
\node[circle,fill=white] (6) at (3,0){};

\node[draw=none] at (0,-2.6) {$f_3=f_{12}$};

\node[circle,fill=white] (9) at (-5,0){};
\node[circle,fill=white] (10) at (5,0){};
\node[circle,fill=white] (11) at (6.5,2){};
\node[circle,fill=white] (12) at (6.5,-2){};
\draw (11) -- (12) -- (10);
\draw[line width=1.5pt] (10) -- (11);

\node[draw=none] (7) at (5,2){};
\node[draw=none] (8) at (-5,2){};
\draw (5) -- (9);
\draw (3) -- (9);
\draw (2) -- (9);

\draw (6) -- (5,0);
\draw (4) -- (12);
\draw[line width=1.5pt] (3) -- (5);
\draw (4) -- (6);
\draw (3) -- (4);

\draw (1) -- (6);
\draw (1) -- (5);
\draw (2) -- (11);
\node[circle,fill=white] (2) at (0,2){};
\node[circle,fill=white] (1) at (0,0){};
\node[circle,fill=white] (5) at (-3,0){};
\node[circle,fill=white] (3) at (-3,-2){};
\node[circle,fill=white] (4) at (3,-2){};
\node[circle,fill=white] (6) at (3,0){};
\node[circle,fill=white] (9) at (-5,0){};
\node[circle,fill=white] (10) at (5,0){};
\node[circle,fill=white] (11) at (6.5,2){};
\node[circle,fill=white] (12) at (6.5,-2){};
\end{tikzpicture}
    \caption{The graph $N_{10}$ as it appears in the proof of \ref{sta:different-1+1+1-graph}}
    \label{fig:2+1+1 graph}
\end{figure}
         Observe that if each of $G_1$ and $G_2$ is the $\theta$ graph then $G$ is the bicorn~$R_8$, and each of $e_1$~and~$e_2$ participates in a solitary doubleton; contrary to our hypothesis. On the other hand, if precisely one of $G_1$ and $G_2$, say $G_1$, is the $\theta$ graph, then $G$ is the graph~$N_{10}$ shown in Figure~\ref{fig:2+1+1 graph}, and $e_1$ participates in a solitary doubleton; a contradiction once again. Finally, if each of $G_1$ and $G_2$ is $K_4$ then $G$ is the graph shown in Figure \ref{fig:different-1+1+1-graph}, and $e_1,e_2$ and $e_3$ are indeed solitary singletons as discussed earlier.
        \end{proof}

        Henceforth, we may thus assume that the labeled subgraph $H$ is the graph shown in Figure \ref{fig:a}. Our goal is to demonstrate that $G$ is one of the four graphs shown in Figures \ref{fig:tricorn}, \ref{fig:tricorn+2vertices}, \ref{fig:tricorn+4vertices} and \ref{fig:tricorn+6vertices}.
        To this end, we shall establish various properties of the graph $G$ step-by-step.
        \begin{sta}\label{sta:induced-six-cycle}
            $|\partial(H)|=6$ and $H=Q$ is an induced $6$-cycle of~$G$.
        \end{sta}
        \begin{proof}
            Observe that, if $H=Q$ is a conformal cycle in $G$ then $e_1,e_2$ and $e_3$ lie in a common \pmg\ of~$G$, contrary to the fact that they are mutually exclusive edges. Since $G$ is $3$-edge-colorable and \mbox{$3$-regular}, we invoke Lemma~\ref{lem:conformal-lemma} to deduce that $|\partial(H)|\notin \{0,2,4\}$. 
        \end{proof}
        
        We shall find it useful to fix a $3$-edge-coloring $(M_1,M_2,M_3)$ of $G$ so that the solitary edge $e_i$ lies in the \pmg~$M_i$. It is easy to see that this determines the color of each edge of $H$ as follows. 
        \begin{sta}\label{sta:H-determined}
            For distinct $i,j,k\in \{1,2,3\}$, the edge $f_{ij}$ lies in the \pmg~$M_k$.\qed
        \end{sta}
        \begin{sta}\label{sta:empty-set}
            $E(Q)\cap \{f_1,f_2,f_3\}=\emptyset$.
        \end{sta}
        \begin{proof}
            By symmetry, it suffices to argue that $f_1\notin E(Q)$. Suppose to the contrary that $f_1\in E(Q)$. Since $e_1\rightarrow f_1$, it follows from \ref{sta:H-determined} that $f_1=f_{23}$. Adjust notation so that $u_2\in X_1$ and $v_3\in Y_1$. Consequently, $e_2\in G[X_1]$ and $e_3\in G[Y_1]$; furthermore, $f_{12}\in C_1$ and $f_{13}\in D_1$. Let $g$ denote the unique edge in $C_1\cap \partial(v_1)$. Since $e_1\in M_1$ and $f_{13}\in M_2$, we infer that $g\in M_3$. However, $M_3\cap C_1=\{f_{12},g\}$; a contradiction since $C_1$ is an odd cut.
        \end{proof}
        
        For each $i\in \{1,2,3\}$, since $f_i\notin E(Q)$, using the fact that each cut meets each cycle in an even number of edges, we immediately infer the following.
        \begin{sta}
            For each $i\in \{1,2,3\}$, $Q$ is a subgraph of $G[\overline{X_i}]$ or otherwise of $G[\overline{Y_i}]$.\qed
        \end{sta}
             For each $i\in \{1,2,3\}$, adjust notation so that $Q$ is a subgraph of $G[\overline{X_i}]$.
        \begin{sta}\label{sta:theta-k_4-sta}
            For each $i\in \{1,2,3\}$, the $C_i$-contraction $G_i:=G/\overline{X_i}$ is either $\theta$ or $K_4$.
        \end{sta}
        \begin{proof}
        By symmetry, it suffices to prove for $i=1$.
            Since $Q$ is a subgraph of $G[\overline{X_1}]$, 
            the edges $e_1,e_2$ and $e_3$ lie in $G[\overline{X_1}]$. By \ref{sta:all-three-solitary-edges-in-one-shore}, we infer that $G_1$ is either $\theta$ or $K_4$.
        \end{proof}
        
        \begin{sta}\label{sta:pairwise-disjoint}
            The sets $X_1,X_2$ and $X_3$ are pairwise disjoint.
        \end{sta}
        \begin{proof}
        By symmetry, it suffices to prove that $X_1$ and $X_2$ are disjoint. From Proposition~\ref{prp:3-cut-matching-laminar}, we infer that one of $X_2$ and $\overline{X_2}$ is a subset of one of $X_1$ and $\overline{X_1}$. We shall eliminate three of these four possibilities.

            Firstly, as per our notation, $V(Q)\subseteq \overline{X_1}\cap \overline{X_2}$; this implies that $\overline{X_2}\not\subseteq X_1$.
            Secondly, note that the unique neighbour of $u_2$, that lies in $X_2$, does not belong to $X_1$; thus
            $X_2\not\subseteq X_1$. An analogous argument implies that $X_1\not\subseteq X_2$; equivalently, $\overline{X_2}\not\subseteq\overline{X_1}$.
            Therefore, $X_2\subseteq \overline{X_1}$.
        \end{proof}
        

\begin{sta}\label{sta:diajoint-three-3-cuts}
    The edges $f_1,f_2$ and $f_3$ do not lie in $\partial(V(Q))$. Furthermore, the cuts $C_1,C_2$ and $C_3$ are pairwise disjoint.
\end{sta}
\begin{proof}
    Since $e_1\rightarrow f_1$, clearly $f_1\notin \partial(\{u_1,v_1\})$. Note that each $g\in \partial(V(Q))$ has one end in $\{u_i,v_i\}$ and the other end in $X_i$ for some $i\in \{1,2,3\}$. This fact, along with \ref{sta:pairwise-disjoint}, implies that $f_1\notin \partial(\{u_2,v_2,u_3,v_3\})$ since $f_1\in \partial(X_1)$. Thus, $f_1\notin \partial(V(Q))$. By symmetry, this proves the first part.

    Note that $C_1\cup C_2\cup C_3=\partial(V(Q))\cup \{f_1,f_2,f_3\}$. Since $|\partial(V(Q))|=6$ by \ref{sta:induced-six-cycle}, and since the two sets on the right hand side are pairwise disjoint, we infer that the cuts $C_1,C_2$~and~$C_3$ are pairwise disjoint.
\end{proof}

\begin{figure}[!htb]
        \centering
\begin{tikzpicture}[scale=1.2]
\node[draw=none] at (30+60:3.4) {$e_1$};
\node[draw=none] at (150+60:3.4) {$e_2$};
\node[draw=none] at (270+60:3.4) {$e_3$};
\node[draw=none] at (90+60:3.4) {$f_{12}$};
\node[draw=none] at (210+60:3.4) {$f_{23}$};
\node[draw=none] at (330+60:3.4) {$f_{13}$};
\node[circle,fill=white] (1) at (0:3.5){};
\node[circle,fill=white] (2) at (60:3.5){};
\node[circle,fill=white] (3) at (120:3.5){};
\node[circle,fill=white] (4) at (180:3.5){};
\node[circle,fill=white] (5) at (240:3.5){};
\node[circle,fill=white] (6) at (300:3.5){};
\node[circle,fill=white] (7) at (90:0.6){};
\node[circle,fill=white] (8) at (210:0.6){};
\node[circle,fill=white] (9) at (330:0.6){};
\node[circle,fill=white] (10) at (90:1.6){};
\node[circle,fill=white] (11) at (210:1.6){};
\node[circle,fill=white] (12) at (330:1.6){};
\draw (7) -- (10);
\draw (8) -- (11);
\draw (9) -- (12);
\draw (1) -- (0-15:2.1);
\draw (2) -- (60+15:2.1);
\draw (3) -- (120-15:2.1);
\draw (4) -- (180+15:2.1);
\draw (5) -- (240-15:2.1);
\draw (6) -- (300+15:2.1);
\draw [bend left=60, looseness=1.50] (290:4) to (15:4);
\draw [bend left=60, looseness=1.50] (290+120:4) to (15+120:4);
\draw [bend left=60, looseness=1.50] (290+240:4) to (15+240:4);
\node[draw=none] at (90-14:1.19) {$f_{1}$};
\node[draw=none] at (90:2.3) {$X_{1}$};
\node[draw=none] at (90+120:2.3) {$X_{2}$};
\node[draw=none] at (90+240:2.3) {$X_{3}$};
\node[draw=none] at (0:0) {$Z$};
\node[draw=none] at (90-14+120:1.19) {$f_{2}$};
\node[draw=none] at (90-14+240:1.19) {$f_{3}$};
\draw circle(0.8);
\draw (330-240:2) circle(0.7);
\draw (330-120:2) circle(0.7);
\draw (330:2) circle(0.7);
 \node[draw=none] at (30-34:3.7) {$u_3$};
 \node[draw=none] at (30+36:3.7) {$v_1$};
  \node[draw=none] at (150-34:3.7) {$u_1$};
 \node[draw=none] at (150+36:3.7) {$v_2$};
 \node[draw=none] at (270-34:3.7) {$u_2$};
 \node[draw=none] at (270+36:3.7) {$v_3$};
\draw (1) -- (2) -- (3) -- (4) -- (5) -- (6)-- (1);
\end{tikzpicture}
\caption{Illustration for \ref{sta:Z-sta} and \ref{sta:final-sta}}
  \label{fig:6-cycle-in-1+1+1}
        \end{figure}

See Figure~\ref{fig:6-cycle-in-1+1+1} and let $Z:=V(G) -V(Q)- X_1-X_2- X_3$. Observe that $e_1,e_2,e_3\in E(G[\overline{Z}])$ and $\partial(Z)=\{f_1,f_2,f_3\}$; hence, by Proposition~\ref{prp:r-cuts}, $\partial(Z)$ is a separating cut. We let $z_i\in Z$ denote the end of $f_i$ for each $i\in\{1,2,3\}$; note that $z_1,z_2$ and $z_3$ need not be distinct.
 By~\ref{sta:all-three-solitary-edges-in-one-shore}, $G/\overline{Z}$ is either $\theta$ or $K_4$. We now eliminate one of these possibilities.
\begin{sta}\label{sta:Z-sta}
The graph $G/\overline{Z}$ is $\theta$. In other words, $z_1=z_2=z_3$.
\end{sta}
\begin{proof}
    Suppose not; assume that $G/\overline{Z}$ is $K_4$. Note that, since $e_1,f_1\in M_1$, the restriction of $M_1$ to the subgraph $J:=G-X_1-u_1-v_1-z_1$ is a \pmg~of $J$. However, the reader may verify that irrespective of whether each of $G[X_2]$ and $G[X_3]$ is $K_1$ or $K_3$, the matchable subgraph $J$ is $2$-edge-connected; whence, by Lemma \ref{lem:unique-pm-in-a-graph}, $J$ has at least two \pmg s, and using these we infer that $e_1$ is not solitary; contradiction.
\end{proof}

Recall that $G_i:=G/\overline{X_i}$ for each $i\in \{1,2,3\}$. 
\begin{sta}\label{sta:final-sta}
    The graph $G$ is one of the four graphs shown in Figures~\ref{fig:tricorn}, \ref{fig:tricorn+2vertices}, \ref{fig:tricorn+4vertices} and \ref{fig:tricorn+6vertices}.
\end{sta}
\begin{proof}
Note that if each of $G_1,G_2$ and $G_3$ is $\theta$ then $G$ is the tricorn graph shown in Figure~\ref{fig:tricorn}, whereas if each of them is $K_4$ then $G$ is the graph shown in Figure~\ref{fig:tricorn+6vertices}. On the other hand, if precisely one of these three graphs is $K_4$ then $G$ is the graph shown in Figure~\ref{fig:tricorn+2vertices}, whereas if two of them are $K_4$ then $G$ is the graph shown in Figure~\ref{fig:tricorn+4vertices}.
\end{proof}

 This completes the proof of Theorem~\ref{thm:3SS}.
\end{proof}


This concludes our characterizations of all of the solitude patterns claimed in Table~\ref{table1}. Now, we recall Corollary~\ref{cor:at-least-n/2-solitary-edges} which states that, for a \mcg, the cardinality of any largest equivalence class is at most $\frac{n}{2}$, and that if equality holds then every such class is a solitary class.
In the following section, we provide a complete characterization of the tight examples.

\section{Matching covered graphs that satisfy \texorpdfstring{$\varepsilon=\frac{n}{2}$}{}}\label{sec:largest-EC-n/2}
In this section, we provide a complete characterization of all \mcg s whose largest equivalence class has cardinality~$\frac{n}{2}$, thus fixing an error in~\cite[Proposition~1.8]{lkfz20}. One may easily observe that such an equivalence class is a solitary class; recall Corollary~\ref{cor:at-least-n/2-solitary-edges}.
We begin by investigating the behavior of dependence relation across even $2$-cuts.

\subsection{Dependence and equivalence classes across even \texorpdfstring{$2$}{}-cuts}

We begin with a pertinent consequence of Propositions~\ref{prp:gluing}~and~\ref{prp:2-bond}.
\begin{lem}\label{lem:dependence-across-even-2-cuts}
 {\sc[Dependence across Even $2$-Cuts]}
\newline Let $G$ be a \mcg\ that has an even $2$-cut~$C$, let $G_1$ denote a marked \mbox{$C$-component} of~$G$, and let $e_1$ denote its marker edge. Then, for each pair $d\in E(G_1)-e_1$ and $f\in C$, the following statements hold:
\begin{enumerate}[(i)]
            \item $d\xrightarrow{G} f$ if and only if $d\xrightarrow{G_1} e_1$, and
            \item $f\xrightarrow{G} d$ if and only if $e_1\xrightarrow{G_1} d$.
\end{enumerate}
Furthermore, for each pair $d,d'\in E(G_1)-e_1$, the following holds:
\begin{enumerate}
 \item[(iii)] $d\xrightarrow{G} d'$ if and only if $d\xrightarrow{G_1} d'$.
\end{enumerate}
\end{lem}
\begin{proof}
We let $G_2$ denote the marked $C$-component distinct from~$G_1$, and $e_2$ its marker edge.
For all implications, we shall prove their contrapositives.

    Suppose that $G_1$ has a \pmg~$M$ containing~$d$ but not~$e_1$. Let $N$ be a \pmg\ of~$G_2$ not containing~$e_2$. By Proposition~\ref{prp:gluing}, $M\cup N$ is a \pmg\ of~$G$ containing $d$ but not containing $f$.
Conversely, suppose that $G$ has a \pmg~$M$ containing~$d$ but not~$f$. By Proposition~\ref{prp:2-bond}, $M\cap E(G_1)$ is a \pmg\ of~$G_1$ containing~$d$ but not~$e_1$. This proves~$(i)$.

Suppose that $G_1$ has a \pmg~$M$ containing~$e_1$ but not~$d$. Let $N$ be a \pmg\ of~$G_2$ containing~$e_2$. By Proposition~\ref{prp:gluing}, $(M-e_1)+(N-e_2)+C$ is a \pmg\ of~$G$ containing~$f$ but not~$d$. Conversely, suppose that $G$ has a \pmg~$M$  containing~$f$ but not~$d$. By Proposition~\ref{prp:2-bond}, $(M\cap E(G_1))+e_1$ is a \pmg\ of~$G_1$ containing~$e_1$ but not~$d$. This proves~$(ii)$.

Finally, suppose that $G_1$ has a \pmg~$M$ containing~$d$ but not~$d'$. Let $N$ be a \pmg\ of~$G_2$ containing~$e_2$ and $N'$ be a \pmg\ of~$G_2$ not containing~$e_2$.
By Proposition~\ref{prp:gluing}, if $e_1\in M$ then $(M-e_1)+(N-e_2)+C$, or otherwise $M\cup N'$, is a \pmg\ of~$G$ containing $d$ but not~$d'$.
Conversely, suppose that $G$ has a \pmg~$M$ containing~$d$ but not~$d'$. 
By Proposition~\ref{prp:2-bond}, if $f\in M$ then $(M\cap E(G_1))+e_1$, or otherwise $M\cap E(G_1)$, is a \pmg\ of~$G_1$ containing~$d$ but not~$d'$. This proves statement~$(iii)$.
%
%
%
%
%
\end{proof}

For a \mcg~$G$ that has an even $2$-cut~$C$, the first part of the following corollary may be used to compute the equivalence classes of its marked $C$-components $G_1$~and~$G_2$ from the equivalence classes of~$G$. Its second part allows one to compute the equivalence class of~$G$ that contains~$C$ using equivalence classes of $G_1$~and~$G_2$ that contain the corresponding marker edge. It is an immediate consequence of the above.

\begin{cor}\label{lem:2-cut-decomposition}
Let $G$ be a \mcg\ that has an even $2$-cut~$C$, let $G_1$ denote a marked $C$-component of~$G$, and let $e_1$ denote its marker edge. Then, precisely one of the following holds for any equivalence class~$D$ of~$G$:
\begin{enumerate}[(i)]
            \item either $C\subseteq D$ and $(D\cap E(G_1))+e_1$ is an equivalence class of~$G_1$, or
            	\item otherwise $C\cap D=\emptyset$ and $D\cap E(G_1)$ is either empty or an equivalence class of~$G_1$. 
\end{enumerate}
Furthermore, let $G_2$ be the marked $C$-component of~$G$ distinct from~$G_1$, and $e_2$ its marker edge. Then, $D:=(D_1-e_1)+(D_2-e_2)+C$ is an equivalence class of~$G$, where $D_i$ is the equivalence class of~$G_i$ containing~$e_i$ for each $i\in \{1,2\}$. \qed
\end{cor}

In the following part, we state and prove the main result of this section as well as its consequences

\subsection{The family \texorpdfstring{$\mathcal{L}$}{}}

We let $\mathcal{C}_{2}$ denote the family that comprises graphs of order two and size two or more. We define the family~$\mathcal{K}_4$ as follows: a graph~$G\in \mathcal{K}_4$ if and only if its underlying simple graph is~$K_4$ and it has a \pmg\ $\{uv,yz\}$ such that $\mu(u,v)=\mu(y,z)=1$. Observe that each member~$G$ of $\mathcal{C}_2 \cup \mathcal{K}_4$ satisfies~$\varepsilon(G)=\frac{n}{2}$.
Now, we define the family $\mathcal{L}$ recursively as follows: (i)~$\mathcal{C}_{2}\cup \mathcal{K}_4\subset \mathcal{L}$, and (ii)~if $G_1,G_2\in \mathcal{L}$ then any graph~$G$ obtained by gluing them at $e_1$~and~$e_2$, where $e_i$ belongs to an equivalence class of~$G_i$ of cardinality~$\frac{n_i}{2}$ for each $i\in\{1,2\}$, also belongs to~$\mathcal{L}$. We proceed to prove the following.

\begin{thm}\label{thm:mcgs-with-solitary-pmg}
A graph~$G$ is \mc\ and satisfies $\varepsilon(G)=\frac{n}{2}$ if and only if \mbox{$G\in \mathcal{L}\cup \{K_2\}$}.
\end{thm}
\begin{proof}
The reverse implication is an immediate consequence of Corollary~\ref{prp:2-cut-connection-preserves-mc} and the last part of Corollary~\ref{lem:2-cut-decomposition}.
It remains to prove the forward implication.
We proceed by induction on the order of~$G$, and let $D$ be an equivalence class of cardinality~$\frac{n}{2}$. By Proposition~\ref{prp:pm-union-of-eqs}, $D$ is a \pmg\ of~$G$.
If the underlying simple graph of~$G$ is either $K_2$ or $K_4$, the reader may easily verify that $G\in \{K_2\}\cup \mathcal{C}_2\cup \mathcal{K}_4$ and we are done; now suppose that this is not the case.

Consequently, either the underlying simple graph is~$C_4$, or otherwise $n\ge 6$. In both cases, since $|D|=\frac{n}{2}$, we infer that $D$ contains an even $2$-cut~$C$; recall Lucchesi-Murty Theorem~(\ref{thm:source-class-2-cut}). We let $G_1$~and~$G_2$ denote the marked $C$-components of~$G$, and $e_1$~and~$e_2$ denote their marker edges, respectively.
By Corollary~\ref{lem:2-cut-decomposition}~(i), for each $i\in \{1,2\}$, the set~$D_i:=(D\cap E(G_i))+e_i$ is an equivalence class of~$G_i$; observe that $D_i$ is a \pmg\ of~$G_i$ and thus $|D_i|=\frac{n_i}{2}$. Ergo, by the induction hypothesis, $G_1, G_2\in \mathcal{L}$. Since $G$ is obtained by gluing $G_1$~and~$G_2$ at $e_1$~and~$e_2$, we conclude that $G\in \mathcal{L}$.
\end{proof}

We now use the above result to deduce a characterization of $r$-graphs that satisfy~$\varepsilon=\frac{n}{2}$. Observe that the degree sequence of a $2$-connected graph~$G$ (of even order) is the same as the degree sequence of the disjoint union of its marked \mbox{$C$-components}, where $C$ is any even $2$-cut of~$G$. This inspires the following subclass of~$\mathcal{L}$ for each fixed integer $r\ge 3$.

We define $C^{r}_{2}$ as the only $r$-regular member of~\mbb{\theta}{}{1}, and we define $\mathbb{K}_4^r$ as the set comprising $r$-regular members of \mbb{K}{4}{2}.
The family $\mathcal{L}^r$ is defined recursively as follows: (i)~$\mathbb{K}_4^r \cup \{C^{r}_{2}\} \subset \mathcal{L}^r$, and (ii)~if $G_1,G_2\in \mathcal{L}^r$ then any graph~$G$ obtained by gluing them at $e_1$~and~$e_2$, where $e_i$ belongs to an equivalence class of~$G_i$ of cardinality~$\frac{n_i}{2}$ for each $i\in\{1,2\}$, also belongs to~$\mathcal{L}^r$. Theorem~\ref{thm:mcgs-with-solitary-pmg} implies the equivalence of (i)~and~(ii) in the following corollary; using Corollary~\ref{cor:at-least-n/2-solitary-edges} and Theorem~\ref{thm:at-most-n/2-solitary-edges}, the reader may observe the equivalence of~(i) with each of~(iii)~and~(iv).


\begin{cor}\label{cor:r-graphs-n/2-solitary-edges}
For any $r$-graph~$G$, of order six or more, the following statements are equivalent:
\begin{enumerate}[(i)]
\item $\varepsilon(G)=\frac{n}{2}$,
\item $G\in \bigcup\limits_{r\ge 3}^{} \mathcal{L}^r$,
\item the solitary edges of~$G$ comprise a \pmg, and
\item $G$ has precisely $\frac{n}{2}$ solitary edges.\qed
\end{enumerate}
\end{cor}

We remark that the equivalence of (i)~and~(ii) in the above holds for all $r$-graphs; however, for the other two statements, one needs to exclude $r$-graphs of order at most four.

\section{Small examples of solitude patterns $(1,1), (2)$ and~$(1)$}

It is worth noting that all of the \ecc{3} $r$-graphs with specific solitude patterns that we have characterized in this paper, as shown in Table~\ref{table1}, happen to be planar. However, this need not hold for the remaining solitude patterns. Figure~\ref{fig:nonplanar-graph} shows a \ecc{3} $4$-graph with solitude pattern~$(2)$ that is nonplanar.

\begin{figure}[!htb]
    \centering
        \begin{tikzpicture}[scale=1]
\node [circle,fill=white] (0) at (0, 0.5) {};
		\node [circle,fill=white] (1) at (1, 1) {};
		\node [circle,fill=white] (2) at (1, 0) {};
		\node [circle,fill=white] (3) at (2, 1) {};
		\node [circle,fill=white] (4) at (2, 0) {};
		\node [circle,fill=white] (5) at (3, 1) {};
		\node [circle,fill=white] (6) at (3, 0) {};
		\node [circle,fill=white] (7) at (4, 0.5) {};
		
		\draw [bend left=45] (0) to (1);
		\draw (1) to (3);
		\draw (3) to (5);
		\draw[color=red, thick] (5) to (7);
		\draw [bend left=45] (7) to (6);
		\draw (6) to (4);
		\draw (4) to (2);
		\draw[color=red, thick] (2) to (0);
		\draw (1) to (2);
		\draw [bend left=45] (3) to (4);
		\draw (5) to (6);
		\draw [bend left=90, looseness=1.25] (0) to (7);
		\draw (0) to (1);
		\draw (2) to (5);
		\draw (3) to (4);
		\draw (6) to (7);
        \end{tikzpicture}
         \caption{A nonplanar $3$-edge-connected $4$-graph with solitude pattern~$(2)$}
         \label{fig:nonplanar-graph}
%
\end{figure}

Clearly, the natural open problem that arises from our work is to characterize those $3$-edge-connected $r$-graphs that have any of the following solitude patterns: $(1,1)$~and~$(1)$ for all values of $r$, and $(2)$ for $r\ge 4$. Figure~\ref{fig:the-smallest-$3$-edge-connected-$3$-graphs-with-solitary-pattern~$(1),(1,1)$} shows examples of such graphs for $r=3$; we verified using computations that they are the smallest such examples. On the other hand, Figure~\ref{fig:smallest-$3$-edge-connected-$4$-graphs-with-solitary-pattern~$(1,1),(2),(1)$} shows some small examples for $r= 4$; the first two are obtained by multiplying specific \pmg s of the graphs shown in Figures~\ref{fig:N10} and \ref{fig:tricorn}, respectively, and one may easily construct examples for $r\ge 5$ by multiplying appropriately chosen \pmg s. In each of these figures, the solitary classes are depicted using different colors.

\begin{figure}[!htb]
\centering
   \begin{subfigure}[b]{.4\textwidth}
       \centering
        \begin{tikzpicture}[scale=0.8]
\node [circle,fill=white] (0) at (0, 1) {};
		\node [circle,fill=white] (1) at (0, -1) {};
		\node [circle,fill=white] (2) at (1, 0) {};
		\node [circle,fill=white] (3) at (1.5, 1) {};
		\node [circle,fill=white] (4) at (2.5, 1) {};
		\node [circle,fill=white] (5) at (2, 0.5) {};
		\node [circle,fill=white] (6) at (2, 0) {};
		\node [circle,fill=white] (7) at (3, 0) {};
		\node [circle,fill=white] (8) at (3, -1) {};
		\node [circle,fill=white] (9) at (4, 0) {};
		\node [circle,fill=white] (10) at (5, 1) {};
		\node [circle,fill=white] (11) at (5, -1) {};
		
		\draw (0) to (3);
		\draw[color=red] (3) to (4);
		\draw (4) to (10);
		\draw (11) to (10);
		\draw[color=blue] (10) to (9);
		\draw (9) to (11);
		\draw (11) to (8);
		\draw (8) to (7);
		\draw (7) to (9);
		\draw (6) to (7);
		\draw (6) to (5);
		\draw (5) to (4);
		\draw (5) to (3);
		\draw (0) to (2);
		\draw (2) to (1);
		\draw (1) to (0);
		\draw (2) to (6);
		\draw (1) to (8);
		\end{tikzpicture}
         \end{subfigure}
   \begin{subfigure}[b]{.5\textwidth}
       \centering
        \begin{tikzpicture}[scale=0.8]
\node [circle,fill=white] (0) at (0, 2) {};
		\node [circle,fill=white] (1) at (0, 0) {};
		\node [circle,fill=white] (2) at (1, 1) {};
		\node [circle,fill=white] (3) at (2, 1) {};
		\node [circle,fill=white] (4) at (2, 0) {};
		\node [circle,fill=white] (5) at (3, 1) {};
		\node [circle,fill=white] (6) at (2.5, 2) {};
		\node [circle,fill=white] (7) at (3.5, 2) {};
		\node [circle,fill=white] (8) at (3, 1.5) {};
		\node [circle,fill=white] (9) at (4, 1) {};
		\node [circle,fill=white] (10) at (4, 0) {};
		\node [circle,fill=white] (11) at (5, 1) {};
		\node [circle,fill=white] (12) at (6, 2) {};
		\node [circle,fill=white] (13) at (6, 0) {};

\draw (0) to (2);
		\draw (2) to (1);
		\draw (1) to (0);
		\draw (2) to (3);
		\draw (3) to (4);
		\draw (3) to (5);
		\draw (5) to (8);
		\draw (8) to (6);
		\draw (8) to (7);
		\draw (0) to (6);
		\draw[color=red] (6) to (7);
		\draw (7) to (12);
		\draw (12) to (11);
		\draw (11) to (13);
		\draw (13) to (12);
		\draw (5) to (9);
		\draw (9) to (11);
		\draw (1) to (4);
		\draw (4) to (10);
		\draw (10) to (13);
		\draw (9) to (10);
		\end{tikzpicture}
         \end{subfigure}
   \begin{subfigure}[b]{.48\textwidth}
       \centering
        \begin{tikzpicture}[scale=0.8]
\node [circle,fill=white] (0) at (0, 2) {};
		\node [circle,fill=white] (1) at (0, 0) {};
		\node [circle,fill=white] (2) at (1, 1) {};
		\node [circle,fill=white] (3) at (2, 1) {};
		\node [circle,fill=white] (4) at (2, 0) {};
		\node [circle,fill=white] (5) at (3, 1) {};
		\node [circle,fill=white] (6) at (3, 1.5) {};
		\node [circle,fill=white] (7) at (2.5, 2) {};
		\node [circle,fill=white] (8) at (3.5, 2) {};
		\node [circle,fill=white] (9) at (4, 1) {};
		\node [circle,fill=white] (10) at (4, 2) {};
		\node [circle,fill=white] (11) at (5, 1) {};
		\node [circle,fill=white] (12) at (6, 2) {};
		\node [circle,fill=white] (13) at (6, 0) {};

\draw (0) to (1);
		\draw (1) to (2);
		\draw[color=red] (2) to (0);
		\draw (4) to (1);
		\draw (4) to (3);
		\draw (3) to (5);
		\draw (3) to (2);
		\draw (6) to (5);
		\draw (5) to (9);
		\draw (9) to (11);
		\draw (12) to (11);
		\draw (11) to (13);
		\draw (7) to (8);
		\draw (8) to (10);
		\draw (10) to (12);
		\draw (12) to (13);
		\draw (10) to (9);
		\draw (6) to (8);
		\draw (6) to (7);
		\draw (0) to (7);
		\draw (4) to (13);
		\end{tikzpicture}
        \vspace{400pt}
         \end{subfigure}
   \begin{subfigure}[b]{.46\textwidth}
       \centering
        \begin{tikzpicture}[scale=0.8]
\node [circle,fill=white] (0) at (0, 2) {};
		\node [circle,fill=white] (1) at (0, 0) {};
		\node [circle,fill=white] (2) at (1, 1) {};
		\node [circle,fill=white] (3) at (2, 1) {};
		\node [circle,fill=white] (4) at (2, 1.5) {};
		\node [circle,fill=white] (5) at (1.5, 2) {};
		\node [circle,fill=white] (6) at (2.5, 2) {};
		\node [circle,fill=white] (7) at (3, 1) {};
		\node [circle,fill=white] (8) at (3, 0) {};
		\node [circle,fill=white] (9) at (4, 2) {};
		\node [circle,fill=white] (10) at (4, 0) {};
		\node [circle,fill=white] (11) at (5, 1) {};
		\node [circle,fill=white] (12) at (6, 2) {};
		\node [circle,fill=white] (13) at (6, 0) {};

\draw (0) to (1);
		\draw (1) to (2);
		\draw (0) to (2);
		\draw (2) to (3);
		\draw (3) to (4);
		\draw (4) to (5);
		\draw (4) to (6);
		\draw (6) to (5);
		\draw (5) to (0);
		\draw (6) to (9);
		\draw (1) to (8);
		\draw (3) to (7);
		\draw (7) to (11);
		\draw (11) to (12);
		\draw (12) to (13);
		\draw[color=red] (13) to (11);
		\draw (8) to (10);
		\draw (10) to (13);
		\draw (7) to (8);
		\draw (9) to (12);
		\draw [bend left=135, looseness=6.25] (9) to (10);
		\end{tikzpicture}
        \vspace{300pt}
         \end{subfigure}
         \vspace{-350pt}
         \caption{Smallest \cc{3}s with solitude patterns~$(1,1)$ and $(1)$}
         \label{fig:the-smallest-$3$-edge-connected-$3$-graphs-with-solitary-pattern~$(1),(1,1)$}
\end{figure}

\begin{figure}[!htb]
    \centering
    \begin{subfigure}[b]{.45\textwidth}
       \centering
    \begin{tikzpicture}[scale=1]
        \node [circle,fill=white] (0) at (0, 2) {};
		\node [circle,fill=white] (1) at (0, 0) {};
		\node [circle,fill=white] (2) at (1, 1) {};
		\node [circle,fill=white] (3) at (2, 1) {};
		\node [circle,fill=white] (4) at (2, 0) {};
		\node [circle,fill=white] (5) at (3, 1) {};
		\node [circle,fill=white] (6) at (3, 2) {};
		\node [circle,fill=white] (7) at (4, 1) {};
		\node [circle,fill=white] (8) at (5, 2) {};
		\node [circle,fill=white] (9) at (5, 0) {};

\draw (0) to (1);
		\draw (1) to (2);
		\draw [bend right] (2) to (0);
		\draw (2) to (3);
		\draw[color=red] (3) to (4);
		\draw [bend left] (3) to (5);
		\draw[color=blue] (5) to (6);
		\draw (5) to (7);
		\draw (7) to (8);
		\draw [bend right] (7) to (9);
		\draw (9) to (8);
		\draw [bend right] (1) to (4);
		\draw (4) to (9);
		\draw [bend left] (6) to (8);
		\draw (0) to (6);
		\draw (0) to (2);
		\draw (1) to (4);
		\draw [bend right] (3) to (5);
		\draw (6) to (8);
		\draw (7) to (9);
    \end{tikzpicture}
    \end{subfigure}
    \begin{subfigure}[b]{.45\textwidth}
       \centering
       \begin{tikzpicture}[scale=0.8]
            \node[circle,fill=white] (1) at (90:1.25){};
            \node[circle,fill=white] (2) at (210:1.25){};
            \node[circle,fill=white] (3) at (330:1.25){};
            \node[circle,fill=white] (4) at (80-5-5:2+0.5){};
            \node[circle,fill=white] (5) at (100+5+5:2+0.5){};
            \node[circle,fill=white] (6) at (200-5-5:2+0.5){};
            \node[circle,fill=white] (7) at (220+5+5:2+0.5){};
            \node[circle,fill=white] (8) at (320-5-5:2+0.5){};
            \node[circle,fill=white] (9) at (340+5+5:2+0.5){};
            \node[circle,fill=white] (0) at (0:0){};
            \draw [bend right] (0) to (1);
            \draw (0) to (2);
            \draw (0) to (3);
            \draw (4) to (1);
            \draw (5) to (1);
            \draw (6) to (2);
            \draw (7) to (2);
            \draw (8) to (3);
            \draw (9) to (3);
            \draw (4) to (5);
            \draw (5) to (6);
            \draw[color=red] (6) to (7);
            \draw (7) to (8);
            \draw[color=blue] (8) to (9);
            \draw (9) to (4);
            \draw [bend left] (6) to (2);
            \draw [bend left] (3) to (9);
            \draw [bend right] (7) to (8);
            \draw [bend right] (1) to (0);
            \draw [bend right] (4) to (5);
        \end{tikzpicture}
       \end{subfigure}
           \begin{subfigure}[b]{.45\textwidth}
       \centering
       \begin{tikzpicture}[scale=1]
        \node [circle,fill=white] (0) at (0, 1) {};
		\node [circle,fill=white] (1) at (1, 2) {};
		\node [circle,fill=white] (2) at (1, 0) {};
		\node [circle,fill=white] (3) at (4, 2) {};
		\node [circle,fill=white] (4) at (4, 0) {};
		\node [circle,fill=white] (5) at (5, 1) {};

        \draw (1) to (0);
		\draw (0) to (2);
		\draw[color=red] (2) to (1);
		\draw (1) to (3);
		\draw[color=red] (3) to (4);
		\draw (4) to (5);
		\draw (5) to (3);
		\draw (2) to (4);
		\draw (1) to (4);
		\draw [bend left=105, looseness=1.50] (0) to (5);
		\draw [bend right=45, looseness=1.25] (0) to (2);
		\draw [bend left=45, looseness=1.25] (3) to (5);
    \end{tikzpicture}
       \end{subfigure}
       \begin{subfigure}[b]{.45\textwidth}
       \centering
       \begin{tikzpicture}[scale=1]
\node [circle,fill=white] (0) at (0, 1) {};
		\node [circle,fill=white] (1) at (1, 2) {};
		\node [circle,fill=white] (2) at (1, 0) {};
		\node [circle,fill=white] (3) at (2.5, 2) {};
		\node [circle,fill=white] (4) at (2.5, 0) {};
		\node [circle,fill=white] (5) at (4, 2) {};
		\node [circle,fill=white] (6) at (4, 0) {};
		\node [circle,fill=white] (7) at (5, 1) {};

        \draw (0) to (1);
		\draw (1) to (3);
		\draw (3) to (5);
		\draw [bend right] (5) to (7);
		\draw (7) to (6);
		\draw (6) to (4);
		\draw (4) to (2);
		\draw [bend right] (2) to (0);
		\draw (1) to (2);
		\draw[color=red] (3) to (4);
		\draw (5) to (6);
		\draw [bend left=90, looseness=1.5] (0) to (7);
		\draw (0) to (2);
		\draw (5) to (7);
		\draw (1) to (4);
		\draw (3) to (6);
    \end{tikzpicture}
       \end{subfigure}
    \caption{$3$-edge-connected $4$-graphs with solitude patterns~$(1,1), (2)$ and $(1)$}
    \label{fig:smallest-$3$-edge-connected-$4$-graphs-with-solitary-pattern~$(1,1),(2),(1)$}
\end{figure}

\noindent
{\bf Declaration of AI usage:} We did \underline{\bf not} use any AI tools or AI assistants at any stage of this work. In particular, \underline{\bf no} LLMs were employed.

\bibliographystyle{plainnat}
\bibliography{clm}

\begin{thebibliography}{27}
\providecommand{\natexlab}[1]{#1}
\providecommand{\url}[1]{\texttt{#1}}
\expandafter\ifx\csname urlstyle\endcsname\relax
  \providecommand{\doi}[1]{doi: #1}\else
  \providecommand{\doi}{doi: \begingroup \urlstyle{rm}\Url}\fi

\bibitem[Bondy and Murty(2008)]{bomu08}
J.~A. Bondy and U.~S.~R. Murty.
\newblock \emph{Graph Theory}.
\newblock Springer, 2008.
\newblock Available at \url{https://link.springer.com/book/9781846289699}.

\bibitem[Carvalho(1997)]{carv97}
M.~H. Carvalho.
\newblock \emph{{D}ecomposicao \'{O}tima em {O}relhas {P}ara {G}rafos
  {M}atching {C}overed}.
\newblock PhD thesis, University of Campinas, Brazil, 1997.
\newblock in Portuguese. Available at
  \url{https://repositorio.unicamp.br/Busca/Download?codigoArquivo=489478}.

\bibitem[Carvalho et~al.(1999)Carvalho, Lucchesi, and Murty]{clm99}
M.~H. Carvalho, C.~L. Lucchesi, and U.~S.~R. Murty.
\newblock Ear {D}ecompositions of {M}atching {C}overed {G}raphs.
\newblock \emph{Combinatorica}, 19:\penalty0 151--174, 1999.
\newblock \doi{10.1007/s004930050051}.

\bibitem[Carvalho et~al.(2005)Carvalho, Lucchesi, and Murty]{clm05}
M.~H. Carvalho, C.~L. Lucchesi, and U.~S.~R. Murty.
\newblock Graphs with {I}ndependent {P}erfect {M}atchings.
\newblock \emph{J.~Graph Theory}, 48:\penalty0 19--50, 2005.
\newblock \doi{10.1002/jgt.20036}.

\bibitem[Carvalho et~al.(2013)Carvalho, Lucchesi, and Murty]{clm13}
M.~H. Carvalho, C.~L. Lucchesi, and U.~S.~R. Murty.
\newblock On the number of perfect matchings in a bipartite graph.
\newblock \emph{\textsc{siam} J. Discrete Math.}, 27\penalty0 (2):\penalty0
  940--958, 2013.
\newblock \doi{10.1137/120865938}.

\bibitem[Carvalho et~al.(2018)Carvalho, Lucchesi, and Murty]{clm18}
M.~H. Carvalho, C.~L. Lucchesi, and U.~S.~R. Murty.
\newblock On tight cuts in matching covered graphs.
\newblock \emph{Journal of Combinatorics}, 9\penalty0 (1):\penalty0 163--184,
  2018.
\newblock \doi{10.4310/JOC.2018.v9.n1.a8}.

\bibitem[Chen et~al.(2025{\natexlab{a}})Chen, Jing, and Zang]{ggw25}
G.~Chen, G.~Jing, and W.~Zang.
\newblock Proof of the {G}oldberg–{S}eymour conjecture on edge–colorings of
  multigraphs.
\newblock \emph{J. Comb. Optim.}, 50\penalty0 (3), 2025{\natexlab{a}}.
\newblock \doi{10.1007/s10878-025-01348-6}.

\bibitem[Chen et~al.(2025{\natexlab{b}})Chen, Jing, and Zang]{ggw25v2}
G.~Chen, G.~Jing, and W.~Zang.
\newblock Publisher {C}orrection: {P}roof of the {G}oldberg–{S}eymour
  conjecture on edge–colorings of multigraphs.
\newblock \emph{J. Comb. Optim.}, 51\penalty0 (1), 2025{\natexlab{b}}.
\newblock \doi{10.1007/s10878-025-01372-6}.

\bibitem[Edmonds et~al.(1982)Edmonds, Lov\'asz, and Pulleyblank]{elp82}
J.~Edmonds, L.~Lov\'asz, and William~R. Pulleyblank.
\newblock Brick {D}ecomposition and the {M}atching {R}ank of {G}raphs.
\newblock \emph{Combinatorica}, 2:\penalty0 247--274, 1982.
\newblock \doi{10.1007/BF02579233}.

\bibitem[Esperet et~al.(2010)Esperet, Kr\'al, \v{S}koda, and
  \v{S}krekovski]{ekss10}
L.~Esperet, D.~Kr\'al, P.~\v{S}koda, and R.~\v{S}krekovski.
\newblock An improved linear bound on the number of perfect matchings in cubic
  graphs.
\newblock \emph{European Journal of Combinatorics}, 31\penalty0 (5):\penalty0
  1316--1334, 2010.
\newblock \doi{10.1016/j.ejc.2009.11.008}.

\bibitem[Fowler(1998)]{f98}
Thomas~George Fowler.
\newblock \emph{Unique Coloring of Planar Graphs}.
\newblock PhD thesis, Georgia Institute of Technology, 1998.
\newblock Available at \url{https://thomas.math.gatech.edu/FC/fowlerphd.pdf}.

\bibitem[Goedgebeur et~al.(2026)Goedgebeur, Mattiolo, Mazzuoccolo, Renders, and
  Wolf]{gmmrw24}
J.~Goedgebeur, D.~Mattiolo, G.~Mazzuoccolo, J.~Renders, and I.~H. Wolf.
\newblock Cubic graphs with edges in exactly one perfect matching.
\newblock \emph{J.\ Graph Theory}, 112\penalty0 (3):\penalty0 276--289, 2026.
\newblock \doi{10.1002/jgt.70024}.

\bibitem[Holyer(1981)]{holyer81}
I.~Holyer.
\newblock The {N}{P}-{C}ompleteness of {E}dge-{C}oloring.
\newblock \emph{SIAM Journal on Computing}, 10\penalty0 (4):\penalty0 718--720,
  1981.
\newblock \doi{10.1137/0210055}.

\bibitem[Jing(2026)]{gj26}
G.~Jing.
\newblock On {E}dge {C}oloring of {M}ultigraphs.
\newblock Available at \url{https://arxiv.org/abs/2308.15588}, 2026.

\bibitem[Kothari and Carvalho(2020)]{koca17}
N.~Kothari and M.~H. Carvalho.
\newblock Generating {S}imple {N}ear-{B}ipartite {B}ricks.
\newblock \emph{J.~Graph Theory}, 95:\penalty0 594--637, 2020.

\bibitem[Kothari et~al.(2020)Kothari, Carvalho, Lucchesi, and Little]{kcll20}
N.~Kothari, M.~H. Carvalho, C.~L. Lucchesi, and C.~H.~C. Little.
\newblock On essentially $4$-edge-connected cubic bricks.
\newblock \emph{The Electronic J.~of Combin.}, 27, 2020.
\newblock \doi{10.37236/8594}.

\bibitem[Kothari(2019)]{koth17}
Nishad Kothari.
\newblock Generating {N}ear-{B}ipartite {B}ricks.
\newblock \emph{J.~Graph Theory}, 90\penalty0 (4):\penalty0 565--590, 2019.
\newblock \doi{10.1002/jgt.22414}.

\bibitem[Kr\'al et~al.(2009)Kr\'al, Sereni, and Stiebitz]{kss09}
D.~Kr\'al, J.~Sereni, and M.~Stiebitz.
\newblock A new lower bound on the number of perfect matchings in cubic graphs.
\newblock \emph{\textsc{siam} J. Discrete Math.}, 23\penalty0 (3):\penalty0
  1465--1483, 2009.
\newblock \doi{10.1137/080723843}.

\bibitem[Lov\'asz(1987)]{lova87}
L.~Lov\'asz.
\newblock Matching {S}tructure and the {M}atching {L}attice.
\newblock \emph{J.~Combin.~Theory Ser.~B}, 43:\penalty0 187--222, 1987.
\newblock \doi{10.1016/0095-8956(87)90021-9}.

\bibitem[Lov\'asz and Plummer(1986)]{lopl86}
L.~Lov\'asz and M.~D. Plummer.
\newblock \emph{Matching Theory}.
\newblock Number~29 in Annals of Discrete Mathematics. Elsevier Science, 1986.
\newblock Available at
  \url{https://shop.elsevier.com/books/matching-theory/plummer/978-0-444-87916-5}.

\bibitem[Lu et~al.(2020)Lu, Kothari, Feng, and Zhang]{lkfz20}
F.~Lu, N.~Kothari, X.~Feng, and L.~Zhang.
\newblock Equivalence classes in matching covered graphs.
\newblock \emph{Discrete Mathematics}, 343\penalty0 (8), 2020.
\newblock \doi{10.1016/j.disc.2020.111945}.

\bibitem[Lucchesi and Murty(2024)]{lumu24}
C.~L. Lucchesi and U.~S.~R. Murty.
\newblock \emph{Perfect Matchings: A Theory of Matching Covered Graphs}.
\newblock Springer Nature Switzerland, 2024.
\newblock \doi{10.1007/978-3-031-47504-7}.

\bibitem[Narayana et~al.(2024)Narayana, Kothari, and Gohokar]{dmkg24}
D.~V.~V. Narayana, N.~Kothari, and K.~Gohokar.
\newblock Equivalence classes, solitary patterns and cubic graphs.
\newblock Available at \url{https://arxiv.org/abs/2409.00534}, 2024.

\bibitem[Norine and Thomas(2007)]{noth07}
S.~Norine and R.~Thomas.
\newblock Generating {B}ricks.
\newblock \emph{J.~Combin.~Theory Ser.~B}, 97:\penalty0 769--817, 2007.
\newblock \doi{10.1016/j.jctb.2007.01.002}.

\bibitem[Seymour(1979)]{s77}
P.~D. Seymour.
\newblock On {M}ulti-{C}olourings of {C}ubic {G}raphs, and {C}onjectures of
  {F}ulkerson and {T}utte.
\newblock \emph{Proceedings of the London Mathematical Society}, s3-38\penalty0
  (3):\penalty0 423--460, 1979.
\newblock \doi{10.1112/plms/s3-38.3.423}.

\bibitem[Seymour(1980)]{seym80}
P.~D. Seymour.
\newblock Decomposition of {R}egular {M}atroids.
\newblock \emph{J.~Combin.~Theory Ser.~B}, 28:\penalty0 305--359, 1980.
\newblock \doi{10.1016/0095-8956(80)90075-1}.

\bibitem[Szigeti(2002)]{szig02}
Zolt{\'a}n Szigeti.
\newblock Perfect {M}atchings versus {O}dd {C}uts.
\newblock \emph{Combinatorica}, 22:\penalty0 575--589, 2002.
\newblock \doi{10.1007/s00493-002-0008-6}.

\end{thebibliography}

\end{document}